\documentclass[12pt, reqno,toc=indent]{amsart}

\usepackage[margin=1.5in]{geometry}

\usepackage[utf8]{inputenc}
\usepackage[T1]{fontenc}
\usepackage[english]{babel}

\usepackage{mathrsfs}

\usepackage{graphicx}
\usepackage{tikz-cd}
\usepackage{tikz}
\usetikzlibrary{arrows.meta,positioning}
\usepackage{tabularx}
\usepackage{enumitem}
\usepackage{float}

\usepackage{wrapfig,lipsum,booktabs}
\usepackage{subcaption}

\setlist[enumerate]{
    itemsep=0.5em,
    topsep=0.5em,
    label=\rm(\roman*)
}

\usepackage{amsmath, amsthm, amssymb, amsfonts}
\usepackage{bm} 
\usepackage{nccmath}
\usepackage{physics}

\usepackage{xcolor}
\definecolor{antiquefuchsia}{rgb}{0.57, 0.36, 0.51}
\definecolor{azure}{rgb}{0.0, 0.5, 1.0}

\usepackage[colorlinks = true, linkcolor = azure, citecolor = antiquefuchsia]{hyperref}

\usepackage{comment}
\usepackage[disable]{todonotes}

\theoremstyle{plain}
\newtheorem{theorem}{Theorem}
\newtheorem*{theorem*}{Theorem}
\newtheorem*{metatheorem*}{Meta--Theorem}
\newtheorem{proposition}[theorem]{Proposition}
\newtheorem{corollary}[theorem]{Corollary}
\newtheorem{lemma}[theorem]{Lemma}

\newtheorem{introthm}{Theorem}

\theoremstyle{definition}
\newtheorem{definition}[theorem]{Definition}
\newtheorem{example}[theorem]{Example}

\newtheorem{remark}[theorem]{Remark}

\numberwithin{theorem}{section}
\numberwithin{equation}{section}

\newcommand{\Z}{\mathbb{Z}} 
\newcommand{\R}{\mathbb{R}} 
\newcommand{\C}{\mathbb{C}}
\newcommand{\D}{\mathscr{D}} 
 
\newcommand{\DD}{\mathbb D} 
\renewcommand{\P}{\mathbb{P}} 
\newcommand{\Id}{\mathrm{Id}}

\newcommand{\s}{\mathbf{s}}

\newcommand{\p}{\mathsf{p}}

\newcommand{\V}{\mathbf V}

\DeclareMathOperator{\Hom}{Hom}
\DeclareMathOperator{\Aut}{Aut}
\DeclareMathOperator{\rk}{rk}

\DeclareMathOperator{\spn}{span}

\renewcommand{\O}{\mathcal O}

\title{Zoll magnetic structures and ruled surfaces}
\date{\today}

\author[J. Bohr]{Jan Bohr}
\address{ 
Mathematical Institute of the University of Bonn,
Endenicher Allee 60, 53115 Bonn, Germany}
\email {bohr@math.uni-bonn.de}

\author[G.P. Paternain]{Gabriel P. Paternain}
\address{ Department of Mathematics, University of Washington, Seattle, WA 98195, USA}
\email {gpp24@uw.edu}

\begin{document}
\maketitle


\begin{abstract}
A {\it Zoll magnetic system} on an oriented closed surface $M$  is a Riemannian metric $g$ together with a function $\lambda\colon M\to \R$, such that every unit speed solution of the ODE  $\ddot \gamma(t)=\lambda(\gamma(t))\gamma(t)^\perp$ is periodic and the minimal period depends continuously on $\gamma$. The trivial example is given by $g$ with constant curvature $K$ and $\lambda\equiv {\rm const.}$ such that $\lambda^2+K>0$. This article  exhibits non-trivial Zoll magnetic systems for every genus ---  for genus $\ge 2$ these are the first such examples. 

The approach is twistor theoretic: To a general magnetic system $(g,\lambda)$ one associates its {\it transport twistor space} $Z(g,\lambda)$, which is the unit disk bundle $DM$, equipped with a degenerate complex structure that encodes the magnetic flow. For the  trivial Zoll magnetic systems explicit {\it holomorphic blow-down maps} $\beta\colon Z(g,\lambda)\to W$  into  certain ruled surfaces $W\to M$ are constructed, mapping $\partial Z(g,\lambda)$ onto a Lagrangian $P\subset W$. For small Lagrangian perturbations $P'\approx P$ the procedure can be reversed and this results in a large class of (non-trivial) nearby Zoll magnetic systems.
\end{abstract}


\section{Introduction}

Let $(M,g)$ be a closed oriented Riemannian surface, $SM$ its unit circle bundle and $\lambda\colon SM\to \R$ a smooth function. A unit speed curve $\gamma\colon \R\to M$ is called  {\it $\lambda$-geodesic}, if it satisfies the ordinary differential equation
\begin{equation}\label{ODE}
	\nabla_{\dot \gamma} \dot \gamma = \lambda(\gamma,\dot\gamma)\dot \gamma^\perp,
\end{equation}
where $\nabla$ is the Levi-Civita connection and $\perp$ denotes rotation by $+\pi/2$. The associated flow $(\phi_t)$ on $SM$ is called a {\it thermostat flow} and the tuple $(g,\lambda)$ is referred to as a {\it thermostat}.


\begin{definition}
	The thermostat $(g,\lambda)$ is called {\it Zoll}, if  the flow $(\phi_t)$ has the same orbits as a free $S^1$-action.
\end{definition}

If $\lambda\equiv 0$, we recover the notion of a Zoll metric. In this case, and more generally when the flow is {\it reversible} (equivalently, if $\lambda(x,-v)=-\lambda(x,v)$\todo{Prove/Leave out?}), the Zoll property enforces $M$ to be diffeomorphic to $S^2$. 

If $\lambda(x,v)=\lambda(x)$ is independent of the direction, the thermostat is also referred to as {\it magnetic system} and \eqref{ODE} may be viewed as model for the motion of a charged particle in a magnetic field, or more geometrically, 
as prescribed (geodesic-)curvature equation for $\gamma$. Moreover, $(\phi_t)$ is the Hamiltonian flow of $H(x,v)=|v|_g^2/2$, restricted to the energy level $H=1/2$ and with respect to the following  twisted symplectic structure on $TM$:
\[
	\omega_\lambda = \omega_0 + \pi^*(\lambda\, d\mathrm{vol}_g).
\]
\todo[inline]{G: check signs, $\omega_{0}=\pm d\alpha$?

J: Flat case with $\omega_0=-(pdq)=dq\wedge dp$ and $X=\dot q\partial_q + \dot p \partial_p$ gives
\begin{align*}
	i_X\omega_\lambda &= i_X(dq\wedge dp + \lambda dq_1\wedge dq_2) = \dot q dp - \dot p dq + \lambda \dot q_1 dq_2 - \lambda \dot q_2 dq_1
\\
&=^! H_p dp + H_q dq,
\end{align*}
such that
\[
	\begin{cases}
		\dot q = H_p\\
		\dot p =  - H_q  + \lambda \dot q^\perp
	\end{cases}
\]
}Here $\omega_0$ is the canonical symplectic structure with respect to $g$ and $\pi\colon TM\to M$ is the foot-point projection. The interpretation as constant curvature equation provides an immediate example of a Zoll magnetic system on every surface of constant curvature $K$ --- namely, if \begin{equation}\label{eqn_modelcase}
	K\equiv \mathrm{const.},\quad \lambda\equiv \mathrm{const.} \quad \text{ and } \quad K+\lambda^2>0,
\end{equation}
then the solutions of \eqref{ODE}, lifted to the universal cover, parametrise the boundaries of geodesic disks and are therefore periodic with a common period. In this article we show that there is an abundance of nearby Zoll magnetic systems, essentially one for each choice of a closed $1$-form on $M$ together with one additional real parameter. The quest for  Zoll magnetic systems was initiated by Asselle--Benedetti--Berti \cite{AsBe21,ABB24}, who constructed the first non-trivial examples on the torus---in the present article we give the first ones for genus $\ge 2$.

Let us briefly comment on the definition of the Zoll property:
If all the orbits of $(\phi_t)$ are closed, one can define the minimal period function \[
	T\colon SM\to (0,\infty),\quad T(x,v)=\inf\{t>0:\phi_t(x,v)=(x,v)\}.
\]
By a theorem of Epstein \cite{Eps72}, there is a smooth first integral $\tilde T\colon SM\to (0,\infty)$, such that $\tilde T/T$ takes values in the positive integers and is $\ge 2$ at most on a finite number of {\it exceptional orbits}. An equivalent definition of the Zoll property is thus that all orbits of $(\phi_t)$ are closed and $T$ is continuous, i.e.~there are no exceptional orbits. In many cases the latter is automatic:  If $M$ has genus $\ge 1$ exceptional orbits are precluded by the classification of Seifert fibred manifolds \cite[Theorem 5.1]{JaNe83}; if $M$ is diffeomorphic to $S^2$ and $\lambda\equiv 0$, they were ruled out by Gromoll--Grove \cite{GrGr81}, who also proved that the geodesics of a Zoll metric are necessarily simple and have a common minimal period (i.e.~$T\equiv \mathrm{const.}$).
For $\lambda \not \equiv 0$, exceptional orbits can arise for variants of the classical {\it Katok example} on $S^2$ \cite{Kat74,Ben16}. Note that we neither insist on simplicity of $\lambda$-geodesics, nor on the existence of a common minimal period.

The first non-trivial examples of  Zoll metrics on $S^2$ (i.e.~$\lambda\equiv 0$) were surfaces of revolution constructed by Otto Zoll \cite{Zol03}---see also \cite{GuMa26} for a recent exposition of this construction and a discussion in the context of geometric rigidity questions (such as in \cite{MaSu18}). 
The construction of Zoll metrics without rotational symmetry is a more difficult matter. Funk \cite{Fun13} had predicted the correct functional freedom, but the first rigorous construction goes back to Guillemin \cite{Gui76}: Given a smooth odd function $h\colon S^2\to \R$, he used the Nash--Moser inverse function theorem to construct a $1$-parameter family of Zoll metrics $(g_t)$ on $S^2$ such that
\[
g_t =\exp(t h + O(t^2)) g_{\mathrm{can}},\quad t\to 0.
\] 
A similar technique was employed more recently by Ambrozio--Marques--Neves \cite{AMN25} to construct exotic metrics on $S^{d}$ ($d\ge 2$) admitting so-called Zoll families of minimal hypersurfaces and by Asselle--Benedetti--Berti \cite{ABB24} to construct Zoll magnetic systems on the $2$-torus.


The twistor theoretic approach to construct Zoll metrics (and more generally, projective structures) on $S^2$ was introduced by LeBrun--Mason in \cite{LeMa02}.
They set up a $1:1$ correspondence between Zoll metrics on $S^2$ near the round one and small perturbations of  $\R P^2$ within $\C P^2$ that satisfy a natural Lagrangian condition. This allows to recover the functional freedom (namely, of an odd function on $S^2$) near the round metric, and has further been extended to construct Zoll metrics  far away from the round one \cite{LeMa10}. Besides the conceptual appeal of this approach, there is also a striking analytical advantage: The twistor correspondence turns the direct perturbation problem (which involves a loss of derivatives necessitating the use of Nash--Moser's inverse function theorem) into a more approachable existence question for holomorphic disks---further instances of this phenomenon are discussed e.g.~in \cite{Don02,LeB06,LeMa07}.

Our construction of Zoll magnetic systems parallels LeBrun--Mason's approach, but becomes considerably more involved due to the presence of a magnetic field: To deal with irreversible flows we introduce a generalisation of the {\it transport twistor spaces} that we first considered in \cite{BoPa23}. The role of $\C P^2$ is now played by a certain ruled surface $W$ over $M$---a fact that was already foreshadowed by the appearence of the branched double cover $\C P^1\times \C P^1\to \C P^2$  in \cite{LeMa10}. We are then led to consider holomorphic disks in $W$ and since these now reside in a $3$-dimensional moduli space, fixing a base point leaves one degree of freedom. This turns out to be an essential feature of the magnetic setting, which can already be observed in the constant curvature model \eqref{eqn_modelcase}, where it is parametrised by $\lambda\in \R$. A key observation of the present article is that this additional degree of freedom corresponds to choosing a leaf in a natural codimension-$1$-foliation of the moduli space. Our construction of exotic Zoll magnetic systems then rests on the persistence of closed leafs under small perturbations of the boundary condition of the holomorphic disks.


\subsection{Existence theorem}\label{sec_thma} 
Before setting up the twistor theory proper, we formulate a theorem that captures the correct functional freedom. We fix a closed oriented Riemannian surface $(M,g)$
and call $\lambda\in C^\infty(SM,\R)$ {\it conformally magnetic}, if it has the form
\begin{equation}\label{confmag}
	\lambda(x,v) = \lambda_0(x) - *d\sigma_x(v),
\end{equation}
where $\lambda_0,\sigma\colon M\to \R$ are smooth functions and $*$ is the Hodge star of $g$. The associated flow is orbit equivalent, via the scaling map $\mathrm{sc}_\sigma(x,w)=(x,e^{\sigma(x)}w)$ to the magnetic system $(e^{2\sigma}g,e^{-\sigma}\lambda_0)$. That is, we can capture all Zoll magnetic systems in the conformal class of $g$ with the
 following moduli space:
\[
	\mathbf{\Lambda}_\mathrm{mag}(g):=\{\lambda \in C^\infty(SM,\R) \text{ conformally magnetic:} ~ (g,\lambda) \text{ Zoll} \}/\Aut(M)\times \Z_2.
\]
Here $\Aut(M)$ is the automorphism group of the underlying Riemann surface, that is,  an element $\varphi\in \Aut(M)$ is orientation preserving and satisfies $\varphi^*g=e^{2u}g$ for some $u\in C^\infty(M,\R)$---it acts on $\lambda\in C^\infty(SM,\R)$ by
\begin{equation}\label{defaction}
	(\lambda \triangleleft \varphi)(x,v) := e^{u(x)} \lambda(x,e^{-u(x)}d\varphi_x(v)) - *du_x(v),\quad (x,v)\in SM.
\end{equation}
This is the natural action coming from rescaling the flow of the magnetic system $(e^{2\sigma}g,e^{-\sigma}\lambda_0)$, where $\varphi\in \Aut(M)$ acts by pull-back:
\[
	\varphi^*(e^{2\sigma}g,e^{-\sigma}\lambda_0) = (e^{2(\varphi^*\sigma +u)} g, e^{-(\varphi^*\sigma + u)}  e^{u}\varphi^*\lambda_0)
\]  
Finally, the group $\Z_2$ acts on $C^
\infty(SM,\R)$ by sending $\lambda$ to $-\lambda\circ \mathsf{f}$, where $\mathsf{f}(x,v)=(x,-v)$---this  corresponds to flipping the sign of the magnetic field.

 \begin{introthm}\label{introthmA}
	Let $(M,g)$ be a closed and oriented Riemannian surface of constant curvature $K$ and $\lambda_0\equiv \mathrm{const.}$ with $K+\lambda_0^2>0$. Then there is a neighbourhood $\mathcal U\subset \left(\Omega^1(M,\R)\cap \ker d\right) \times \R$ of $(0,0)$  in the $C^\infty$-topology and a map
  	\[
  		 \mathcal U\to \mathbf{\Lambda}_\mathrm{mag}(g)
  	\]
	that sends $(0,0)$ to $[\lambda_0]$  	
  	and such that the preimage of $\{[\lambda]: \lambda\equiv \mathrm{const.}, K+\lambda^2>0\}$ is contained in a finite dimensional submanifold of $\mathcal U$. 
	\\[.2em]
{\rm [Proof contained in Section \ref{sec_pfA}.]}
  \end{introthm}

That is, to a closed $1$-form
$\alpha\in \Omega^1(M,\R)$ and a parameter $\ell\in \R$, both assumed to be sufficiently small, 
 we can associated a class of Zoll conformally magnetic systems in $\mathbf{\Lambda}_\mathrm{mag}(g)$ and generically this class is non-trivial.
  In our twistor correspondence, the closed $1$-forms $\alpha$ correspond to La\-gran\-gians in a ruled surface $W$ and the trivial deformations  arise from Lagrangians on a fixed  orbit of the action by the automorphism group $\operatorname{Aut}(W)$, which is finite dimensional. Of course, over $S^2$ there also reversible Zoll  deformations and these are picked out by a suitable symmetry condition on the Lagrangians in question---we will discuss this in Section \ref{sec_genuszero}.

We expect that the image of
$\mathcal U\to \mathbf{\Lambda}_{\mathrm{mag}}(g)$ contains a full neighbourhood of $[\lambda_0]$. This would follow from reversing our twistor correspondence for nearby $[\lambda]\in \mathbf{\Lambda}_\mathrm{mag}(g)$ and a {\it fixed} ruled surface $W$. We will come back to this problem in future work.

 \subsection{Holomorphic blow-down maps} 
Let $(M,g)$ be a closed oriented  Riemannian surface. This is naturally a Riemann surface with respect to the following complex structure on the tangent spaces $TM$:
\[
	(x,\omega v) := (x,\Re(\omega) v +  \Im(\omega) v^\perp),\quad (x,v)\in TM,\omega \in \C.
\]
The unit disk bundle $DM$ is a compact $4$-manifold with boundary $\partial DM=SM$ and interior $DM^\circ=DM\backslash SM$.
Our strategy is based on considering {fibrewise holomorphic blow-down maps}
 \[
	\beta\colon DM\to W 
 \]
 into a complex surface. {\it Fibrewise holomorphic} means that the restriction to each fibre $D_x M$ is holomorphic, or equivalently, $\omega \mapsto  \beta(x,\omega v)$ is a holomorphic map on the the disk $\DD=\{\omega\in \C:|\omega|\le 1\}$ for all $(x,v)\in SM$. 
 Further, we call $\beta$ a {\it blow-down map}, if the image $P=\beta(SM)$ is an embedded surface and there is a factorisation as follows:
 \[
\begin{tikzcd}
	DM \arrow["\sim"]{r} \arrow[swap, "\beta"]{dr}&  {[W,P]} \arrow{d}\\
	& W
\end{tikzcd}	
\]
Here $[W,P]\to W$ is the blow-up of $W$ along $P$ in the sense of Melrose and we ask the top arrow to be a diffeomorphism --- see Section \ref{sec_blow} for full definitions. 

Let now $\lambda\colon SM\to \R$ be a smooth function.
Then the $\lambda$-geodesic flow  is generated by the following vector field  on $SM$:
\[
	X+\lambda V.
\]
Here $X$ is the geodesic vector field, that is, the generator of the flow for $\lambda=0$ and the {vertical vector field} $V$ is defined as generator of the fibrewise rotations $(x,v)\mapsto (x,e^{it}v)$ on $SM$. Any blow-down map $\beta\colon DM\to W$ restricts to an $S^1$-bundle $SM\to P$. Hence the condition
\[
	d \beta(X+\lambda V)=0\text{ on } SM,
\]
which means that $\beta|_{SM}$ is a first integral of the flow, immediately implies that  $(g,\lambda)$ is a Zoll thermostat. Additionally imposing fibrewise holomorphicity allows to bound the {degree} of $\lambda$, which is defined as follows: We say that \[\deg(\lambda)\le m,\] iff $\lambda(x,\cdot)\colon S_xM \to \R$ extends to a polynomial of degree $\le m$ on $T_xM$ ($x\in M$). That is, magnetic systems correspond to $\deg(\lambda)\le 0$ and conformally magnetic ones to $\deg(\lambda)\le 1$, subject to an exactness condition on the linear portion.

To formulate a precise result to this effect, we require a few more definitions: First,  $P=\beta(SM)$ is called {\it totally real}, iff
\[
	T_p P \cap J(T_pP)=0,\quad p\in P,
\]
where $J$ is the complex structure tensor of $W$. We can then define
\[
	\mu(\beta):= \text{Maslov index of the holomorphic disk $\beta(D_xM)$ in $(W,P)$},
\]
which is an integer that is independent from $x\in M$ (cf.~Section \ref{sec_chernmaslov}). Given a smooth function $\tau\colon M\to \R$, we define
\[
	\beta_\tau\colon DM\to W,\quad \beta_\tau(x,v)=\beta(x,e^{i\tau(x)}v),
\]
which is again a fibrewise holomorphic blow-down map -- we are merely rotating every fibre separately.

\begin{introthm}\label{introthmB}
Let $(M,g)$ be a closed oriented Riemannian surface, $W$ a complex surface and $\beta\colon DM\to W$ a fibrewise holomorphic blow-down map. Then the following are equivalent:
\begin{enumerate}
	\item\label{introthmBi} There exists $\lambda\colon SM\to \R$ with  $\deg(\lambda)\le 1$ and a rotation $\tau\colon M\to \R$ such that $d\beta_\tau(X+\lambda V)=0$ on $SM$.
	\item\label{introthmBii} The following properties are satisfied:
	\begin{enumerate}[label=\rm (\alph*)]
		\item\label{introthmBiia} $P=\beta(SM)$ is totally real and $\mu(\beta)=4$;
		\item\label{introthmBiib} $q:=\beta(\cdot,0)\colon M\to W$ is holomorphic and, writing $Q=q(M)$,  the complex lines $d\beta_{(x,0)}(T_0D_xM)$ ($x\in M$) form a holomorphic subbundle of $TW|_Q$.
	\end{enumerate}
\end{enumerate}
The function $\lambda$ is unique up to the involution $\lambda\mapsto -\lambda\circ \mathsf{f}$.\\[.2em]
{\rm [Proof contained in Section \ref{sec_tts}.]}
\end{introthm}

Of course, one still needs to determine when the thermostats produced by the theorem are actually conformally magnetic (rather than just having degree $\le 1$), but we will postpone this discussion for the moment.

Fibrewise holomorphic first integrals arise frequently in geometric inverse problems \cite{PSU23} and have an appealing interpretation in terms of  so-called {transport twistor spaces}---let us explain these:  An {\it involutive structure} on $DM$ is a complex subbundle $\D\subset T_\C DM = TDM\otimes \C$, which is closed under taking Lie brackets. Suppose that $\lambda\colon SM\to \R$ is smooth and there exists an involutive structure $\D$ of rank $2$ that is orientation compatible and satisfies
\begin{equation}\label{definvol}
	T^{0,1}(D_x M) \subset \D \quad (x\in M)\quad \text{ and } \quad \D\cap \overline{ \D} =\begin{cases} \C(X+\lambda V)& SM\\
	0 & DM^\circ.
	\end{cases}
\end{equation}
Then the tuple $(DM,\D)$ is called the transport twistor space of $(g,\lambda)$.  Since $\D$ does not contain any real vectors over $DM^\circ$, it must be the $(0,1)$-bundle of a complex structure and with respect to this we understand the orientation compatibility. For $\lambda=0$ we constructed transport twistor space in \cite{BoPa23}, but the construction extends seamlessly to $\deg(\lambda)\le 2$. The involutive structure $\D=\D_\lambda$ is then unique and we write the resulting transport twistor space as \[Z(g,\lambda)=(DM,\D_\lambda).\]
We may think of this as a complex surface whose complex structure degenerates at the boundary $\partial DM$. 
A smooth map
\[
	\beta\colon Z(g,\lambda)\to W
\]
is {\it holomorphic} (that is, $\beta_*\D\subset T^{0,1}W$) if and only if it is fibrewise holomorphic and restricts to a first integral on $\partial Z(g,\lambda)=SM$. In this sense the blow-down map $\beta_\tau$ in Theorem \ref{introthmB} can be thought of as a genuinely  holomorphic map.

The action of the group $\Aut(M)\times \Z_2$ defined below \eqref{defaction} has the following interpretation: Namely,  two functions $\lambda_1,\lambda_2\in C^\infty(SM,\R)$ with $\deg(\lambda_k)\le 2$ lie on the same orbit of the group action if and only if there is a biholomorphism $\Phi\colon Z(g,\lambda_{1})\xrightarrow{\sim} Z(g,\lambda_2)$ that sends fibres into fibres (see Proposition \ref{prop_biholo}).

We can now sketch the proof of Theorem \ref{introthmB}: Pulling back the $(0,1)$-bundle of $W$ via $\beta$ produces an involutive structure $\D$ on $DM$ and we may ask when this comes from a transport twistor space. For this we give a precise characterisation (see Propositon \ref{prop_charZ}) that may be of independent interest.

\subsection{Blow-downs in the model cases}  In order to apply Theorem \ref{introthmB}, we need to produce fibrewise holomorphic blow-down maps. To this end we first find explicit holomorphic blow-down maps in the model cases ($g$ of constant curvature and $\lambda\equiv \mathrm{const.}$).

\begin{introthm}\label{introthmC}
Let $(M,g)$ be a closed oriented surface of constant curvature $K$. Suppose $\lambda\equiv \mathrm{const.}$ with $K+\lambda^2>0$.
Then there exists a holomorphic blow-down map
\[
	\beta\colon Z(g,\lambda)\to W
\]
into a complex surface $W$  such that $Q=\beta(M)$ is complex and $P=\beta(SM)$ is totally real--- 
see Tables \ref{table1} and \ref{table2}. Moreover, there is a ruling $\rho\colon W\to M$ such that $\rho\circ \beta(x,0)=x$ for all $x\in M$.\\[.2em]
{\rm [Proof contained in Section \ref{sec_model}.]}
\end{introthm}

A \textit{ruling} is a holomorphic map $\rho\colon W\to M$, all of whose fibres are diffeomorphic\footnote{By the Fischer--Grauert theorem \cite[Theorem 6.2.4]{Huy05}, $\rho$ is then automatically a holomorphic $\C P^1$-bundle. As such, every ruled surface over $M$ is isomorphic the the projectivisation of a rank $2$ holomorphic vector bundle $E\to M$.}  to $\C P^1$. In the given context it has the following geometric meaning: if $p\in P$, then $\beta^{-1}(\{p\})$ is an orbit of the magnetic flow, that is, $\pi(\beta^{-1}(p))\subset M$ lifts to the universal cover of $M$ as boundary of a geodesic disk---the point $\rho(p)\in M$ is the projection of the centre of this disk.

All  of the objects in the theorem are constructed explicitly in Section \ref{sec_model}. The starting point in each case is a judicious ansatz for $\beta$ in isothermal coordinates; the appropriate surface $W$ then appears naturally if one insists on the correct invariance properties and turns out to be a certain ruled surface. In all cases it is diffeomorphic to $M\times S^2$; however if $M$ has genus $\ge 1$ it is {\it not} biholomorphic to $M\times \C P^1$.  In fact, if $M$ has genus $1$, one can identify $W$ with {\it Atiyah's ruled surface} over $M$ (see Remark \ref{rmk_atiyah}).

We note that the totally real surface $P\subset W$ does not depend on the magnetic field strength $\lambda$. From this observation we deduce the following corollary:

\begin{corollary}
	Let $(M,g)$ be a closed oriented Riemannian surface  of constant curvature $K$ and suppose that $\lambda_1,\lambda_2\in \R$ both satisfy $K+\lambda_k^2 >0$ ($k=1,2$). Then there exists a biholomorphism:
	\[
		\Phi\colon Z(g,\lambda_1)\xrightarrow{\sim} Z(g,\lambda_2)
	\]
\end{corollary}

\begin{proof}
Let $\beta_1,\beta_2$ be the maps from Theorem \ref{introthmC}, applied to $\lambda_1$ and $\lambda_2$, respectively. Since $\beta_1(SM)=P=\beta_2(SM)$,  the composition $\Phi=\beta_2^{-1}\circ \beta_1\colon Z(g,\lambda_1)^\circ \to Z(g,\lambda_2)^\circ$ defines a biholomorphism between the interiors. Since both $\beta_1$ and $\beta_2$ are blow-down maps, $\Phi$ extends to a diffeomorphism up to the boundary.
\end{proof}

This is in contrast to the {\it biholomorphism rigidity} observed in \cite{BMP25}. One can identify the restriction $\Phi|_{SM}$  as the flow (at a fixed time) of the horizontal vector field $H=[V,X]$; a map of this type has already appeared as exotic  biholomorphism  over $\R^2$ in the just cited article.

\begin{table}
\renewcommand{\arraystretch}{2}
\centering
\begin{tabular}{c|c|c}
Genus &  $M$ &$W$ \\
\hline
$0$ & $\C \cup \{\infty\}$ & $\C P^1\times \C P^1$ \\
$1$ & $\C/\Lambda$ & $(\C \times \C P^1)/\Lambda$ \\
$\ge 2$ & $\DD^\circ/\Gamma$ & $(\DD^\circ\times \C P^1)/\Gamma$ \\
\end{tabular} 

\vspace{1em}
\caption{Here $\Lambda\le \C$ is a lattice, acting on $\C\times \C P^1$ by $(w,\xi)\mapsto (w+\omega,\xi+\bar \omega)$ ($\omega\in \Lambda$) and $\Gamma\le \mathrm{Aut}(\DD)$ a cocompact Fuchsian group, acting on $\DD^\circ\times \C P^1$ by $(w,\xi)\mapsto (A(w),A(\xi))$ ($A\in \Gamma$). The rulings  are induced by the projection onto the first factor.}\label{table1}

\renewcommand{\arraystretch}{2.3}
\begin{tabular}{c|c|c|c}
Genus & Q& $P$ &$\Upsilon$\\
\hline
$0$ & $w_1=w_2$ &
$w_1\bar w_2 +1 =0$ &$\displaystyle \frac{1}{2i}\frac{dw_1\wedge dw_2}{(w_1-w_2)^{2}}$ \\
$1$ & $\xi=\infty$ &  $w = \bar \xi$ 
& $\displaystyle \frac{1}{2i}{dw\wedge d\xi}$\\
$\ge 2$ &  $w=\xi$ &  $w\bar \xi - 1 =0$ & $\displaystyle \frac{1}{2i}\frac{dw \wedge d\xi}{(w-\xi)^{2}}$\\[.1em]
\end{tabular} 
\vspace{1em}
\caption{Defining equations for $Q$ and $P$  and the meromorphic $2$-form $\Upsilon$.  Everything is displayed in affine coordinates on universal cover, where $\xi$ is the fibre variable for genus $\ge 1$.}\label{table2}
\end{table}

 \subsection{Perturbations and blow-ups}\label{sec_introperturb} Given a complex surface $W$, our goal is to find conditions on a  totally real surface $P\subset W$ that ensure  the existence  of a fibrewise holomorphic blow-down map $\beta\colon DM\to W$ with $\beta(SM)=P$. 
 Guided by the examples above we propose the following setting:  Let
\[
	\rho\colon W\to M
\]
be a {\it ruled surface} and let $q\colon M\to W$ be a holomorphic section with image $Q=q(M)$. Assume that $-2Q$ is a {\it canonical divisor}: equivalently, there exists a meromorphic $2$-form $\Upsilon$ on $W$ whose divisor equals
\begin{equation}\label{divisorcondition}
	\operatorname{div}(\Upsilon)=-2Q.
\end{equation}
This means that $\Upsilon$ is holomorphic and nowhere vanishing on $W\backslash Q$ and that it has a double pole along $Q$. Such a $2$-form is unique up to multiplication by a nonzero complex number and we assume from now on that it has been fixed. We also assume that there is no residue along $Q$:
\[
 \operatorname{Res}_Q(\Upsilon)=0	
\]
We remark in passing that in general there is no invariant notion of residue for poles of order $\ge 2$, but in the present setting  it can be defined in terms of the ruling (cf.~Section \ref{sec_residue}). The tuples $(W,Q,\Upsilon)$ displayed in Tables \ref{table1} and \ref{table2} above are easily checked to satisfy all of the preceding requirements.

Given such a meromorphic $2$-form $\Upsilon$, the open manifold $W\backslash Q$ comes with a natural symplectic structure:

\begin{lemma}\label{lemimups}
	The imaginary part $\Im \Upsilon$ is a symplectic form on $W\backslash Q$.
\end{lemma}

\begin{proof}
	On $W\backslash Q$, the form $\Upsilon$ is of type $(2,0)$ and satisfies $\bar \partial \Upsilon =0$, hence $d\Upsilon$ is of type $(3,0)$ and thus vanishes for dimension reasons. Hence also $\Im \Upsilon = (\Upsilon - \bar \Upsilon)/2i$ is closed on $W\backslash Q$. Moreover, $\Im \Upsilon \wedge \Im \Upsilon =  (\Upsilon\wedge \bar \Upsilon)/2$ and thus its vanishing locus coincides with that of $\Upsilon$. But $\Upsilon$ is nowhere vanishing on $W\backslash Q$ and therefore $\Im \Upsilon$ is non-degenerate.
\end{proof}

This symplectic structure allows to identify conformal magnetic systems amongst thermostats with $\deg(\lambda)\le 1$ (see Section \ref{sec_pfB} for a proof):

\begin{lemma}\label{lem_confmag} Let $\lambda\colon SM\to \R$ be smooth with $\deg(\lambda)\le 1$ and let \(\beta\colon Z(g,\lambda)\to W\) be a holomorphic blow-down map into a complex surface with $\beta(M)=Q$. Then the following are equivalent:
\begin{enumerate}
	\item $\lambda$ is conformally magnetic;
	\item there exists  $c\in \C\backslash \{0\}$ such that $P=\beta(SM)$ is Lagrangian for $\Im(c\Upsilon)$.
\end{enumerate}
\end{lemma}

Let $[F]\in H_2(W)$ be the fibre class of the ruling, that is, the one induced by an arbitrary fibre $F=\rho^{-1}(\{x_0\})$. We consider the following class of surfaces, where we write `$\cdot$' for the intersection pairing on $H_2(W)$:

\begin{definition}[Adapted Surface]\label{def_adapted}
 	An embedding $p\colon M\hookrightarrow W$ will be called \textit{adapted} (to the ruled surface and the distinguished section), if its image $P=p(M)$ satisfies the following properties:
 	\begin{enumerate}[label=(\rm \alph*)]
 		\item $P\cap Q = \emptyset$;
 		\item $P$ is totally real;
 		\item $[P]\cdot [F] = 1$.
\end{enumerate} 	
In this situation we will simply speak of $P\subset W$ as an {\it adapted surface}.
\end{definition}

Again, the surfaces in Table \ref{table2} fall into this category (for the intersection property, one realises $P$ as image of a {\it section} $p\colon M\to W$). Moreover, with the normalisation in the table,
\[
	P \text{ is Lagrangian for } \Im \Upsilon.
\]
Note that on a K{\"a}hler manifold every Lagrangian is totally real---in our setting the two conditions are not related since $\Im \Upsilon$ is not of type $(1,1)$. 
We define
\begin{equation}\label{defpp}
	\mathbf{P}:=\{P\subset W: \text{adapted surface that is Lagrangian for} \Im \Upsilon\}
\end{equation}
and equip this with the $C^\infty$-topology.

\begin{introthm}\label{introthmD} The following properties describe an open subset  $\mathbf{P}_*\subset \mathbf P$:  there exists a fibrewise holomorphic blow-down map $\beta\colon DM\to W$ such that
\begin{enumerate}
	\item  $\beta^{-1}(P)=SM$; 
	\item $\beta(x,0)=q(x)$ and the complex lines $d\beta(T_0D_xM)\subset T_{q(x)}W$ ($x\in M$) form a holomorphic subbundle  $L\subset TW|_Q$. 
\end{enumerate}
Moreover, for a fixed $P\in \mathbf{P}_*$ the set of holomorphic line subbundles $L\subset TW|_Q$ that can be realised this way locally forms a $1$-dimensional manifold.
\\[.2em]
{\rm [Proof contained in Section \ref{sec_diskfam}.]}
\end{introthm}

\begin{figure}[h!]
\begin{center}
\begin{tikzpicture}[>=Stealth,thick]

\begin{scope}

    \draw[very thick] (0,2) -- (3,2);
     \draw[red, very thick] (0,0) -- (3,0);
      \draw[red, very thick] (0,4) -- (3,4);
    
    \draw[blue] (2.2,0) -- (2.2,4);
    \draw[dashed] (0,0) -- (0,4);
     \draw[dashed] (3,0) -- (3,4);
    
    \fill[blue] (2.2,4) circle (2pt);
    
    \fill[blue] (2.2,0) circle (2pt);
    
    \node[red] at (0.6,3.6) {$SM$};
    \node at (0.6,2.3) {$M$};
    \node[below] at (1.5,-.1) {$Z$};
    
    \draw[blue,->] (2.2,4.8)  -- (2.2,4.2);
    \node[blue] at (2.2,5.1) {$D_xM$};
\end{scope}

\draw[->] (4,2) to[bend left=30] (6,2);
\node at (5,1.9) {$\beta$};

\begin{scope}[shift={(7,0)}]

   
    \draw[dashed, black] (0,0) -- (0,4);
     \draw[dashed, black] (3,0) -- (3,4);
    
    \draw[black, very thick] (0,2) -- (3,2); 
    \node at (0.5,2.3) {$Q$}; 
    
    \draw[dashed] (0,4) -- (3,4);
    \draw[dashed] (0,0) -- (3,0);
    
    \draw[red, very thick] (0,3) .. controls (1.2,3.2) and (1.8,2.5) .. (3,2.7);
    \node[red] at (2,3) {$P$};
    
    \draw[blue] (1.2,0) .. controls (0.8,1.2) and (1.3,2.0) .. (1.5,2.85);
    \fill[blue] (1.5,2.85) circle (2pt);
    \fill[blue] (1,2.96) circle (2pt);
    \draw[blue] (1.2,4) .. controls (1.3,3.6) .. (1,2.96);
     \draw[blue, ->] (1.2,4.8) -- (1.2,4.2);
    \node[blue] at (1.2,5) {$\Delta_q$};
    
    \draw[->] (2.5,4.8) -- (2.5,4.2);
    \node at (2.5,5.1) {$\C P^1$};
     \draw (2.5,0) -- (2.5,4);

    \node[below] at (1.5,-.1) {$W$};

\end{scope}

\end{tikzpicture}
\caption{Holomorphic blow-down maps and disk families}
\end{center}
\end{figure}
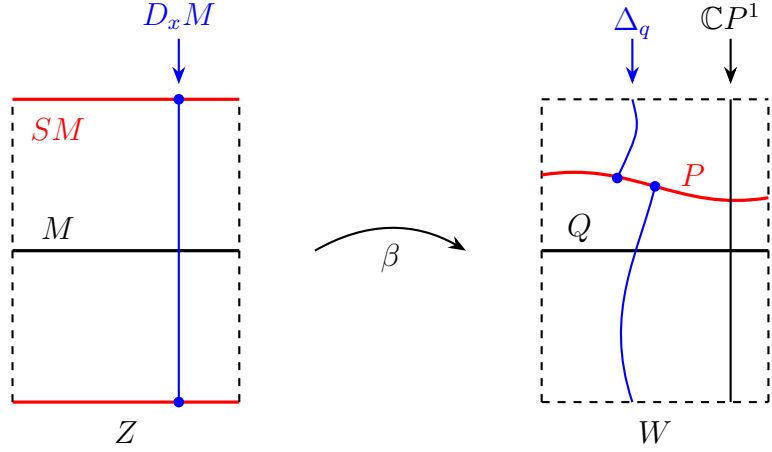

In order to prove the theorem, we consider families of holomorphic disks:

\begin{definition}[Disk families]\label{def_diskfamily} Given an adapted surface $P\subset W$, we define  $\mathcal F(P)$ as the set of all continuous families  $(\Delta_q:q\in Q)$, such that:
\begin{enumerate}[label=(\rm \alph*)]
	\item\label{diskfamilya} $(\Delta_q,\partial \Delta_q)\subset (W,P)$ is an {immersed} holomorphic disk with $\Delta_q\cap P = \partial \Delta_q$ and such that $\Delta_q\pitchfork Q=\{q\}$ is an {embedded point with a transversal intersection};
	\item\label{diskfamilyb} $L_q:=T_q\Delta_q\subset T_qW$  defines a holomorphic line subbundle of $TW|_Q$.
\end{enumerate}
\end{definition}

Any fibrewise holomorphic blow-down map $\beta\colon DM\to W$ as in Theorem \ref{introthmD} defines such a disk family by assigning to $q=q(x)\in Q$ the disk $\Delta_q=\beta(D_xM)$ {(see also Remark \ref{rk_immdisk}). 
 Note that the boundary $\partial \Delta_q$ might not be embedded: If $d\beta(X+\lambda V)=0$ and  there is a $\lambda$-geodesic with a self-intersection in $x\in M$, then the corresponding orbit intersects $S_xM$ in more than one point, and all these points are collapsed under the $\beta$-map. This already happens in the model cases  if $K + \lambda^2$ is small.}

Also the converse is true: any disk family $(\Delta_q:q\in Q)\in \mathcal F(P)$ can be realised by a such a blow-down map (cf.~Proposition \ref{prop_defbeta}), that is,
\[
	\mathbf{P}_*=\{P\in \mathbf{P}: \mathcal F(P)\neq \emptyset\}.
\]

We are thus led to consider the following moduli-space of (unparametrised) {immersed} holomorphic disks: 
\[
	\mathcal N(P):=\left\{(\Delta,\partial \Delta)\subset (W,P):  \begin{array}{l} \text{$\Delta$ {immersed holomorphic disk}, }\\
	 \Delta\cap P = \partial \Delta \text{ and }   |\Delta\pitchfork Q|=1
	 \end{array}  \right\}.
\] 
For an adapted surface $P$ these disks have total Maslov index $4$, such that $\mathcal N(P)$ becomes a $3$-dimensional manifold.  Moreover, the disk families in $\mathcal F(P)$ can be shown to correspond to the level sets of a natural map 
\[
	\varphi\colon \mathcal N(P)\to \mathcal{L}=\{L\subset TW|_Q \text{ holomorphic line subbunde}: L\cap TQ=0\}.
\]
This assigns to $\Delta$ the unique holomorphic line bundle $L\subset TW|_Q$ with $L_q=T_q \Delta$, where $\Delta\cap Q=\{q\}$. In the given setting, $\mathcal{L}\cong \C$ and $\varphi$ has constant rank $1$, such that the level sets are $2$-dimensional. 

The openness of $\mathbf{P}_*$ then corresponds to the persistence of closed level sets under small perturbations of $P$. Moreover, the set of line bundles in Theorem \ref{introthmD} is an open subset of $\varphi(\mathcal N(P))$, which is an immersed $1$-dimensional manifold.

\subsection{Proof of Theorem \ref{introthmA}}\label{sec_pfA} Combining Theorems \ref{introthmB},\ref{introthmC} and \ref{introthmD}, we can conclude Theorem \ref{introthmA}:
 We start with a closed oriented Riemannian surface $(M,g)$ of constant curvature and a constant $\lambda_0\in \R$ such that $K + \lambda_0^2>0$. By Theorem \ref{introthmC}, this comes with a tuple $(W,Q,P,\Upsilon)$  as in Table \ref{table1} and \ref{table2}. In reference to this data one defines the sets $\mathbf{P}$ as in \eqref{defpp} and $\mathbf{P}^*\subset\mathbf{P}$ as in Theorem \ref{introthmD}.
We then have $P\in \mathbf{P}_*$ due to Theorem \ref{introthmB}. 

Weinstein's tubular neighbourhood theorem yields a symplectomorphism between neighbourhood of $P\subset W\backslash Q$ with a neighbourhood of the zero section in $T^*P$, where all nearby Lagrangians are graphs of closed $1$-forms. This gives  a  parametrisation of all $\Im \Upsilon$-Lagrangians close to $P=P_0$:
\[
	\mathbf A=\{\alpha\in \Omega^1(M,\R):d\alpha=0, ||\alpha||_{C^k(M)}<\epsilon\} \to  \mathbf{P},\quad \alpha\mapsto P_\alpha.
\]If $k\in \mathbb N$ is sufficiently large and $\epsilon>0$ sufficiently small, this map indeed lands in $\mathbf{P}$, since adaptedness is an open condition in the $C^\infty$-topology. By Theorem \ref{introthmD} we may even assume that $P_\alpha\in \mathbf{P}_*$. Let $\ell\in \R$ ($|\ell|<\epsilon$) parametrise the additional degree of freedom we have in choosing the holomorphic line bundle $L\subset TW|_{Q}$. Given $\alpha\in \mathbf{A}$ and $\ell\in (-\epsilon,\epsilon)$, Theorems \ref{introthmB} and \ref{introthmD} combined yield a function
\[
	\lambda = \lambda_{\alpha,\ell} \in C^\infty(SM,\R),\quad \deg(\lambda)\le 1,
\]
such that the associated transport twistor space admits a holomorphic blow-down map \(\beta_{\alpha,\ell}\colon Z(g,\lambda)\to W\), mapping $SM$ onto $P_\alpha$. 
 As discussed before Theorem \ref{introthmB}, the thermostat $(g,\lambda)$ is automatically Zoll and since $P_\alpha$ is $\Im \Upsilon$-Lagrangian,  Lemma \ref{lem_confmag} guarantees that it is conformally magnetic. This gives a map as in Theorem \ref{introthmA}:
 \[
 	\mathcal U:= \mathbf{A}\times (-\epsilon,\epsilon)\to \mathbf{\Lambda}_{\mathrm{mag}}(g),\quad (\alpha,\ell)\mapsto [\lambda_{\alpha,\ell}].
 \]
It remains to verify that the thermostats produced this way are generically non-trivial. Suppose that $[\lambda]$ is trivial, then by Proposition \ref{prop_biholo}
there is a biholomorphism $Z(g,\lambda)\cong Z(g,\lambda_0')$, where $\lambda_0'\equiv \mathrm{const.}$ with $K+(\lambda_0')^2>0$. Combined with Theorem \ref{introthmC}, we get a holomorphic blow-down map $\beta\colon Z(g,\lambda)\to W$ that sends $SM$ to $P_0$ and the concatenation
\[
	\Phi=\beta_{\alpha,\ell}\circ \beta^{-1}\colon W\backslash P_0 \to W\backslash P_\alpha
\]
is a biholomorphism. Moreover, for every $w\in P_0$ the fibre $\beta^{-1}(\{w\})$ is an orbit of the thermostat flow, such that $\beta_{\alpha,\ell}(\beta^{-1}\{w\})$ contains a unique point  $\Phi(w)\in P_\alpha$. This gives a continuous extension $\Phi\colon W\to W$ and as the roles of the two blow-down maps may be reversed, we deduce that $\Phi$ extends to an automorphism of $W$ with $\Phi(P_0)=P_\alpha$. To conclude the proof, we need to show that this gives a finite dimensional constraint on $\alpha$.

Writing $\operatorname{Aut}(W)$ for the automorphism group of $W$, we claim that the following set is a finite dimensional submanifold of $\mathbf{A}$:
\[\mathbf{A}_0:=\{\alpha' \in \mathbf{A}: \exists \Phi\in \operatorname{Aut}(W): \Phi(P_0)=P_{\alpha'}\}.\]
Indeed, $\operatorname{Aut}(W)$ is a finite dimensional Lie group that acts smoothly on the Fr{\'e}chet manifold of smooth embedded totally real  surfaces inside $W$. Hence \[
	\{P'\in \mathbf{P}: \exists \Phi\in \operatorname{Aut}(W) \text{ with } \Phi(P_0)=P'\},
\]
being an open subset of the orbit of $P_0$, is a finite dimensional (immersed) submanifold of $\mathbf{P}$. Via the parametrisation $\alpha\mapsto P_\alpha$ this corresponds to $\mathbf{A}_0$. 
 \qed

\subsection{Genus zero revisited}\label{sec_genuszero} Let us consider the case of the round $2$-sphere
\[
	S^2 = \{x\in \R^3: |x|^2=1\}\subset \R^3
\]
in some more detail and compare the situation to \cite{LeMa02,LeMa10}. If there is a constant magnetic parameter $\lambda\in \R$, the blow-down maps
\[
	\beta_\lambda\colon DS^2=\{(x,v)\in S^2\times \R^3: x\cdot v =0,~ |v|\le 1\} \to \C P^1 \times \C P^1
\]
admit a clear geometric interpretation: Namely, the $\lambda$-geodesic through $(x,v)\in SS^2$ cuts the sphere into two caps, with centres $C_{\lambda,\pm}(x,v)\in S^2$. Consider the stereographic projection through the south pole $\s=(0,0,-1)$:
\[
	\pi_\s\colon S^2\to \C P^1,\quad 
	\pi_\s(x)=\begin{cases} [x_1+ix_2:1+x_3] & x\neq \s\\
	[1:0] & x=\s
	\end{cases}.
\]
Then $\beta_\lambda(x,v)=\left(\pi_\s\circ C_{\lambda,+}(x,v),\pi_\s \circ C_{\lambda,-}(x,v)\right)$ is a first integral with values in $\C P^1\times \C P^1$. It turns out that this has a fibrewise holomorphic extension to $DS^2$, which is  also a blow-down map (see Propositions \ref{prop_constructionpositive} and \ref{prop_sphereblow}). Clearly, $C_{\lambda,+}(x,v)=-C_{\lambda,-}(x,v)$ and since the stereographic projection turns the antipodal map $x\mapsto -x$ into the involution $w\mapsto -1/\bar w$, we have
\[
	  \beta_\lambda(SS^2)  = \bar \Delta := \{\bar w_1w_2+1=0\}\subset \C P^1\times \C P^1.
\]

In the absence of a magnetic field (i.e.~for $\lambda=0$),  the blow-down map $\beta_0$ becomes equivariant with respect to the following $\Z_2$-actions: the flip $(x,v)\mapsto (x,-v)$ on $DM$ and the swap $(w_1,w_2)\mapsto (w_2,w_1)$ on $\C P^1\times \C P^1$. Hence it descends to a map between the respective quotients:
\begin{equation}\label{cdiag}
\begin{tikzcd}
 	 DS^2 \arrow{d} \arrow["\displaystyle \beta_0"]{r}& \C P^1\times \C P^1 \arrow{d}{m} \\
 	DS^2/\Z_2 \arrow["\displaystyle \beta_{0,\mathbb{P}}"]{r}& \C P^2.	
 \end{tikzcd}	
\end{equation}
The right arrow equals $m(w_1,w_2)=[w_1+w_2:w_1w_2-1: i(w_1w_2+1)]$ in affine coordinates and is a branched double cover. Moreover,
\[
	\beta_{0,\P}(SS^2/\Z_2)=m(\bar \Delta)=\R P^2\subset \C P^2.
\]

LeBrun--Mason's work essentially takes place downstairs, although the right double cover is already employed in \cite{LeMa10}. One may view $DM/\Z_2$ as a twistor space in its own right: If $(M,g)$ is oriented and $\lambda\colon SM\to \R$ is odd ($\lambda(x,-v)=-\lambda(x,v)$) and has $\deg(\lambda)\le 3$, there is a natural involutive structure $\D_{\lambda,\P}\subset T_\C(DM/\Z_2)$ analogous to above
and we refer to the degenerate complex surface
\[
	Z_\P(g,\lambda)=(DM/\Z_2,\D_{\lambda,\P})
\]
as {\it projective twistor space}. It was shown in  \cite{Met21,MePa20}
 that this class of $\lambda$'s corresponds precisely to the projective structures on $M$ (i.e.~torsion free affine connections, modulo reparametrisation of their geodesics) and with this interpretation projective twistor spaces have been considered much earlier \cite{Dub83,OB-R_85}.
Coming back to the $2$-sphere and  diagram \eqref{cdiag}, one of the main results in  \cite{LeMa02} can be rephrased as follows:  There is a holomorphic blow-down map
 \[\beta_{0,\P}\colon Z_\P(S^2,g_{\mathrm{can}},0)\to \C P^2\] with $\beta_{0,\P}(\partial Z_\P)=\R P^2$ and  small perturbations $R\approx \R P^2$ (or large `docile' ones, see \cite{LeMa10}) can be realised by holomorphic blow-down maps $Z_\P(S^2,g_\mathrm{can},\lambda)\to \C P^2$ for a reversible Zoll thermostat $(g_\mathrm{can},\lambda) $ with $\deg(\lambda)\le 3$ (that is, corresponding to a Zoll projective structure). Moreover, Zoll metrics are distinguished by  $R$ satisfying a Lagrangian condition.

Let us go back upstairs and consider a totally real surface $P\approx \bar \Delta\subset \C P^1\times \C P^1$ that is realised by a holomorphic blow-down map associated to a conformally magnetic system $(g,\lambda)$ as in Theorem \ref{introthmD}. We claim that:
\begin{equation}\label{oddimpliesswap}
	\lambda \text{ odd }\quad \Rightarrow \quad P = m^{-1}(R) \text{ for some $R\subset \C P^2$}.
\end{equation}
Indeed,  $\lambda$ being odd implies that the flip $(x,v)\mapsto (x,-v)$ is an automorphism of $Z(g,\lambda)$ and pushing this forward along the blow-down map gives an automorphism $\Phi$ of $\C P^1\times \C P^1\backslash P$ that fixes the diagonal $\Delta = \{w_1=w_2\}$ pointwise. It is not hard to see that $\Phi$ extends continuously across $P$, that is,  $\Phi\in \mathrm{Aut}(\C P^1\times \C P^1)$. But the only such automorphisms that fix $\Delta$ are the identity and the swap $(w_1,w_2)\mapsto (w_2,w_1)$. We can rule out the identity (since the flip acts nontrivially on $DS^2$), and thus $P$ is swap-invariant and hence double covers a totally surface $R\subset \C P^2$, as claimed. This argument shows that our construction indeed yields  non-trivial Zoll magnetic systems on $S^2$ rather than just recovering Zoll metrics. For a converse of \eqref{oddimpliesswap} we additionally have to ask the corresponding disk family to lie in
\[
	\mathcal F_*(P) =\{(\Delta_q:q\in Q)\in \mathcal F(P): \text{each $\Delta_q$ is swap-invariant}\}.
\]
Of course, these disk families are precisely the ones that descend to $\C P^2$ and, assuming that $R=m(P)$ is {\it weakly unknotted} (that is, $R=\varphi(\R P^2)$ for a diffeomorphism $\varphi$ of $\C P^2$),  it was shown in \cite{LeMa10} and independently by Rochon \cite{Roc11} that they are unique if they exist, that is,  \[|\mathcal{F}_*(P)|\le 1.\]

Finally, it is instructive to make the functional freedom for Zoll magnetic systems appear explicitly: To this end we  make the identification $S^2\times S^2 \cong \C P^1 \times \C P^1$ implicit and note that symplectic form can be written as
\[
	\Im \Upsilon = \frac{1}{8} d\left[\frac{xdy - y dx}{1-x\cdot y}\right] \text{ on } S^2\times S^2\backslash \Delta.
\] 
(See Remark \ref{symps2} for this computation.) Weinstein's tubular neighbourhood theorem can then be implemented explicitly by the following diffeomorphism:
\[\begin{split}
	S^2\times S^2\backslash \Delta \to T^*S^2 = \{(q,p)\in S^2\times \R^3: q\cdot p = 0\},\\ (x,y) \mapsto (q,p):= \left(\frac{x-y}{|x-y|}, \frac{x+y}{|x-y|}\right).
	\end{split}
\] 
The antidiagonal $\bar \Delta = \{(x,-x):x\in S^2\}$ is mapped onto the zero section and a straightforward computation shows that
\[
	pdq  = -\frac{1}{2} \frac{xdy - ydx}{1-x\cdot y},
\]
such that the map is (up to a factor) indeed a symplectomorphism with respect to the canonical symplectic structure on $T^*S^2$. Moreover, the swap $(x,y)\mapsto (y,x)$ corresponds to $(q,p)\mapsto (-q,p)=(\mathsf{a}(q),-d\mathsf{a}_q^\top (p))$, where $\mathsf{a}\colon S^2\to S^2$ is the antipodal map. In the graph parametrisation
\[
	 \Omega^1(S^2,\R)\cap \ker d \ni \alpha \rightarrow P_\alpha 
\]
the swap invariance of $P_\alpha$ thus corresponds to the equation $\mathsf{a}^*\alpha = -\alpha$, that is, $\alpha$ must be odd with respect to the antipodal map. Of course, since we are on $S^2$ we must have $\alpha = dh$ for some function $h\colon S^2\to \R$ (say, normalised to have zero average) and we can further translate this into oddness of $h$. This recovers Guillemin's functional freedom (as was already observed in \cite{LeMa02}) and illustrates what is gained by switching on a magnetic field: The functional freedom then increases to {\it all} zero average functions $h\colon S^2\to \R$.

\subsection*{Acknowledgements} We would like to thank Jonny Evans for some helpful discussions about the moduli space of holomorphic disks. We are also grateful to Tyson Klingner for discussions on ruled surfaces. 
JB was supported by the DFG  through the SFB 1720 -- 539309657.
 GPP was supported by NSF grant DMS-2347868.

\section{Preliminaries}

\subsection{Blow-down maps}\label{sec_blow} Let $W$ be a manifold and $P\subset W$ a properly embedded submanifold  of codimension $k\ge 1$. The {\it blow-up} of $W$ along $P$ (in the sense of Melrose \cite{MelroseDAOMWC}) is a manifold with boundary, denoted $[W,P]$, together with a smooth {\it blow-down map}
\[
	\beta\colon [W,P]\to W.
\]
The map $\beta$ sends the interior of $[W,P]$ diffeomorphically onto $W\backslash P$, while over the boundary there is a natural identification
\[
	\begin{tikzcd}
		\partial [W,P] \arrow["\sim"]{r}\arrow[swap,"\beta"]{dr} & SNP\arrow{d} \\
		& P
	\end{tikzcd}
\]
with the spherical normal bundle $SNP\to P$. The latter is given by first forming  the normal bundle $NP:=TW|_P/TP$ and then taking the quotient $SNP:=(NP\backslash 0)/\R_{>0}$ with respect to fibrewise scaling; this results in an $S^{k-1}$-bundle over $P$.

The blow-up along $P$ can be constructed as follows: By the tubular neighbourhood theorem there is a diffeomorphism $n\colon NP\xrightarrow{\sim} U\subset W$ of the total space of $NP$ with a neighbourhood of $U$ of $P$ that restricts to the identity along the $0$-section. Upon choosing a fibrewise inner product  on $NP$, there is a natural embedding  $\iota\colon SNP\to NP$ as unit sphere bundle. One then defines
\[
\begin{split}
	[W,P]&:=W\backslash P\cup_\varphi \left(SNP\times [0,\infty)\right),\\
	 \beta|_{W\backslash P} &:=\Id, \\
	\beta|_{SNP\times [0,\infty)}(p,[w],r) &:= n(r\cdot\iota(p,[w])),
\end{split}	
\]
where the gluing diffeomorphism $\varphi\colon SNP\times(0,\infty)\xrightarrow{\sim}  U\backslash P \subset W\backslash P$ is the obvious one that makes $\beta$ into a continuous map. The blow-up $[W,P]$ comes with a natural smooth structure that makes $\beta$ smooth; moreover, it is independent of the involved choices in the sense that altering $n$ or $\iota$ produces  a blow-up $[W,P]'$ and a blow-down map $\beta'\colon [W,P]' \to W$ such that $\beta'=\beta\circ \psi$ for a diffeomorphism $\psi\colon [W,P]'\xrightarrow{\sim}[W,P]$. More generally, we adopt the following terminology:

 \begin{definition}\label{def_blowdown}
 	Let $Z$ be a manifold with boundary and $f\colon Z\to W$ a smooth map. We will say that $f$ is a {\it blow-down map (along the submanifold $P\subset W$),} if there exists a diffeomorphism $\psi\colon Z\xrightarrow{\sim} [W,P]$ such that $f = \beta\circ \psi$.
 \end{definition}
 
\begin{remark}
	We also speak of $f\colon Z\to W$ as a blow-down map, when $P$ has not been specified beforehand. It is then understood that $P:=f(\partial Z)$ and that this is a properly embedded submanifold of $W$. 
\end{remark}

The purpose of this section is to identify conditions on $f$ that allow us to identify it as blow-down map of a given submanifold $P\subset W$.  Recall that a {\it boundary defining function} on $Z$ is a smooth map $\rho\colon Z\to [0,\infty)$ with $\partial Z=\{\rho =0\} $ and $d\rho\neq 0$ on $\partial Z$.

\begin{lemma}[Lifts  through the blow-down map]\label{lem_lift}
	Let $Z$ be a manifold with boundary and $f\colon Z\to W$ a smooth map with the following properties:
	\begin{enumerate}
		\item $f^{-1}(P)=\partial Z$
		\item $df_x\colon N_x\partial Z\to N_{f(x)}P$ is injective for all $x\in Z$.
	\end{enumerate}
	Then $f = \beta\circ \hat f$ for a smooth map $\hat f\colon Z\to [W,P]$ that can be specified by
	\[
		 \hat f(x) := \begin{cases}
			f(x)\in W\backslash P, & x\in Z^\circ\\
			[df_x(\nu)] \in SNP=\partial [W,P], & x\in \partial Z, \nu \in T_x Z\backslash T_x\partial Z.
		\end{cases}
	\]
Moreover, if $\rho_Z$ and $\rho_{[W,P]}$ are  boundary defining functions, then $(\hat f^*\rho_{[W,P]})/\rho_Z$ extends to a nowhere vanishing function $ Z\to \R$.
\end{lemma}

\begin{proof}
	Smoothness in $Z^\circ$ is clear. Whether $\hat f$ is smooth up to the boundary is a local question, so we can restrict our attention to a neighbourhood $V$ of some $x_0\in \partial Z$ and assume that $f(V) \subset \nu(NP|_U)$ for an open set $U\subset P$ with $NP|_U\cong U\times \R^k$. Let us make these identifications implicit and simply assume that $f(V)\subset U\times \R^k$. Then $g = \mathrm{pr}_2\circ f\colon V\to \R^k$ is a smooth map and by property (i) we have $g|_{V\cap \partial Z} = 0 $. We may therefore write $g(x)=\rho_Z(x)h(x)$ for a boundary defining function $\rho_Z\colon V\to \R$ and a smooth function $h\colon V\to \R^k$. Since
	\[
		dg_x(v)=(d\rho_Z)_x(v) h(x),\quad (x,v)\in TZ|_{\partial Z},
	\]
	property (ii) implies that $h|_{V\cap \partial Z}\neq 0$. After shrinking $V$ if necessary, we may assume that $h\neq 0$ on $V$ and hence the following functions are smooth:
	\[
		r(x)=\rho_Z(x)|h(x)|\quad \text{ and } \quad  \omega(x)=\frac{h(x)}{|h(x)|},\quad x\in V.
	\]
	Hence also the map
	\[
		\hat f|_V\colon V\to U \times S^{k-1}\times [0,\infty),\quad \hat f(x)=(\mathrm{pr}_1\circ f(x), \omega(x), r(x))
	\]	
	is smooth and one checks that this locally describes the lift defined above. Locally we can take $\rho_{[W,P]} = r$ as boundary defining function on the blow-up, and obtain $f^*\rho_{[W,P]}/\rho_Z = |h|$, which has the desired property.
\end{proof}

\begin{lemma}\label{lem_melroselocdiff}
	Let $Z$ be a manifold with boundary with $\dim Z=\dim W$ and suppose that $f\colon Z\to W$ satisfies the properties of Lemma \ref{lem_lift} and additionally:
	\begin{enumerate}
	\setcounter{enumi}{2}
		\item there exist volume forms $\omega_Z$ and $\omega_W$ such that $f^*\omega_W = \rho_Z^{k-1} \omega_Z$ for a boundary defining function $\rho_Z\colon Z\to \R$ (where $k=\operatorname{codim}(P)$);
		
	\end{enumerate}
	Then the lift $\hat f\colon Z\to [W,P]$ is a local diffeomorphism. 
\end{lemma}

\begin{remark}\label{rmk_covering}
If $Z$ is compact, then the lift $\hat f$ is a proper local diffeomorphism and therefore  a covering map \cite[Lemma 2]{Ho75}. To show that $\hat f$ is a global diffeomorphism it thus suffices to exhibit a single $w_0\in W\backslash P$ with $|f^{-1}(\{w_0\})| = 1$.
\end{remark}

\begin{proof}[Proof of Lemma \ref{lem_melroselocdiff}]
	We have $\beta^*\omega_W=\rho_{[W,P]}^{k-1} \omega_{[W,P]}$ for a suitable boundary defining function and a volume form on $[W,P]$. The relation $\hat f = \beta^{-1}\circ f$ on $Z^\circ$ yields
	\[
		\hat f^*\omega_{[W,Z]} = \left(\frac{\rho_Z}{\hat f^*\rho_{[W,P]}}\right)^{k-1} \omega_Z
	\]
	and by the preceding lemma, the expression in the bracket extends to a nowhere vanishing function on all of $Z$, as desired. 
\end{proof}

\begin{lemma}\label{lem_lift2}	Let $Z$ be a manifold with boundary with $\dim Z=\dim W=d\ge 2$ and assume that $\operatorname{codim}(P)=2$. Suppose that $f\colon Z\to W$ satisfies:
\begin{enumerate}
	\item  $f(\partial Z)\subset P$ and $f(Z^\circ)\subset W\backslash P$;
	\item $f|_{Z^\circ}\colon Z^\circ \to W\backslash P$ is a local diffeomorphism;
	\item $L_x=\ker (df_{x}\colon T_{x}Z\to T_{f(x)}W)\subset T_x\partial Z$ and $\dim L_x=1$ for all $x\in \partial Z$;
	\item given a smooth $1$-parameter family $(x_\lambda,\nu_\lambda) \in TZ|_{\partial Z}$ with  $L_{x_\lambda} = \spn_\R \dot x_\lambda$ and $\nu_\lambda\notin T_{x_\lambda}\partial Z$  consider
	\[
		w_\lambda=df_{x_\lambda}(\nu_\lambda) \in T_{p}W,\quad p=f(x_\lambda)=f(x_0),\quad  (|\lambda|<\epsilon).
	\]
	Then $\{w_0,\dot w_0\}$ can be complemented, with vectors in $T_pP$,  to a basis of $T_pW$.
\end{enumerate}
Then there exists a local diffeomorphism  $\hat f\colon Z\to [W,P]$ with $f=\beta\circ \hat f$.
\end{lemma}

\begin{proof}
	Since $f$ maps $\partial Z$ into $P$, the differental at $x\in \partial Z$ is block triangular:
	\[
		df_x= \begin{bmatrix}
			A & B\\
			0 &C
		\end{bmatrix}\colon T_x\partial Z\oplus N_x\partial Z\xrightarrow{\,} T_{f(x)}P\oplus N_{f(x)} P.
	\]
	By property (iii) we have $\operatorname{rk}df_x = d-1$ and since $\operatorname{rk} A \le \dim P = d-2$, this enforces $\operatorname{rk}C=1$. Thus $f$ meets the requirements of Lemma \ref{lem_lift} and the lift $\hat f$ is smooth.
	To show that $d\hat f$ has full rank at a point  $x_0\in \partial Z$, it suffices to show that $\ker d\hat f_{x_0}\cap L_{x_0} = 0$. But $\hat f(x_\lambda)=[w_\lambda]\in SN_pP = \beta^{-1}(\{p\})$ and thus
	\[
	\begin{split}
		\dot x_0 \in \ker d\hat f_{x_0}\cap L_{x_0} \quad &\Leftrightarrow \quad \frac{d}{d\lambda}\Big|_{\lambda=0}[w_\lambda] = 0 \in T_{w_0} (SN_pP)\\[.2em]
		&\Leftrightarrow \quad \{[w_0],[\dot w_0]\} \text{ linearly dependent in } N_pP.
	\end{split}	
	\]
	The second equivalence is a general observation in $\R^k$; namely if $(w_\lambda)\subset \R^k\backslash 0$, then $(d/d\lambda)(w_\lambda/|w_\lambda|) = \left(\dot w_\lambda - w_\lambda  (\dot w_\lambda\cdot w_\lambda)/|w_\lambda|\right)/|w_\lambda|$ and this derivative vanishes if and only if $\{w_\lambda,\dot w_\lambda\}\subset \R^k$ is linearly dependent. Back to $N_pP$, we conclude the proof by noting that linear independence of $\{[w_0],[\dot w_0]\}\subset N_pP$ is equivalent to condition (iv).
\end{proof}

\subsection{Bundle pairs} Let $(X,Y)$ be a smooth manifold , together with an embedded smooth submanifold. A {\it bundle pair} $(E,F)$ over $(X,Y)$ consists of a complex vector bundle $E\to X$ together with a real subbundle $F\subset E|_Y$ such that $E|_Y = F \oplus i F$  (The bundle $F$ is then also called a {\it totally real}). 

\subsubsection{Maslov index}\label{sec_chernmaslov} If $(E,F)\to (\DD,\partial \DD)$ is a bundle pair over the disk, one defines the {\it Maslov index} $\mu(E,F)\in \Z$ as follows (cf.~\cite[Section C.3]{McSa04}): If $\rk_\C(E)=1$ and $\xi\in C^\infty(\DD,E)$ is a nowhere vanishing section, then $\mu(E,F)$ is the unique integer $k\in \Z$ such that for all $\theta\in \R$ it holds that:
 \begin{equation}\label{maslovchar}
 	e^{ik\theta/2}\xi(e^{i\theta}) \in F.
 \end{equation}
If $\rk_\C(E)=n$, one defines $\mu(E,F):= \mu(\Lambda^n E,\Lambda^n F)$. Alternatively, one has
\begin{equation}\label{maslovchern}
	\mu(E,F)=\langle 2 c_1(E,F),[\DD]\rangle,
\end{equation}
where $\langle\cdot,\cdot\rangle$ denotes the duality pairing  $H^2(\DD,\partial\DD)\times H_2(\DD,\partial \DD)\to \Z$ and $c_1(E,F)$ is the {\it relative $1$st Chern class}. The latter can be defined for a general bundle pair $(E,F)\to (X,Y)$ as class in $H^2(X,Y; \frac 12 \Z)$.

For $\rk_\C E=1$ the Chern-Weil description of the relative $1$st Chern class is
\[
	c_1(E,F) = \frac{i}{2\pi} [(\Omega,\omega)] \in H^2_{\mathrm{dR}}(X,Y;\C),
\] 
where the tuple $(\Omega,\omega)\in \Omega^2(X,i\R)\times \Omega^1(Y,i\R)$ is defined as follows: Choose a Hermitian inner product on $E$, let $\nabla$ be a unitary connection on $E$ and $\nabla'$ the induced connection on $E|_Y$. If $U\subset Y$ is an open set supporting unit sections $e_U\in \Omega^1(U,E|_Y)$ and $f_U\in \Omega^1(U,F)$, then there exist $\alpha_U\in \Omega^1(U,i\R)$ and $v_U\colon U\to S^1$ with
\[
	\nabla'e_U = \alpha_U e_U\quad \text{ and }\quad f_U = v_U e_U.
\]
 We can consider the $1$-form
	 \(
	 	\omega_U = \alpha_U + v_U^{-1} dv_U\in \Omega^1(U,i\R)
	 \)
	 and note that
	 \(
	 	\nabla' f_U = \omega_U f_U.
	 \)
Given another such tuple $(U',e_{U'},f_{U'})$, we must have $f_{U'}=\pm f_{U}$ and hence $\omega_U=\omega_{U'}$ on the overlap $U\cap U'$. Hence these local $1$-forms glue to some $\omega\in \Omega^1(Y,i\R)$. If $\Omega$ is the curvature $2$-form of $\nabla$ and $i_Y\colon Y \hookrightarrow X$ the inclusion, then \(
	 	 d\Omega = 0\) and \(i_Y^*\Omega=d\omega
	 \)
	 and thus it defines a class in the relative de Rham cohomology \cite{BoTu82}---this is the class $[(\Omega,\omega)]$ considered above. As an illustration, we include the following:

	 \begin{proof}[Proof of formula \eqref{maslovchern}]  If $(E,F)\to (\DD,\partial \DD)$ has the frame $\xi\in C^\infty(\DD,E)$, we can define a Hermitian inner product  and a unitary connection by declaring $\xi$ to be of unit length and parallel. On $U=\{z\in \partial \DD: z\neq {-1}\}$  we can take $\alpha_U=0$ and $v_U(e^{i\theta}):= e^{ik\theta/2}$, $\theta \in (-\pi,\pi)$, where $k=\mu(E,F)$. This results in 
	 $\Omega=0$ and $\omega = \frac{ik}{2} d\theta$, such that
	 \[
	 	\langle c_1(E,F),[\DD]\rangle = \frac{i}{2\pi} \left(\int_\DD\Omega - \int_{\partial \DD} \omega \right) = \frac{k}{2} = \frac 12 \mu(E,F).
	 \]
 \end{proof}

\begin{lemma}\label{lem_totallyrealequivalence} Let $W$ be a complex manifold and $P\subset W$ an embedded submanifold with $\dim_\R P = \dim_\C W$. Write $\Lambda^{p,0}_\R W\to P$ for the
subbundle of $\Lambda^{p,0}W|_P$ that is made up of $(p,0)$-forms whose restriction to $TP$ is real valued.
Then the following are equivalent:
\begin{enumerate}
	\item $P$ is totally real (that is, $TP\oplus J TP = TW$ along $P$);
	\item the restriction map $\Lambda^{1,0}W|_P\to \Lambda^1_\C P$, $\omega\mapsto \omega|_{TP}$ is an isomorphism;
	\item the restriction map $\Lambda^{1,0}_\R W|_P\to \Lambda^1 P$, $\omega\mapsto \omega|_{TP}$ is an isomorphism;
	\item $\Lambda^{1,0}W|_P=\Lambda^{1,0}_\R W \oplus i\Lambda^{1,0}_\R W$.
\end{enumerate}
In this case we have
\[
	c_1(W,P):=c_1(\Lambda^{2,0}W,\Lambda^{2,0}_\R W) = - c_1(TW,TP).
\]	
\end{lemma}

\begin{proof}
	We prove the equivalence at a given point $p\in P$.\\[.4em]
	(i)$\Rightarrow$(ii): A real linear map $\omega\colon T_pP \to \C$ has a unique $J$-linear extension $\omega_\C\colon  T_pW\to \C$ defined via the decomposition in (i) by $\omega_\C(w_1+ J w_2) = \omega(w_1) + i \omega(w_2)$, where $w_1,w_2\in T_pP$. Hence (ii) holds true.\\[.4em]
	(ii)$\Leftrightarrow$(iii): Using (ii), a real linear map $\omega\colon T_pP\to \R$ has an extension  $\omega_\C\in \Lambda^{1,0}_pW$. But this automatically has to lie in $(\Lambda^{1,0}_\R)_p$,  so (iii) holds true. Vice versa, if $\omega\colon T_p P\to \C$ is real linear, then also real and imaginary part are and extending them as in (iii) yields (ii).
	\\[.4em]
	(ii)+(iii)$\Rightarrow$(iv): The restriction of $\omega_\C\in \Lambda^{1,0}_p W$ to $TP$ gives an element  $\omega \in (\Lambda^1_\C P)_p$. Let $\omega',\omega''\in (\Lambda^{1,0}_\R W)_p$ be the extensions of $\Re \omega$ and $\Im \omega$ guaranteed by (iii). Then the decomposition $\omega_\C = \omega' + i\omega''$ holds true on $T_pP$ and by (ii) this implies equality everywhere.	\\[.4em]
(iv)$\Rightarrow$(i): If $w\in T_pP \cap JT_pP$ and $\omega \in (\Lambda^{1,0}_{\R})_p$, then both $\omega(w)$ and $\omega(Jw)=i\omega(w)$ have to be real valued, which implies that $\omega(w)=0$. The decomposition in (iv) implies that $\omega_\C(w)=0$ for all $\omega_\C(w)=0$ for all $\omega_\C \in\Lambda^{1,0}_p W$ and thus $w=0$. Thus (i) holds true.	\\[.4em]
The statement regarding the relative $1$st Chern class follows from standard properties, analogous to the non-relative version: taking an exterior power does not change $c_1$ and dualising flips the sign of $c_1$.
\end{proof}

\subsubsection{Holomorphic bundle pairs} We say that $(E,F)\to (\DD,\partial \DD)$ is {\it holomorphic}, if $E\to \DD$ is a holomorphic vector bundle, that is, it comes with a first order differential operator $\bar \partial \colon C^\infty(\DD,E)\to \Omega^{0,1}(\DD,E)$ that obeys the Leibniz rule with respect to the scalar $\bar \partial$-operator. Its Fredholm theory with boundary values in $F$ is a version of the Riemann--Roch theorem \cite[Theorem C.1.10]{McSa04}:

\begin{theorem}\label{thmrr}
	Let $(E,F)\to (\DD,\partial \DD)$ be a holomorphic bundle pair. Then
	\begin{equation}\label{dbarrr}
	\bar \partial\colon  C^\infty_F(\DD,E)\to \Omega^{0,1}(\DD,E)
		\end{equation}
	has real Fredholm index $\mu(E,F)+n$, where $n$ is the complex rank of $E$. If $n=1$ and $\mu(E,F)\ge -1$, then \eqref{dbarrr} is onto.\qed
\end{theorem}

\begin{lemma}\label{lem_split}
Let 
\(
	0\to (E',F')\xrightarrow{\iota} (E,F)\xrightarrow{\pi} (E'',F'')\to 0
\)
be a short exact sequence of holomorphic bundle pairs over $(\DD,\partial \DD)$ such that 
\begin{equation}\label{h1obstr}
\bar \partial \colon C_{\Hom_\R(F'',F')}^\infty(\DD,\Hom_\C(E'',E')) \to \Omega^{0,1}(\DD,\Hom_\C(E'',E'))  
\end{equation}
is onto. Then there exists a holomorphic splitting $s_0\colon (E'',F'')\to (E,F)$. Any further holomorphic splitting is given by $s_\phi=s_0+\iota\phi$ for some $\phi$ in the kernel of \eqref{h1obstr}.
A sufficient criterion for \eqref{h1obstr} to be onto is that $\rk_\C E'=\rk_\C E''=1$ and \(\mu(E',F') \ge \mu(E'',F'')-1.\)
\end{lemma}

\begin{proof}
	Let $s\colon (E'',F'')\to (E,F)$ be a {\it smooth} splitting and consider $s_\phi = s + \iota \phi$ for $\phi$ in the domain of \eqref{h1obstr}, which also gives a splitting. Since $\pi s=\Id$, we have $0 = \bar \partial (\pi s) = \pi \bar\partial s$ and hence $\bar \partial s = \iota \psi$ for some $\psi\in \Omega^{0,1}(\DD, \Hom_\C(E'',E'))$. Taking $\phi$ to be any solution of $\bar \partial \phi + \psi = 0$ ensures that $\bar \partial s_\phi = \bar \partial s + \iota \bar \partial \phi = 0$. The sufficient criterion follows from  Theorem \ref{thmrr} and the observation
	\[
		\mu(\Hom_\C(E'',E'),\Hom_\R(F'',F')) = \mu(E',F') - \mu(E'',F'').
	\] 
\end{proof}

\begin{lemma}\label{lem_normalform} Let $(E,F)\to (\DD,\partial \DD)$ be a bundle pair of rank $1$ and Maslov index $k=\mu(E,F)\ge 0$. Suppose further that $E$ is holomorphic. Then there exists a holomorphic nowhere vanishing section $\eta\in C^\infty(\DD,E)$  such that
\[
	C^\infty_F(\DD,E)\cap \ker \bar \partial = \left\{a \eta: a(\omega)= \sum_{j=0}^k a_j \omega^j \text{ with } a_{j} = \bar a_{k-j}\right\} \cong \R^{k+1}.
\]
\end{lemma}

\begin{proof}
	Start with an arbitrary nowhere vanishing holomorphic section $\eta_0\in C^\infty(\DD,E)$. On $\partial \DD\backslash \{\pm 1\}$ we have $F = \spn_\R(v_\pm\eta_0)$  for smooth $v_\pm\colon \partial \DD\backslash \{\pm 1\} \to S^1$. Their squares must agree on the overlaps and so we get a smooth function $h\colon \partial \DD\to S^1$ with $h=v_\pm^2$ on $\partial \DD\backslash \{\pm 1\}$. It has degree $\deg h = k$ and therefore $h(\omega)\omega^{-k} =\exp(i g(\omega))$ for some $g\colon \partial \DD\to \R$. Write $(g_j:j\in \Z)$ for its Fourier modes and define
	\[
		\psi(\omega) = \exp( \frac{i}{2}g_0 + i \sum_{j\ge 1} g_j \omega^j),\quad \omega\in \DD.
	\]
	Then $\psi\colon \DD\to \C^\times$ is holomorphic and satisfies $\psi/\bar \psi = h\omega^{-k}$ on $\partial \DD$. Define $\eta=\psi \eta_0$ and consider $f=a \eta$ for some holomorphic $a\colon \DD\to \C$. Then $f|_{\partial \DD} \in F$ if and only if $\Im(\bar v_\pm a \psi)=0$ on $\partial \DD\backslash \{\pm 1\}$, which is to say that
	\[
		a = 	\bar a (v_\pm\psi)/(\bar v_\pm \bar \psi) = \bar a h (\psi/\bar \psi)=\bar a \omega^k.
	\]
	Expanding $a$ into a power series and comparing coefficients gives the restrictions $a_j=0$ for $j\ge k+1$ and $a_{j} = \bar a_{k-j}$ for $j=0,\dots, k$.
\end{proof}	

\begin{example}
	For the tangent bundle pair $(T\DD,T\partial \DD)$ we can take $\eta\equiv  \partial_t = i\partial_s$, where $s+it$ is the complex coordinate of $\DD$.
\end{example}

\subsection{Moduli space of holomorphic disks}\label{prelim_moduli}
Let $(W,P)$ be a complex surface together with an embedded totally real surface. We write
\[
	\mathcal B(P):=\{u\colon (\DD,\partial \DD)\looparrowright (W,P): \text{ smooth immersion up to } \partial \DD\}.
\]
This is a Fr{\'e}chet manifold with tangent space at $u\in \mathcal B(P)$ given by
\[
	T_u\mathcal B(P)=\{\xi \in C^\infty(\DD,u^*TW): \xi(\partial \DD)\subset u^*TP\}.
\]
Upon choosing a Riemannian metric on $W$ that makes $P$ totally geodesic, the associated exponential map provides us with a chart \(
	\exp_u\colon T_u\mathcal B(P)\to \mathcal B(P).
\)
If $u\in \mathcal B(P)$ is holomorphic, then $(u^*TW,u^*TP) \to (\DD,\partial \DD)$ is a holomorphic bundle pair and its Maslov index is, by definition, the {\it (total) Maslov index} of  $u$. We consider the following moduli spaces of immersed holomorphic disks:
\[
	\mathcal M_\mu(P):=\{u\in \mathcal B(P): u \text{ is holomorphic and has (total) Maslov index } \mu\}.
\]

\begin{theorem}\label{thm_modulimaster}  Let $\mu\ge 1$.
\begin{enumerate}
	\item\label{thm_modulimasteri}  The moduli space $\mathcal M_\mu(P)\subset \mathcal B(P)$ is a smooth submanifold of dimension $\mu+2$. Its tangent space at $u\in \mathcal M_\mu(P)$ is given by
\[
\qquad 	T_u\mathcal M_\mu(P)=\{\xi\in C^\infty(\DD,u^*TW) \text{ holomorphic}: \xi(\partial \DD)\subset u^*TP\}.
\]
\item\label{thm_modulimasterii}  Given a compact set $K\subset \mathcal M_\mu(P)$, if $P'$ is a totally real surface that is sufficiently close to $P$ in the $C^\infty$-topology, there exists a neighbourhood $U\subset \mathcal M_\mu(P)$ of $K$ and a $C^1$-embedding
\(
	\Psi\colon U\to \mathcal M_\mu(P').
\)
This can be chosen to have the following properties: $\mathrm{ev}_0\circ \Psi = \mathrm{ev}_0$, where $\mathrm{ev}_0(u)=u(0)$; and $(\Psi(u))_\lambda=\Psi(u_\lambda)$, where $u_\lambda(\omega)=u(e^{i\lambda} \omega)$. 
 		
\end{enumerate}

\end{theorem}

\begin{proof} The arguments in the proof are quite standard, so we confine ourselves to a brief sketch, following McDuff--Salamon \cite{McSa04}. To set up a Banach space implicit function theorem, we pass to the Sobolev completion of $\mathcal B(P)$,
		\[
		\mathcal B^{k,p}(P):=\left\{u\in W^{k,p}(\DD,W): \begin{array}{l} u \text{ is an immersion up to $\partial \DD$}\\
		\text{and } u(\partial \DD)\subset P
		\end{array} \right\}.
		\]
		We assume that $k\in \mathbb N, p\in (1,\infty)$ and $k-2/p>1$, such $\mathcal B^{k,p}(P)\subset C^1(W,P)$. This is a smooth separable Banach manifold with tangent space
		 \(
		 T_u\mathcal B^{k,p}(P):=\{\xi\in W^{k,p}(\DD,u^*TW): \xi (\partial \DD)\subset u^*TP\}.
		 \)
		We also consider the Banach bundle  $\mathcal E^{k-1,p}\to \mathcal B^{k,p}(P)$ with fibres
		 \[
		 	\mathcal{E}_u^{k-1,p}=W^{k-1,p}(\DD,\Lambda^{0,1}\DD\otimes_\C u^*TW),\quad u\in \mathcal B^{k,p}(P),
		 \]
		 as well as the section 
		 \[
		 	\mathcal S\colon \mathcal B^{k,p}(P)\to \mathcal E^{k-1,p},\quad \mathcal S(u)=(u,\bar \partial_J(u)).
		 \] 
		 Here $\bar \partial_J(u)$ is the non-linear Cauchy--Riemann operator associated to the complex structure tensor $J$ on $W$. Writing $\omega=s+it$ for the complex coordinate on $\DD$, this is given by \[\bar \partial_J(u)=\frac 12\big(\partial_su +J(u)\partial_t u\big) d\bar \omega.\]
		Then $\mathcal S(u)=0$ if and only if $u$ is holomorphic. To prove that $\mathcal M_\mu(P)$ is a smooth submanifold of $\mathcal B^{k,p}(P)$, it thus suffices to show that $\mathcal S$ is transverse to the zero section over $\mathcal M_\mu(P)$ and apply the implicit function theorem.
		
		Transversality 	at $u\in \mathcal M_\mu(P)$ means that the so-called vertical differential \(D_u\colon T_u\mathcal B^{k,p}(P)\to \mathcal E_u^{k-1,p}\) has a bounded right inverse.
		 One defines $D_u\xi$ as composition of  \(d\mathcal S(u)\colon T_u\mathcal B^{k,p}(P) \to T_{(u,0)}\mathcal E^{k-1,p}=T_{ u} \mathcal B^{k,p}(P)\oplus \mathcal E^{k-1,p}_u\) with the projection onto the second summand. In our situation, where $J$ is integrable, one readily computes that
		\[
			D_u = \bar \partial_{u^*TW}\colon W^{k,p}_{u^*TP}(\DD,u^*TW)\to W^{k-1,p}(\DD,\Lambda^{0,1}\DD\otimes_\C u^*TW),
		\]
		that is, the standard $\bar \partial$-operator of the holomorphic vector bundle $u^*TW\to \DD$. The subscript in the domain means that $\xi(\partial \DD)\subset u^*TP$. By Theorem C.1.10 in \cite{McSa04}, this is a Fredholm operator and the dimension of kernel and cokernel are independent of the precise regularity $(k,p)$. In particular, $D_u$ being surjective already implies the existence of a bounded right inverse and we may check surjectivity over smooth sections.
		
		To prove the required surjectivity, we consider the short exact sequence  \[0\to (T\DD,T\partial \DD)\xrightarrow{du} (u^*TW,u^*TP) \to (E,F) \to 0,\] where $(E,F)$ is defined as quotient. All members of this sequence are holomorphic bundle pairs over $\DD$. Since $\mu(T\DD,T\partial \DD)=2$, the additivity of the Maslov index implies that $\mu(E,F)=\mu-2$. Choosing a smooth (not necessarily holomorphic) right splitting, we may view $E\subset u^*TW$ as complex subbundle and $F=E|_{\partial \DD}\cap u^*TP$. With respect to this splitting we have
		\[
			\bar \partial_{u^*TW} = \begin{bmatrix}
				\bar \partial_{T\DD} & *\\
				0 & \bar \partial_E
			\end{bmatrix}\colon  C^\infty_{T\partial \DD}(\DD,T\DD)\oplus C^\infty_F(\DD,E)\to \Omega^{0,1}(\DD,T\DD)\oplus \Omega^{0,1}(\DD,E),
		\]
		where the $0$-entry stems from $du(T\DD)\subset u^*TW$ being a holomorphic subbundle. In view of  Theorem \ref{thmrr}  the bound $\mu(E,F)=\mu-2 \ge -1$ guarantees that $\bar \partial_E$ is onto in this setting. Of course, also $\bar \partial_{T\DD}$ is onto and thus $\bar \partial_{u^*TW}$ must be onto as well.

		The preceding considerations imply that $\mathcal M_\mu(P)$ is a smooth submanifold of $\mathcal B^{k,p}(P)$, whose tangent space at $u$ equals the kernel of $D_u$, that is, the holomorphic sections of $u^*TW$ with boundary values in $u^*TP$. By Theorem \ref{thmrr} this has dimension $\mu+2$, as claimed ---  this concludes the proof of part (i).
		
		Now suppose that $P'\subset W$ is a totally real surface near $P$, which we take to mean that $P'=\psi(P)$ for diffeomorphism  $\psi\in \mathrm{Diff}(W)$ near $\Id$.	 Writing $J'=\psi^*J$, we can then consider the moduli space \[\mathcal M_\mu(P,J')=\{u\in \mathcal B^{k,p}_\mu(P): \bar \partial_{J'}(u)=0\},\]
		which, similar to above, is a smooth $\mu+2$-dimensional manifold. Given  $K\subset \mathcal M(P)$ compact, we now seek out an embedding $\Phi\colon  K\supset U\to \mathcal M_\mu(P,J')$
and obtain $\Psi\colon U\to \mathcal M(P')$ as $\Psi(u)=\psi\circ (\Phi(u))$.	

Having translated the perturbation problem to a fixed $P$ and a varying $J'$, we are led to consider the Banach manifold
\[
	\mathcal J^\ell(P) = 
	\{ J'\in C^\ell(W,\operatorname{End}_\R(TW)):  {J'}^2=-\Id, ~TP \cap J'TP =0  \},\quad \ell\ge k+2.
\]
We redefine $\mathcal E^{k-1,p}$ as a Banach bundle $\mathcal E^{k-1}\to \mathcal B^{k,p}(P)\times \mathcal J^\ell(P)$ with fibres
\[\mathcal E_{(u,J')}^{k-1,p}=W^{k-1,p}(\DD,\Lambda^{0,1}\DD\otimes_{J'} u^*TW)\] and let $\mathcal S$ be the section  $\mathcal S(u,J')=(u,\bar\partial_{J'}(u))$. Both the bundle projection itself and the section $\mathcal S$ are not $C^\infty$, but only $C^{\ell-k}$-regular (cf.\,\cite[Proposition 3.2.1]{McSa04}). For $u_0\in \mathcal M_\mu(P)$ we have already shown that $d_1\mathcal S(u_0,J)$ is transverse to the $0$-section and so the implicit function theorem guarantees that for nearby $J'=\psi^*J$ there is an embedding $\Phi\colon \{u_0\}\supset U\to \mathcal M(P,J')$. 
This proves part (ii) for $K=\{u_0\}$ consisting of one element.

The passage to general compact sets is done as follows:  we consider a smooth $1$-parameter family $(J_t)$ in $\mathcal J^\ell(P)$ with $J_0=J$ and $J_1=J'$ and define $\Phi$ as time-$1$-map of a suitable flow $(\Phi_t)$ on $\mathcal B^{k,p}(P)$. Upon choosing a connection on $\mathcal E^{k-1,p}$ we can consider vertical differentials also away from the zero section, and we denote them by\[d_1^\mathtt{v}\mathcal S(u,J')\colon T_u\mathcal B^{k,p}(P)\to \mathcal E^{k-1,p}_{(u,J')} \quad \text{ and }\quad d_2^\mathtt{v}\mathcal S(u,J')\colon T_{J'}\mathcal  J^\ell(P)\to \mathcal E^{k-1,p}_{(u,J')}. \]
Of course, for $u\in \mathcal M_\mu(P)$ we have $d_1^\mathtt{v}\mathcal S(u,J)=D_u$ and thus a bounded right inverse $(d_1^\mathtt{v}S)^{-1}$ exists also in a neighbourhood of $K\times \{J\}$.  Taking the vertical differential of $\mathcal S(u_t,J_t)=0$ in the $t$-variable and isolating $\dot u_t$ leads to the ODE
	\begin{equation}\label{odeut}
		\dot u_t = - \big(d_1^\mathtt{v} \mathcal S(u_t,J_t)\big)^{-1}d^\mathtt{v}_2\mathcal S(u_t,J_t)\dot J_t \in T_{u_t}\mathcal B^{k,p}.
	\end{equation}
	The right hand side defines a vector field of regularity $C^{\ell-k-1}$ on an open subset of $\mathcal B^{k,p}(P)$ and since $\ell\ge k+2$, the ODE admits a flow of regularity $C^{\ell-k-1}$.
	If $J'$ is sufficiently close to $J$ and $u_0$ lies in a small neighbourhood $U\subset \mathcal M_\mu(P)$ of $K$, the solution $(u_t)$ exists at least until $t=1$. Hence the time-$1$-map gives a well-defined $C^{\ell-k-1}$-embedding $\Phi$ from $U$ onto its image in $\mathcal B^{k,p}(P)$. By construction this satisfies $\mathcal S(\Phi(u),J')=0$ for all $u\in U$ and hence $\Phi(U)\subset \mathcal M(P,J')$, as desired. 

\todo{J: Can we achieve $\Phi\in C^\infty$? Certainly any finite $C^m$ is possible, but the neighbourhood $U$ depends on $m$ and might shrink as we take $m\to \infty$.}

	The property $\mathrm{ev}_0\circ \Psi = \mathrm{ev}_0$ can be achieved as follows: The set $\mathrm{ev}_0(K)\subset W\backslash P$ is compact and we can realise small perturbations of $P$ via diffeomorphisms $\psi$ that are $\equiv\Id$ in a neighbourhood of $\mathrm{ev}_0(K)$. Likewise, the $1$-parameter family $(J_t)$ from the preceding paragraph can be taken to be constant near $\mathrm{ev}_0(K)$.  Shrinking $U$ if necessary, we can thus assume that $\dot J_t = 0$ on $\mathrm{ev}_0(U)$ and hence the sets $\mathrm{ev}_0^{-1}(\{x\})\subset \mathcal B^{k,p}(P)$ ($x\in \mathrm{ev}_0(U)$) are invariant for the flow of \eqref{odeut}. This implies that $\mathrm{ev}_0\circ \Phi = \mathrm{ev}_0$ and likewise for $\Psi$.
	
	Finally, the equivariance with respect to rotations can be deduced from $\mathcal S$ being equivariant with respect to the following $S^1$-action:  $r_\lambda\colon \DD\to \DD,r_\lambda(\omega)=e^{i\lambda}\omega$ acts on $ \mathcal E^{k-1,p}$ by $(u,J,T)\mapsto (u\circ r_\lambda, J, T\circ d r_\lambda)$.
\end{proof}

\begin{lemma}\label{lem_normalformframe}
	If $u\in \mathcal M_4(P)$, then there exists $\eta \in C^\infty(\DD,u^*TW)$ such that $\{\eta,iu'\}$ is a holomorphic frame of $u^*TW$, and moreover, we have $\xi \in T_u\mathcal M_4(P)$ if and only if
	\begin{equation}\label{normalform4}
	\xi(\omega) = (a_0+a_1 \omega + \bar a_0 \omega^2) \eta(\omega) + (b_0+b_1 \omega+\bar b_0)iu'(\omega)
	\end{equation}
	for $(a_0,a_1,b_0,b_1)\in \C\times \R\times \C\times \R$.
	
\end{lemma}

\begin{proof}
	Define a holomorphic bundle pair $(E,F)$ such that $0\to (T\DD,T\partial \DD)\xrightarrow{du} (u^*TW,u^*TP)\to (E,F)\to 0$ is short exact. Then $\mu(E,F) - \mu(T\DD,T\partial \DD)=2-2=0$ and  Lemma \ref{lem_split} applies. Upon choosing a holomorphic splitting, we can view $E\subset u^*TW$ as holomorphic subbundle with $F=E|_{\partial \DD}\cap u^*TP$. Any holomorphic section $\xi$ of $(u^*TW,u^*TP)$ then has a unique direct sum decomposition $\xi = \xi_1 + du(\xi_2)$, where $\xi_1\in C^\infty_F(\DD,E)$ and $\xi_2\in C^\infty_{T\partial \DD}(\DD,T\DD)$ are holomorphic.
	Via Lemma \ref{lem_normalform}, both summands can be put into normal form, namely $\xi_1=a\eta$ and $\xi_2 = b(i\partial_s)$, thus arriving at \eqref{normalform4}.
\end{proof}

{
\begin{lemma}\label{lem_opentouch}
The set of $u\in \mathcal M_\mu(P)$ with $u^{-1}(P)=\partial \DD$ is open.
\end{lemma}

\begin{proof}
	Fix $u\in \mathcal M_\mu(P)$ and $\omega_0 \in \partial \DD$. Then in a small neighbourhood $U$ of $u(\omega_0)\in P$ we can write $P$ as a regular $0$-level set of a submersion $\rho\colon U\to \R^2$. Write $G=u^{-1}(U)\subset \DD$, then $\rho\circ u$ vanishes on $G\cap \partial \DD$ and thus 
	\[
		h(\omega) :=\frac{ \rho\circ u(\omega)}{1-|\omega|^2},\quad \omega\in G,
	\]
	extends to a smooth function $h\colon G\to \R^2$. For $\omega \in G\cap \partial \DD$ we write $\nu(\omega)$ for the outward pointing unit normal and note that by the holomorphicity of $u$,
	 \[h(\omega) = \frac 12 d\rho_{u(\omega)}(du_{\omega}(\nu(\omega)) = -\frac 12 d\rho_{u(\omega)} \big( J(u(\omega)) du_{\omega}(\nu(\omega)^\perp) \big).
	 \]
	 Since $u$ is an immersion and maps $\partial \DD$ into $P$, the vector $du_\omega(\nu(\omega)^\perp)$ is nonzero and lies in $TP$. Since $P$ is totally real,  the complex structure $J(u(\omega))$ sends this vector outside of $T_{u(\omega)}P = \ker d\rho_{u(\omega)}$ and we conclude that $h(\omega)\neq 0$.
	 
Hence $\omega_0$ has a compact neighbourhood on which $h$ is non-vanishing. Since $h$ depends continuously on $u$, this property is preserved by small perturbations. Covering $\partial \DD$ with such neighbourhoods, we find a neighbourhood $\mathcal U\subset \mathcal M_\mu(P)$ of $u$ and a radius $0<r<1$ with $v(\omega)\notin P$ for all $v\in \mathcal U$
and $r<|\omega|<1.$ Finally, if we assume that $u^{-1}(P)=\partial \DD$, then $u(\{|\omega|\le r\})\cap P = \emptyset$ and since $|\omega|\le r$ is compact, this is an open condition. Hence after shrinking $\mathcal U$, we obtain $v^{-1}(P)=\partial \DD$ for all $v\in \mathcal U$.
	\end{proof}
}

\section{Transport twistor space}\label{sec_tts}

In this section we discuss transport twistor spaces and prove Theorem \ref{introthmB}.  Throughout $(M,g)$ is an oriented Riemannian surface without boundary\footnote{We just assume this for simplicity---most results in this section readily carry over to the setting with boundary. Then $DM$ is a manifold with corners and $SM$ is one of its boundary hypersurfaces.}. Frequently, this is also viewed as a Riemann surface, whose tangent spaces carry the complex structure induced by $g$ and the orientation, explicitly:
\[
	(x,\omega v) := (x,\Re(\omega)  v + \Im(\omega) v^\perp),\quad (x,v)\in TM,\quad \omega\in \C.
\]

\subsection{Geometry on $SM$}
We start with a brief review of the geometry on the unit circle bundle, referring to  \cite{PSU23} for a more detailed discussion.
There is a natural frame $\{X,H,V\}$ of tangent vector fields on  $SM$, obeying the following commutator relations:
\[
	[V,X] = H,\quad [V,H]=-X,\quad [X,H]=KV. 
\]
Here $X$ is the geodesic vector field, $V$ generates the vertical rotations $(x,v)\mapsto (x,e^{it}v)$ and  $H$ is defined by the first commutator. Considering $SM$ as boundary of $DM$ there is a fourth vector field $V_\perp$ along $SM$, which is the outward pointing unit normal with respect to the Sasaki metric on $TM$. Along $SM$ the vertical $(0,1)$-bundle $T^{0,1}_v (D_xM)$ is then spanned by $V_\perp + i V$.

Given a smooth function $f\colon SM\to \C$, the $k$th Fourier mode is defined by
\[
	f_k(x,v)=\frac{1}{2\pi} \int_0^{2\pi} e^{-ikt} f(x,e^{it}v) dt,\quad (x,v)\in SM
\]
and lies in the eigenspace $\Omega_k = \{h\in C^\infty(SM): Vh=ik h\}$ ($k\in \Z$) of $-iV$. This gives a global fibrewise Fourier decomposition $f = \sum_{k\in \Z} f_k$ and we have
\[
	\deg(f)\le m\quad \Leftrightarrow\quad f_k=0 \text{ for } |k|>m.
\]
The function $f\colon SM\to \C$ is called \textit{fibrewise holomorphic} if $f_k = 0 $ for $k<0$, or equivalently, if it extends to a smooth function $F\colon DM\to \C$ whose restriction to every fibre $D_xM$ is holomorphic. The latter is taken as a definition of fibrewise holomorphicity for maps $F\colon DM\to W$ into a complex manifold.

The {\it Guillemin-Kazdhan} operators on $SM$ are defined by
\[
	\eta_\pm := \frac 12\left(X\mp i H\right)
\]
and satisfy the following commutator relations:
\[
	[V,\eta_\pm] = i \eta_\pm, \quad [\eta_+,\eta_-] = iKV/2.
\]
As a consequence
\(
	\eta_\pm\colon \Omega_{k} \to \Omega_{k\pm 1}
\)
 ($k\in \Z$).  The complex frame $\{\eta_+,\eta_-,V\}$ can be dualised to a frame $\{\tau^+,\tau^-,\psi\}$ of complex $1$-forms on $SM$, which satisfies the following structure equations:
 \begin{equation}\label{eqn_structure}
 	d\tau^\pm = \mp i \psi\wedge \tau^\pm,\quad d\psi = -\frac{iK}{2}{ \tau^+\wedge \tau^-} 
 \end{equation}

\subsection{Construction and basic properties} Given $\lambda\colon SM\to \R$  with $\deg(\lambda)\le 2$, we wish to define an involutive structure $\D\subset T_\C DM$ such that:
\begin{align}
\D\cap \bar \D &= 0 \text{ on } DM^\circ; \label{t1}\\
	\D \cap \bar \D &= \C(X+\lambda V) \text{ on } SM; \label{t2} \\
	T^{0,1}D_xM &\subset \D \text{ for all } x\in M;
	\label{t3}\\
	& \hspace{-3em} \D \text{ is orientation compatible}. \label{t4}
\end{align}
To this end, it is convenient to consider the principal  $S^1$-bundle
\[
	\mathsf{p}\colon SM\times \DD\to DM,\quad \mathsf{p}(x,v,\omega)=(x,\omega\cdot v)
\]
and define $\D$ as push-forward
\(
	\D := \mathsf{p}_*\spn_\C\left(\xi,\partial_{\bar \omega}\right),
\)
where $\xi$ is a complex vector field on $SM\times \DD$  that we still need to define. First note that the $S^1$-action $(x,v,\omega)\mapsto (x,e^{it}v,e^{-it}\omega)$ associated with $\p$ is generated by  the vector field
\[
	\V = V - i(\omega\partial_\omega - \bar \omega\partial_{\bar \omega}).
\]
We have $\mathcal L_\V \partial_{\bar \omega}=[\V,\partial_{\bar\omega}]=-i \partial_{\bar \omega}$ and hence the line bundle $\C  \partial_{\bar \omega}$ admits a push-forward along $\mathsf{p}$. This is course made up of the vertical $(0,1)$-lines $T^{0,1}D_xM$ ($x\in M$) and so condition \eqref{t3} is automatically satisfied. To ensure that also $\C \xi$ admits a push-forward and $\D$ is involutive, we impose that
\begin{equation}\label{t5}
	[\V,\xi]= i m \xi\quad  (m\in \Z)\quad \text{ and } \quad [\xi,\partial_{\bar \omega}] = 0,
\end{equation}
respectively (below, we will see that $m=-1$). Taking an oriented volume form $\mu$ on $SM\times \DD$, conditions \eqref{t1} and \eqref{t4} can then be encoded as
\begin{equation}\label{t6}
	\mu(\xi,\bar \xi,\partial_{\bar \omega},\partial_\omega, \V)>0 \quad \text{ on } |\omega|<1
\end{equation}
and finally, condition \eqref{t2} can be deduced from \eqref{t5} and  the requirement
\begin{equation}
	\xi(x,v,1) \equiv X(x,v)+\lambda(x,v)V(x,v) \quad \mod \{\partial_{\bar \omega},\V\} \text{ on } SM. \label{t7}
\end{equation}
It is now not difficult to find a suitable vector field $\xi$: 
\begin{equation}\label{def_xi}
	\xi(x,v,\omega) := \eta_- + \omega^2 \eta_+ +  i q(x,v,\omega)\partial_\omega,
\end{equation}
where $q\colon SM\times \DD\to \C$  is yet to be defined. Taking $q\equiv 0$ and $\lambda\equiv 0$, properties \eqref{t5}\eqref{t6} and \eqref{t7} can be verified by a straightforward computation---see \cite{BoPa23} for details. To extend this to thermostats, we take the fibrewise Fourier decomposition $\lambda=\sum_{|k|\le 2} \lambda_k$ and define \begin{equation}\label{defq}
	q(x,v,\omega):=\sum_{k=-2}^2 \lambda_k(x,v)\omega^{k+2}.
\end{equation}
One readily checks that $\partial_{\bar \omega}q= 0, \V q = -2i q$ and $q(x,v,1)=\lambda(x,v)$, which immediately gives the desired properties.  In summary, this leads to:
 
\begin{proposition}\label{prop_defD} Given a smooth function $\lambda\colon SM\to \R$ with $\deg(\lambda)\le 2$, and $q\colon SM\times \DD\to \C$ defined as in \eqref{defq}, the involutive structure
\[
	\D_\lambda:=\p_*\spn_\C\left(\eta_-+\omega^2\eta_++iq\partial_\omega,\partial_{\bar \omega}\right)\subset T_\C DM
\]
satisfies properties \eqref{t1},\eqref{t2},\eqref{t3} and \eqref{t4} and it is uniquely determined by them.
Moreover, if a given $\lambda\colon SM\to \R$ admits an involutive structure with these properties, then $\deg(\lambda)\le 2$.
\end{proposition}

It remains to prove uniqueness of $\D_\lambda$ and the necessity of $\deg(\lambda)\le 2$. We postpone this for the moment an refer to Remark \ref{proofuniquenessD} below for a proof.

\begin{definition}[Transport twistor space] The transport twistor space of the thermostat $(g,\lambda)$
(with $\deg(\lambda)\le 2$) is the   degenerate complex surface
\[
	Z(g,\lambda):=(DM,\D_\lambda).
\]
\end{definition}


\subsubsection{Holomorphic maps}
One of the key features of transport twistor spaces is that they allow to view fibrewise holomorphic first integrals as genuine holomorphic maps:

\begin{lemma}\label{lemholmap} Let $f\colon DM \to W$ be a smooth map into a complex manifold $W$. Then the following are equivalent:
\begin{enumerate}
	\item $f$ is fibrewise holomorphic and $df(X+\lambda V)=0$ on $SM$;
	\item $f$ is holomorphic with respect to $\D_\lambda$, that is, $df(\D_\lambda)\subset T^{0,1}W$.
\end{enumerate}
\end{lemma}

\begin{proof}
	The direction (ii) $\Rightarrow$ (i) follows directly from the properties of $\D_\lambda$. For the converse we consider the lift $h=f\circ \p\colon SM\times \DD\to W$, the bundle $E=h^*T^{1,0}W\to SM\times \DD$ and
	the projection $\pi^{1,0}\colon T_\C W\to T^{1,0}W$ . Viewing \[
		g_1 =\pi^{1,0}\circ dh(\xi)\quad \text{ and } \quad g_2=\pi^{1,0}\circ dh(\partial_{\bar \omega})
	\]
	as sections of $E$,  holomorphicity of $f$ becomes equivalent to  the vanishing of $g_1$ and $g_2$. Fibrewise holomorphicity of $f$ immediately implies that
	 $g_2\equiv 0$. On $SM$ we have the direct sum decomposition $\D_{(x,v)} = \C(X+\lambda V)_{(x,v)} \oplus T_v^{0,1}D_xM$ and together with the assumption $df(X+\lambda V)=0$ we get that $\pi^{1,0}\circ df(\D)=0$ along $SM$. 
	This implies $g_1= 0$ for $|\omega|=1$. Restricted to every fibre $\{(x,v)\}\times \DD$, the vector bundle $E$ carries a natural holomorphic structure and so we may speak of fibrewise holomorphic sections. We will finish the argument by showing that $g_1$ is fibrewise holomorphic, such that the identity theorem, applies fibrewise, results in $g_1=0$ on $|\omega|\le 1$.

Checking fibrewise holomorphicity of $g_1$ is a local matter, so we can fix $(x_0,v_0,\omega_0)\in SM\times \DD$ and consider complex coordinates $w_1,\dots, w_n$ in some chart neighbourhood of the image $h(x_0,v_0,\omega_0)\in W$. With respect to these, we have $h=(h_1,\dots, h_n)$ for functions $h_1(x,v,\omega),\dots, h_n(x,v,\omega)$, defined in a  neighbourhood of $(x_0,v_0,\omega_0)$ and holomorphic in the $\omega$-variable. Moreover,
\[
	\pi^{1,0}\circ dh = \sum_{i=1}^n dh_i \otimes \partial_{w_i}
\]
and hence $g_1(x,v,\omega)=\sum_{i=1}^n (\xi h_i(x,v,\omega)) \partial_{w_i}$. Fibrewise holomorphicity of $g_1$ now means that the coefficients $\xi h_i$ are holomorphic in $\omega$ and this follows from $[\xi,\partial_{\bar\omega}]=0$ and $\partial_{\bar \omega} \xi h_i = \xi \partial_{\bar \omega} h_i  = 0$ ($i=1,\dots, n$).
\end{proof}

Next we consider biholomorphisms between transport twistor spaces. By definition, these are diffeomorphisms of the underlying manifolds-with-boun\-dary, that are holomorphic in the interior.

\begin{lemma}[Lifts of isometries] Let $\varphi\colon (M',g')\to (M,g)$ be an orientation preserving isometry between oriented Riemannian surfaces, then its lift is a biholomorphism:
	\[
		\varphi_\sharp\colon Z(\varphi^*g,\varphi_\sharp^*\lambda)\xrightarrow{\sim} Z(g,\lambda),\quad \varphi_\sharp(y,w)=(\varphi(y),d\varphi_y(w)).
	\]
\end{lemma}	

\begin{proof}
	Let $\mathscr{E}\subset T_\C DM'$ be the involutive structure associated to $(\varphi^*g,\varphi_\sharp^*\lambda)$. Then the push-forward $\D=(\varphi_\sharp)_*\mathscr{E}$ is an involutive structure on $DM$. One readily checks that this satisfies the defining properties of $\D_\lambda$ and so holomorphicity of $\varphi_\sharp$ follows from the uniqueness part of Proposition \ref{prop_defD}.
\end{proof}

\begin{lemma}[Conformal rescaling]\label{lem_conformalrescaling}
Let $\sigma\colon M\to \R$ and $\Lambda\colon SM^\sigma=\{(x,w)\in TM: |w|_{e^{2\sigma}g} =1 \}\to \R$ be smooth functions and consider the rescaling map:
\[
		\mathrm{sc}_\sigma\colon SM^\sigma\to SM,\quad \mathrm{sc}_\sigma(x,w)=(x,e^{\sigma(x)}w).
	\]
	Then $\mathrm{sc}_\sigma$ induces an orbit equivalence between the thermostat flows associated to $(e^{2\sigma}g,\Lambda)$ and $(g,\lambda)$,  where (with $*$ denoting the Hodge star of $g$):
\[
		\lambda(x,v) = e^\sigma(x)\Lambda(x,e^{-\sigma(x)} v) - *d\sigma_x(v),\quad (x,v)\in SM.
	\]
If $\deg(\Lambda)\le 2$, then also $\deg(\lambda)\le 2$ and the extension $\mathrm{sc}_\sigma\colon DM^\sigma\to DM$ is a biholomorphism $Z(e^{2\sigma}g,\Lambda)\xrightarrow{\sim} Z(g,\lambda)$. 	
\end{lemma}

\begin{proof}
	We have to find $\lambda\colon SM\to \R$ with
	\[
		(\mathrm{sc}_\sigma)_*(X^\sigma+\Lambda V^\sigma)\in \R(X+\lambda V).
	\]
	Here we use the superscript $\sigma$ for all objects related to $e^{2\sigma}g$. The preceding display just means that a $\Lambda$-geodesics $\Gamma$ can be reparametrised to a $\lambda$-geodesic, that is, $\mathrm{sc}_\sigma$ induces an orbit equivalence. Recall that the equation for $\Gamma$ is:
	\[
		\nabla^\sigma_{\dot \Gamma}\dot \Gamma = \Lambda(\Gamma,\dot\Gamma)\Gamma^\perp.
	\]
	The Levi--Civita connection of $e^{2\sigma}g$ is given by $\nabla^\sigma_{\xi_1} \xi_2 =\nabla_{\xi_1}{\xi_2} + g(\xi_1,\xi_2)\mathrm{grad}^g\sigma - d\sigma(\xi_1)\xi_2+d\sigma(\xi_2)\xi_1$ for vector fields $\xi_1$ and $\xi_2$ on $M$. Hence
	\[
		\nabla^\sigma_{\dot \Gamma}\dot \Gamma= \nabla_{\dot \Gamma}\dot \Gamma + g(\dot \Gamma,\dot \Gamma) \cdot \mathrm{grad}^g \sigma - 2 d\sigma(\dot \Gamma)\dot\Gamma = \nabla_{\dot \Gamma} \dot \Gamma + d\sigma(\dot \Gamma^\perp)\dot\Gamma^\perp - d\sigma(\dot\Gamma)\dot \Gamma,
	\]
	which leads to the following equivalent ODE:
	\[
		\nabla_{\dot \Gamma}\dot\Gamma = \left( \Lambda(\Gamma,\dot\Gamma) \dot \Gamma^\perp  - d\sigma(\dot \Gamma^\perp) \right)\dot \Gamma^\perp+ d\sigma(\dot \Gamma)\dot \Gamma.
	\]
	If $\gamma(t)=\Gamma(\psi(t))$ is a reparametrisation to $g$-unit speed, then  $\psi'(t)=e^{\sigma(\gamma(t))}$
	 and $\psi''(t)=e^{\sigma(\gamma(t))}d\sigma(\dot \gamma(t))$ and thus a direct computation shows:
	\[
	\begin{split}
		\nabla_{\dot \gamma}\dot \gamma&= \nabla_{\dot \Gamma}\dot \Gamma(\psi(t))\cdot \psi'(t)  + \dot \Gamma(\psi(t))\cdot\psi''(t)
=\left( e^{\sigma} \Lambda(\gamma,e^{-\sigma}{\dot \gamma})  - d\sigma(\dot \gamma^\perp)\right)\dot \gamma^\perp.
	\end{split}	
	\]
	But $d\sigma_x(v^\perp)=-*d\sigma_x(v)$ and so we get the claimed expression for $\lambda$. 
	
	Clearly, $
	\deg(\lambda)\le \max(\deg(\Lambda),1)$ and so for $\deg(\Lambda)\le 2$ both thermostats come with a transport twistor space. The strategy to show that $\mathrm{sc}_\sigma$ is a biholomorphism is the same as in the preceding lemma: We let $\mathscr{E}$ be the involutive structure of $(e^{2\sigma} g,\Lambda)$ and check that $(\mathrm{sc}_\sigma)_*\mathscr{E}$ satisfies the defining properties of $\D_\lambda$---this is obvious in view of the orbit equivalence. 
\end{proof}

\begin{lemma}[Flips]
The flip $\mathsf{f}\colon Z(g,-\lambda\circ \mathsf{f})\to Z(g,\lambda),~ \mathsf{f}(x,v)=(x,-v)$ is a biholomorphism for every $\lambda\in C^\infty(SM,\R)$ with $\deg(\lambda)\le 2$.
\end{lemma}

\begin{proof}
On $SM$ we have $\mathsf{f}_*(X+\lambda V) = -X + (\lambda \circ \mathsf{f}) V$ and since $\mathsf{f}$ is fibrewise holomorphic, the proof is completed as in the preceding two lemmas.
\end{proof}

Recall the action of the group $\Aut(M)\times \Z_2$ defined below \eqref{defaction}.

\begin{proposition}\label{prop_biholo} \,
\begin{enumerate}	
	\item\label{prop_biholo1} Suppose that $\lambda \in C^\infty(SM,\R)$ with $\deg(\lambda)\le 2$ and $\varphi\in \Aut(M)$ with $\varphi^*g=e^{2u}g$. Then the following maps are biholomorphisms: \[
\varphi_\sharp\colon Z(g,\lambda\triangleleft \varphi)\to Z(g,\lambda),\quad \mathsf{f}\colon Z(g,-\lambda\circ \mathsf{f}) \to Z(g,\lambda).
\]
Here $\varphi_\sharp(x,v)=(\varphi(x),e^{-u(x)}d\varphi_x(v))$  and $\mathsf{f}(x,v)=(x,-v)$.

\item\label{prop_biholo2} Let $\lambda_1,\lambda_2\in C^\infty(SM,\R)$ with $\deg(\lambda_k)\le 2$ ($k=1,2$) and suppose that
\[
	\Phi\colon Z(g,\lambda_1)\to Z(g,\lambda_2)
\]
 is a biholomorphism that sends fibres into fibres. Then there exists $\varphi\in \Aut(M)$ such that either $\Phi=\varphi_\sharp$ and $\lambda_1 = \lambda_2\triangleleft \varphi$ or 
 $\Phi=\varphi_\sharp\circ \mathsf{f}$ and $\lambda_1 =  - (\lambda_2\triangleleft \varphi)\circ \mathsf{f}$.
\end{enumerate}

\end{proposition}

\begin{proof} 
For part \ref{prop_biholo1} we  write $\Lambda(x,v)=\lambda(\varphi(x),d\varphi_x(v))$ and obtain the commutative diagram
\[
	\begin{tikzcd}
		Z(e^{2u}g,\Lambda)\arrow["d\varphi"]{r} \arrow["\mathrm{sc}_u"]{d}& Z(g,\lambda)\\
	Z(g,\lambda\triangleleft \varphi)\arrow[swap,"\varphi_\sharp"]{ur} 
	\end{tikzcd},
\]
where  $\lambda\triangleleft \varphi$ is defined as in \eqref{defaction}. The preceding lemmas thus imply that $\varphi_\sharp$ and $\varphi_\sharp\circ \mathsf{f}$ are biholomorphisms as claimed. 

For part  \ref{prop_biholo2} suppose that $\Phi\colon Z(g,\lambda_1)\to Z(g,\lambda_2)$ is a biholomorphism with $\pi\circ \Phi=\varphi\circ \pi$ 
for some  diffeomorphism $\varphi\colon  M\to M$.  Then $(\Phi|_{SM})_*V\in \R V$ and $(\Phi_*)|_{SM}(X+\lambda_1 V) \in \R(X+\lambda_2 V)$, which means that  $(\Phi|_{SM})_*$ leaves the $2$-plane  $\spn_\R(X,V)=\ker {\beta}$ invariant. Here $\beta$ is the $1$-form defined by
\[
	\beta_{(x,v)}(\vartheta) = \langle v^\perp, d\pi_{(x,v)}\vartheta \rangle_{g(x)},\quad (x,v)\in SM,\vartheta\in T_{(x,v)}SM
\]
and the invariance can be expressed as $(\Phi|_{SM})^*\beta=a\beta$ for a function $a\colon SM\to \R\backslash 0$.
	Given $(x,v)\in SM, \vartheta\in T_{(x,v)}SM$ we set $(y,w)=\Phi(x,v)$ and then have
	\[
	\begin{split}
	\langle w^\perp, d\varphi_x (d\pi_{(x,v)}(\vartheta))\rangle_{g(y)} = 	(\Phi|_{SM})^*\beta_{(x,v)}(\vartheta)= a(x) \beta_{(x,v)}(\vartheta) \\= a(x) \langle v^\perp, d\pi_{(x,v)}(\vartheta)\rangle_{g(x)}.
	\end{split}
	\]
	Write $d\varphi_x^\top \colon T_yM\to T_xM$ for the adjoint with respect to $g$, then we must have
	\[
		d\varphi_x^\top(w^\perp) = a(x) v^\perp.
	\]
	Writing $A(x)v=-\big((d\varphi_x^{-\top}) v^\perp\big)^\perp$, this gives $w = a(x) A(x)v$ and taking norms on both sides, we see that $|a(x)|=1/|A(x)v|_{g(y)}$. Hence
	\[
		\Phi(x,v) = \left(\varphi(x),\pm \frac{A(x) v}{|A(x)v|_g} \right),\quad (x,v)\in SM.
	\]
	Since  $v\mapsto A(x)v/|A(x) v|$ has a holomorphic extension $D_xM \to D_{\varphi(x)}M$, the operator $A(x)\colon T_xM\to T_{\varphi(x)}M$ must be $\C$-linear \cite[Lemma 3.3]{BMP25}. Hence also $d\varphi_x^{-\top}$ and thus $d\varphi_x$ are $\C$-linear, which means that $\varphi$ is a biholomorphism. It follows that 
 $\varphi^*g=e^{2u} g$ for some $u\colon M\to \R$ and $A(x) = d\varphi_x^{-\top} = e^{-2u(x)} d\varphi_x$. We conclude that either $\Phi=\varphi_\sharp$ or $\Phi=\varphi_\sharp \circ \mathsf{f}$ on $SM$. The identity theorem, applied fibrewise, shows that equality holds on all of $DM$. If $\Phi=\varphi_\sharp$, then combined with the first part of the proposition we see that $\Id \colon Z(g,\lambda_2 \triangleleft \varphi) \to Z(g,\lambda_1)$ is a biholomorphism and thus $\lambda_1 = \lambda_2\triangleleft \varphi$. If $\Phi=\varphi_\sharp\circ \mathsf{f}$ one argues similarly. 
\end{proof}

\subsubsection{Isothermal coordinates}\label{sec_isocord}
Suppose that over some open set $U\subset M$ we have a complex coordinate $z=x+iy$. Then on $U$ we have $g=e^{2\sigma(z)} |dz|^2$ for some smooth function $\sigma\colon U\to \R$, that is, $z$ is an isothermal coordinate. The unit vector field $\mathbf{1}(z)\equiv  e^{-\sigma(z)}\partial_x$ locally triviales $SM$ and we thus obtain a local coordinate system for $DM$:
\[
	U\times \DD\xrightarrow{\sim} DM|_U,\quad (z,\mu)\mapsto (z,\mu\cdot \mathbf{1}(z)).
\]
We also view $U$ as a subset of $\C$  and refer to $(z,\mu)$ as {\it isothermal coordinates} for $SM$ and $DM$, respectively. In these coordinates the frame $\{X,V,H\}$ takes the following form (cf.~\cite[Lemma 3.5.6]{PSU23}):
\begin{eqnarray}
	X &=& e^{-\sigma} \big[\mu \partial_z + \bar \mu \partial_{\bar z}   - (\mu \sigma_z - \bar \mu \sigma_{\bar z})(\mu \partial_\mu -\bar \mu \partial_{\bar \mu})\big]\\
	H &=& ie^{-\sigma} \big[\mu \partial_z - \bar \mu \partial_{\bar z}   - (\mu \sigma_z + \bar \mu \sigma_{\bar z})(\mu \partial_\mu -\bar \mu \partial_{\bar \mu})\big]\\
	V & =& i (\mu \partial_\mu - \bar \mu \partial_{\bar\mu})
\end{eqnarray}
Given $\lambda\colon SM\to \R$, we also  write
 $\lambda(z,\mu)\equiv \lambda(z,\mu\cdot \mathbf{1}(z))$  in isothermal coordinates and derive the following expression for the generator of the associated thermostat flow:
 \begin{equation}\label{Fcoordinates}
 	X+\lambda V= e^{-\sigma}\big[\mu \partial_z + \bar \mu \partial_{\bar z}   + (-\mu \sigma_z + \bar \mu \sigma_{\bar z} + ie^\sigma \lambda)(\mu \partial_\mu -\bar \mu \partial_{\bar \mu})\big].
 \end{equation}

\begin{lemma}\label{lem_coordinateD}
	Suppose that $\deg(\lambda)\le 2$ and let $f_2\colon U\times \DD\to \C$ be the unique $\mu$-holomorphic function with
	\[
		f_2(z,\mu) = (-\mu \sigma_z + \bar \mu \sigma_{\bar z} + ie^\sigma \lambda)\mu^2 ,\quad |\mu|=1.
	\]
	Then 
	\(
		\D_\lambda = \spn_\C(\partial_{\bar z} + \mu^2 \partial_z + f_2(z,\mu) \partial_\mu, \partial_{\bar \mu})
	\) in the $(z,\mu)$-coordinates for $DM$.
\end{lemma} 

\begin{proof}

Using \eqref{Fcoordinates} it is straightforward to  checks that the proposed involutive structure satisfies the defining properties of $\D_\lambda$.
\end{proof}

\begin{remark}\label{rmk_Xi} 
If $\deg(\lambda)\le 1$, that is, $\lambda(z,\mu)=\lambda_{-1}(z) \mu^{-1} + \lambda_0(z)\mu + \lambda_1(z) \mu$, then we can further write  $X+\lambda V = \bar \mu \Xi$, where
\[
	\Xi= e^{-\sigma}\Big[\partial_{\bar z} + \mu^2\partial_z +(ie^\sigma\lambda_{-1}+\sigma_{\bar z} + ie^\sigma \lambda_0 \mu -\sigma_z\mu^2 + ie^\sigma\lambda_{-1})(\mu\partial_\mu - \bar \mu \partial_\mu)\Big].
\]
\end{remark}

\subsubsection{Holomorphic forms}
The bundle $\Lambda^{p,0}_\lambda \to DM$ ($p=1,2$) is defined by
\begin{equation}\label{10bundle}
	\begin{split}
	 (\Lambda^{p,0}_\lambda)_{(x,v)} &:= \{\omega \in \Lambda^p_\C DM_{(x,v)}: i_\xi\omega =0 \text{ for } \xi\in (\D_\lambda)_{(x,v)}\}.
	\end{split}
\end{equation}
Here $\Lambda^p_\C$ stands for the bundle of $\C$-valued $p$-forms.
Sections of $\Lambda_\lambda^{p,0}$ are referred to as $(p,0)$-forms. (One may also define bundles of $(p,q)$-forms, e.g.~by $\Lambda_\lambda^{0,1}=\Lambda^1_\C/\Lambda_\lambda^{1,0}$, see \cite{Tre92} --- here these will not be needed.)

\begin{lemma}\label{lemdomega}
	There exists a unique $(2,0)$-form $\Omega$ on $Z(g,\lambda)\backslash 0$ with the following properties:
	\begin{enumerate}
		\item $\Omega(H,V)=1$ on $SM$
		\item $\Omega$ is fibrewise holomorphic, that is, $i_\zeta d\Omega = 0$ for all $\zeta\in T^{0,1}D_xM$.
	\end{enumerate}
	Moreover, its lift $\hat \Omega = \p^* \Omega$ to $SM\times (\DD\backslash 0)$ satisfies:
	\begin{equation}\label{domega}
		\iota_\xi d\hat \Omega = r \hat \Omega,\quad r(x,v,\omega)=\sum_{|k|\le 2} ik\omega^{k+1} \lambda_k.
	\end{equation}
\end{lemma}

\begin{proof}
	We first construct a frame for $\p^*\Lambda_\lambda^{1,0}\subset \Lambda^1_\C(SM\times \DD)$: this should consist of $1$-forms $\gamma_1,\gamma_2$ that annihilate $\xi,\partial_{\bar \omega}$ and $\V$. Using the frame $\{\tau^+,\eta^-,\psi\}\subset \Lambda^1_\C SM$, it is not hard to make the following guess:
	\[
		\gamma_1 = \tau^+ - \omega^2 \tau^-,\quad \gamma_2 = d\omega + i \omega \psi - i q \tau^-.
	\] 
	Here $q$ is defined as in \eqref{defq}. It  is straightforward to verify $i_\zeta \gamma_k = 0$ for $\zeta \in \{\xi,\partial_{\bar \omega}, \V\}$ as well as
	\(
		  \mathcal L_V(\gamma_k)=-i \gamma_k.
	\)
	The last property does not immediately allow for a push-forward along $\p$, but this can be remedied by introducing the factor $1/\omega$. This leads to the following definition on $\omega\neq 0$:
	\begin{equation}\label{defhatom}
		\hat \Omega  := -\frac{1}{2}\frac{ \gamma_1\wedge \gamma_2}{\omega^2}
	\end{equation}
	Clearly, $i_{\zeta} \hat \Omega = 0$ for $\zeta\in \{\xi,\partial_{\bar \omega},\V\}$ and $\mathcal L_\V\hat \Omega = 0$, such that $\Omega:=\p_*\hat \Omega$ is a well-defined $(2,0)$-form away from the zero section. On $SM$ we have $H=i(\eta_+-\eta_-)$ and thus
	\[
		\Omega(H,V) = -\frac{1}{2} \det\left.\begin{bmatrix}
		\langle \gamma_1,H\rangle & \langle \gamma_1,V\rangle\\
		\langle \gamma_2,H\rangle & \langle \gamma_2,V\rangle
		\end{bmatrix} \right|_{\omega=1} = -\frac{1}{2} \det\begin{bmatrix}
		2i & 0\\
		\lambda & i
		\end{bmatrix} = 1.
	\]
	Let us rewrite $\hat \Omega$ using the structure equations \eqref{eqn_structure}:
	\[
	\begin{split}
		\hat \Omega &= -\frac 12\left(\frac{1}{\omega} \tau^+ -\omega\tau^-\right)\wedge\left(\frac{d\omega}{\omega} + i\psi - i\frac{q}{\omega} \tau^-\right)
		\\
		&= -\frac{1}{2} \left(-\frac{d\omega}{\omega^2} \wedge\tau^+ - \frac{i}{\omega} \psi\wedge \tau^+ + d\omega \wedge \tau^- + i\omega \psi\wedge \tau^-- \frac{iq}{\omega^2} \tau^+\wedge \tau^-\right) \\
		& = -\frac{1}{2} \left(-\frac{d\omega}{\omega^2} \wedge\tau^+ + \frac{1}{\omega} d \tau^+ + d\omega \wedge \tau^- + \omega d\tau^-- \frac{iq}{\omega^2} \tau^+\wedge \tau^-\right)\\
		&= -\frac{1}{2} d\left(\frac{1}{\omega} \tau^+ + \omega\tau^-\right) + \frac{iq}{2\omega^2} \tau^+\wedge \tau^-
	\end{split}	
	\] 
	The form $i\tau^+\wedge \tau^-/2 = (\pi\circ \p)^*d\mathrm{vol}_g$ is closed and hence:
	\[
		d\hat \Omega =\sum_{|k|\le 2} ik\omega^k \lambda_k\cdot \left(\psi-i \frac{d\omega}{\omega}\right)\wedge \left(\frac{i}{2} \tau^+\wedge \tau^-\right) = \frac{r}{2\omega^2} (d\omega + i\psi)\wedge \tau^+\wedge \tau^-
	\]
	Since the contraction with $\partial_{\bar \omega}$ is zero, we conclude that $\Omega$ is fibrewise holomorphic, as claimed. For \eqref{domega} we observe that
	\[
		i_\xi(d\omega + i\psi) = iq\quad \text{ and } \quad i_\xi(\tau^+\wedge \tau^-) = -\tau^+ +\omega^2\tau^-,
	\]
	such that
	\[
	i_\xi d\hat \Omega = \frac{r}{2\omega^2} \Big(iq \tau^+\wedge \tau^- + (d\omega +i\psi)\wedge(\tau^+-\omega^2\tau^-) \Big) = \frac{r}{2\omega^2} \gamma_2\wedge \gamma_1 = r\hat \Omega.
	\]
	It remains to prove uniqueness: Any other such $(2,0)$-form must take the form $f\Omega$ for a function $f\colon DM\backslash 0 \to \C$  that satisfies $f\equiv 1$ on $SM$ and is fibrewise holomorphic. By the identity principle, applied on each punctured disk $D_xM\backslash 0$ separately, we conclude that $f\equiv 1$ on $DM\backslash 0$.
\end{proof}

\begin{lemma}\label{lem_totallyreal} Define  
the vector bundle
\(
	\Lambda^{1,0}_{\lambda,\R}\to SM\) by \(\Lambda^{1,0}_{\lambda,\R}:=\{\omega \in \Lambda^{1,0}_\lambda: \omega \text{ real valued on } TSM\}.
\)
Then:
	\begin{enumerate}
		\item $\Lambda^{1,0}_{\lambda,\R} = \{\omega \in \Lambda^1SM: \omega(X+\lambda V)=0\}$ over $SM$;
		\item $\Lambda^{1,0}_{\lambda}=\Lambda^{1,0}_{\lambda,\R} \oplus i\Lambda^{1,0}_{\lambda,\R}$ over $SM$ --- in particular, $(\Lambda^{1,0}_\lambda,\Lambda^{1,0}_{\lambda,\R})$ is a bundle pair over $(DM,SM)$;
		\item the Maslov index of $(\Lambda_\lambda^{1,0},\Lambda^{1,0}_{\lambda,\R})$ along each disk $D_xM$ equals $-4$.
	\end{enumerate}
\end{lemma}

\begin{proof}
	The bundle $TDM|_{SM}$ is framed by $\{X+\lambda V, H, V, V_\perp\}$ and $\D_\lambda|_{SM}$ is framed by $X + \lambda V$ and $V_\perp+iV$. Hence for $\omega\in \Lambda^1_\C DM|_{SM}$ it holds that:
	\[	
\omega\in \Lambda^{1,0}_\lambda\quad \Leftrightarrow \quad	\omega(X+\lambda V)=0~ \text{ and }~ 	\omega(V_\perp) = -i \omega(V)
	\]
	Now the inclusion `$\subset$' in (i) is immediate, and for `$\supset$' we extend a given $\omega\colon TSM\to \R$ with $\omega(X+\lambda V)=0$ to an element in $\Lambda^{1,0}_{\lambda, \R}$ by defining the action on $V_\perp$ as in the preceding display. The decomposition in (ii) is achieved by taking the  real and imaginary parts of $\omega(H)$ and $\omega(V)$ and noting that $\omega$ is completely determined by the action on $H$ and $V$.
For the Maslov index we may equivalently consider \[(\p^*\Lambda^{2,0}_\lambda,\p^*\Lambda^{2,0}_{\lambda,\R})\] along a disk $\{(x,v)\}\times \DD$ (here $\Lambda^{2,0}_{\lambda,\R}$ contains the $(2,0)$-forms with real values on $TSM$). We check characterisation \eqref{maslovchar} using the the $2$-form $\hat \Omega$ from  \eqref{defhatom}: 
 $\omega^2 \hat \Omega$ is a nowhere vanishing section and $\omega^{-2} (\omega^2\hat \Omega)$ lies in $\Lambda^{2,0}_{\lambda,\R}$ along $|\omega|=1$, hence the Maslov index must be $-4$.
\end{proof}

\begin{lemma}
Let $u\colon DM\to \C$ be fibrewise holomorphic. Then the following are equivalent:
\begin{enumerate}
	\item $(X+\lambda V) u  = V\lambda$ on $SM$
	\item $\Upsilon = e^{-u}\Omega$ is holomorphic (that is, $i_\zeta d\Upsilon =0$ for all $\zeta\in \D_\lambda$).
\end{enumerate}
\end{lemma}

\begin{proof}
Writing $f=e^{-u}$ and $h=\p^*f$, holomorphicity of $\Upsilon$ is equivalent to 
\[
	0 = i_\xi d(h \hat \Omega) = (\xi h + rh) \hat \Omega  \text{ on } \omega\neq 0,
\] 
where $\hat \Omega$ and $r$ are as in Lemma \ref{lemdomega}. For $\omega=1$ we have
\[
	\xi h + rh = (X+\lambda V) f + (V\lambda) f =\left[ -(X+\lambda V) u + V\lambda\right]e^{-u}
\]
and this shows that (ii) $\Rightarrow$ (i). For the reverse direction, we may start with the assumption that $\xi h + rh = 0$ for $\omega = 1$. Using \eqref{t5} one checks that $(\V+i)(\xi h + rh) = 0$ and hence $\xi h + rh$ vanishes on $|\omega|=1$. Finally, $ \partial_{\bar \omega}(\xi h + rh)=0$ and the identity principle, applied fibrewise, yields $\xi h + rh\equiv 0$ on  $0<|\omega|\le 1$.

\end{proof}

\begin{lemma}\label{lem_confmagZ}
	Let $\deg(\lambda)\le 2$. Then the following are equivalent:
	\begin{enumerate}
		\item $\lambda$ is conformally magnetic; 
		\item there is holomorphic $(2,0)$-form $\Upsilon$ on $Z(g,\lambda)\backslash 0$ such that $\Upsilon(H,V)>0$ on $SM$ and $|v|^2\Upsilon$ is bounded.
	\end{enumerate}
\end{lemma}

\begin{proof}
Recall that $\lambda$ is called conformally magnetic, if it is of the form $\lambda(x,v) = \lambda_0(x) -*d\sigma_x(v)$ for real functions $\lambda_0,\sigma\colon M\to \R$. Equivalently,
\[
	\lambda_{-2} = 0 \quad \text{ and } \quad \lambda_{-1}=-i\eta_-\sigma
\]
for some real valued $\sigma\colon M\to \R$, viewed as element in $\Omega_0$. This follows from the identity  $i(\eta_+-\eta_-)=H = [V,X] = VX = -*d$ on $\Omega_0$. If this is the case, then $u=-\sigma$ solves $(X+\lambda V)u=V\lambda$ and by the preceding lemmas $\Upsilon =e^{-u}\Omega$ satisfies the requirements of (ii).

For the converse, assume that $\Upsilon$ as in (ii) exists. Then we must have $\Upsilon = f\Omega$ for a fibrewise holomorphic map $f\colon DM\to \R$ with real and positive values on $SM$. This implies that $f(x,v)=\exp(-\sigma(x))$ for some $\sigma\colon M\to \R$. Reversing the steps in the first paragraph, we see that $\lambda$ is conformally magnetic.
\end{proof}

\subsection{Local normal form} For an open set $U\subset \DD$ we consider involutive structures on $U\times \DD$ that have the following {\it normal form} in coordinates $(z,\mu)$:
\[
\begin{split}
	\D_f := \spn_\C\left(\partial_{\bar z} + f_1 \partial_z+ f_2\partial_\mu, \partial_{\bar \mu}\right) = \ker\left(dz- f_1 d\bar z\right)\cap \ker\left(d\mu- f_2 d\bar z\right)
	\end{split}.
\]	
Here we assume that $f_1,f_2\colon U\times \DD\to \C$ are smooth, holomorphic in the $\mu$-variable and satisfy
\(
|f_1(z,\mu)| <1 \text{ on } |\mu|<1\) and \(|f_1(z,\mu)| =1 \text{ on } |\mu|=1.
\)
The main example is of course the description in isothermal coordinates given above, but it is worthwhile to investigate the normal form in this generality.

\begin{lemma}[Normal form]\label{lem_normalformD}
Let  $\D\subset T_\C(U\times \DD)$ be an involutive structure with $\partial_{\bar \mu}\in \D$, \[\dim_\C\D\cap \bar \D = 0 \text{ on } |\mu|<1\quad \text{ and } \quad \dim_\C\D\cap \bar \D =1 \text{ on } |\mu|=1\]
and such that $\D$ is orientation compatible on $|\mu|<1$. Then $\D$ can be put into normal form $\D=\D_f$ for a suitable function $f\colon U\times \DD\to \C^2$ as above. Moreover, any $\D_f$ satisfies the just stated properties. 
\end{lemma}

\begin{proof}
We may complement $\partial_{\bar \mu}$ to a frame of $\D$ by a vector field of the form $g_0\partial_{z} +g_1\partial_{\bar z} + g_2 \partial_\mu$, where $g\colon U\times \DD\to \C^3$ is smooth. 
Since $\dim_\C \D\cap \bar \D\le 1$ everywhere, the functions $g_0$ and $g_1$ cannot vanish simultaneously, and moreover  we claim that 
\begin{equation}\label{normalform1}
	|g_0|>|g_1| \text{ on } |\mu|<0\quad \text{ and } \quad |g_0|=|g_1| \neq 0 \text{ on } |\mu|=1.
\end{equation}
Indeed, on $\{g_0\neq 0\}$ the $(1,0)$-forms are spanned by $\omega_1= g_0 dz - g_1d\bar z$ and $\omega_2 = g_0d\mu - g_2d\bar z$, which satisfy
\[
	\omega_1\wedge  \omega_2 \wedge \bar \omega_1\wedge \bar \omega_2 = 4 \left(|g_0|^2-|g_1|^2\right) \frac{d\bar z\wedge dz}{2i} \wedge \frac{d\bar \mu\wedge d\mu}{2i}.
\]
An analogous computation on  $\{g_1\neq 0\}$ results in the same $4$-form and hence the properties in \eqref{normalform1} follows from the orientation compatibility and the degeneracy of $\D$, respectively. Defining $f_1=g_1/g_0$ and $f_2=g_2/g_0$, we have put $\D$ into the correct normal form. Since $\D$ is involutive, we must have
\[
	0 = i_{\partial_{\bar \mu}}d(dz - f_1 d\bar z) = - \partial_{\bar \mu}f_1 d\bar z ,\quad 0 = i_{\partial_{\bar \mu}}d(d\mu  - f_2 d\bar z) = - \partial_\mu f_2 d\bar z
\]
and hence $f_1$ and $f_2$ are holomorphic in the $\mu$-variable. Reversing the above arguments for a given $f$, it is not hard to see that $\D_f$ is an involutive structure with the stated properties.
\end{proof}

\begin{lemma}[Identity theorem for involutive structure]\label{lem_identity}
	Let $\D_f, \D_{g} \subset T_\C(U\times \DD)$ be two involutive structures in normal form. Then 
	\begin{equation*}
		\D_f\cap \bar \D_f = \D_g\cap \bar \D_g \text{ on } |\mu|=1\quad \Rightarrow \quad f=g.
	\end{equation*}
\end{lemma}

\begin{proof}	
	Fix $(z,\mu)$ with $|\mu|=1$. Then $g_1(z,\mu)=\zeta^2$ for some $\zeta\in S^1$ and hence $\D_g\cap \bar \D_g$ is spanned by the real part of
	\(
		\xi = \zeta \partial_z + \bar \zeta g_2 \partial_\mu
	\). This also lies in $\D_f$ iff the following two equations are satisfied:
	\[
		0 = \langle dz - f_1d\bar z, \xi+\bar \xi\rangle = \zeta - f_1\bar \zeta		 \quad \text{ and }\quad 0 = \langle d\mu - f_2d\bar z, \xi+\bar \xi\rangle  = \bar \zeta ( g_2-f_2).
	\]
Hence $f=g$ on $|\mu|=1$ and since both functions are $\mu$-holomorphic, they must agree everywhere.	
\end{proof}

\begin{remark}\label{proofuniquenessD}
	We can now finish the following:
	
	\begin{proof}[Proof of Proposition \ref{prop_defD}] Assume that $\lambda\colon SM\to \R$ is smooth with $\deg(\lambda)\le 2$. Then in isothermal coordinates $(z,\mu)$ we may put $\D_\lambda$ into normal form $\D_f$ (using Lemma \ref{lem_normalformD} instead of Lemma \ref{lem_coordinateD} to avoid circular reasoning!) and use  the just established identity principle to conclude uniqueness. Next, suppose we only know of $\lambda\colon SM\to \R$ that it admits an involutive structure $\D$ as in Proposition \ref{prop_defD} and we wish to conclude that $\deg(\lambda)\le 2$. In isothermal coordinates we may again put it into normal form $\D=\D_f$ and we claim that for $|\mu|=1$ we necessarily have:
	\[
		f_1(z,\mu)=\mu^2\quad \text{ and }\quad f_2(z,\mu):=ie^{\sigma(z)}\lambda\big(z,\mu\cdot \mathbf{1}(z)\big)\mu^2  + \sigma_{\bar z}(z)\mu  - \sigma_z(z) \mu^3.
	\]	
	To see this, we use  the supposed property \eqref{t2} and Lemma \ref{lem_coordinateD} to conclude that the vector field $\Xi$ (defined below this lemma) necessarily lies in $\D_f$. Since the identity principle in Lemma \ref{lem_identity} was proved pointwise for $|\mu|=1$, we can derive the expression in the preceding display. But $\D_f$ being in normal form implies  that $f_2(z,\mu)$ has a $\mu$-holomorphic extension from $|\mu|=1$ to $|\mu|\le 1$ and this shows that the Fourier modes  $S^1\ni \mu\mapsto \lambda(z,\mu\cdot \mathbf{1}(z))$ vanish in degree $\le -3$. Since $\lambda$ is real valued, this enforces $\deg(\lambda)\le 2$, as desired.
	\end{proof}
\end{remark}

As in \eqref{10bundle}, we define  $\Lambda^{1,0}_f\to U\times \DD$ as the annihilator of $\D_f$.

\begin{lemma}[Maslov index]\label{lem_maslovD} Let $\D_f\subset T_\C(U\times \DD)$ be an involutive structure in normal form. Then $\D_f\cap \bar \D_f$ is tangent to $|\mu|=1$ if and only if
\[
	\Lambda^{1,0}_{f,\R} =\{\omega\in \Lambda^{1,0}_f: \omega \text{ real valued on } T(U\times \partial \DD) \}\subset \Lambda^{1,0}_f
\]
defines a totally real subbundle over $|\mu|=1$. In this case the Maslov index of the bundle pair $(\Lambda^{1,0}_f,\Lambda^{1,0}_{\R,f})$ over the disk $\{z\}\times \DD$ ($z\in U$) equals
\[
	-  \deg(f_1(z,\cdot):S^1\to S^1) - 2.
\]	
\end{lemma}	

\begin{remark}\label{rem_maslovD}
	The proof also shows that if $\ell = 2$ and $f_1(z,\mu)=\mu^2 h(z,\mu)^2$, then total reality of $\Lambda^{1,0}_f$ is further equivalent to $\Re(\bar h(z,\mu) \bar \mu^2 f_2) = 0$ for $|\mu|=1$.
\end{remark}

\begin{proof}[Proof of Lemma \ref{lem_maslovD}]
	In polar coordinates $\mu = r e^{i\theta}$ we have $\mu \partial_\mu = \partial_r - i\partial_\theta$. At a given point $(z,\mu)$ with $|\mu|=1$ we can the write $\D_f = \spn_\C(\Xi, \partial_r+i\partial_\theta)$, where $\Xi$ is real. We may write $\Xi=\Xi_{||}+a\partial_r$, where $\Xi_{||}$ is tangent to $|\mu|=1$ and $a\in \R$.  Assume that $\Lambda^{1,0}_{f,\R}$ is totally real at $(z,\mu)$ and $a\neq 0$. Then any $\omega \in \Lambda^{0,1}_f$ can be decomposed as $\omega'+i\omega''$, where $\omega',\omega''\in \Lambda^{1,0}_{f,\R}$. But taking imaginary parts in \[
	0 = \omega'(\Xi) = \omega'(\Xi_{||}) + a \omega'(\partial_r) = \omega'(\Xi_{||}) - ia \omega'(\partial_\theta)
	\] 
	implies that $\omega'(\partial_\theta)=0$, and analogously with $\omega''$. This implies that $\partial_r \in \D_f$ at $(z,\mu)$, which is incompatible with $\dim_\C(\D_f\cap \bar \D_f)=1$. This shows:
	\[
		\Lambda^{1,0}_{f,\R}\subset \Lambda^{1,0}_f|_{|\mu|=1} \text{ totally real } \quad \Rightarrow \quad \Xi \text{ tangent to } U\times \partial \DD. 
	\]
	The reverse implication is proved exactly like in Lemma \ref{lem_totallyreal}. To compute the Maslov index, we write $f_1(z,\mu)=\mu^\ell h(z,\mu)^2$, where $\ell\in \Z$ is the degree of $f_1(z,\cdot)\colon S^1\to S^1$ and $h\colon U\times \DD\to \C\backslash \{0\}$ is nowhere vanishing and satisfies $|h|=1$ on $|\mu|=1$. A vector $\Xi$ as above is then given by
	\[
		\Xi=\bar h e^{-i\ell\theta/2}(\partial_{\bar z} +f_1 \partial_z + f_2 \partial_\mu) + h e^{i\ell\theta/2} \bar f_2 (\partial_{\bar \mu}) = \Xi_{||} + 2 \Re( \bar h e^{-i(\ell+2)\theta/2} f_2) \partial_r 
	\]
	and total reality becomes equivalent to \[\Re( \bar h e^{-i(\ell+2)\theta/2} f_2) = 0.\] 
	The Maslov index of $(\Lambda^{1,0}_f,\Lambda^{1,0}_{f,\R})$ is the same as that of the exterior power $(\Lambda^{2,0}_f,\Lambda^{2,0}_{f,\R})$ and can be characterised as the unique integer $k\in \Z$ such that $\Lambda^{2,0}_f$ has a smooth global frame $\Omega$ with
	\[
	\Omega(z,e^{i\theta})\cdot e^{ik\theta/2} \in \Lambda^{2,0}_{f,\R} \quad (z\in U,\theta\in \R).
	\]
Taking \[
\begin{split}
	\Omega &= \bar h(z,\mu) (dz-f_1d\bar z)\wedge (d\mu-f_2d\bar z)\\[.2em]
	& = e^{i(\ell+2)\theta/2}\left( ie^{-i\ell\theta/2}\bar h\, dz\wedge d\theta - i e^{i\ell\theta/2} h \,d\bar z\wedge d\theta +e^{-i(\ell+2)\theta/2} \bar h f_2 \,{d\bar z}\wedge {dz} \right),
\end{split}
\]
we see the characterising property is satisfied for $k=-(\ell+2)$, since the expression in the last bracket is real.
\end{proof}

\begin{lemma}[Behaviour of zero section]\label{lem_zeroD} Let $\D_f\subset T_\C(U\times \DD)$ be an involutive structure in normal form.  Then:
\begin{eqnarray*}
f_2\equiv 0 \text{ on } \mu=0 ~ &\Leftrightarrow ~ &\{\mu=0\} \text{ is a complex submanifold};\\[.3em]
f_1=f_2\equiv 0 \text{ on } \mu=0 ~ &\Leftrightarrow ~ &z\mapsto (z,0) \text{ is holomorphic}; \\[.3em]
\partial_\mu f_1=f_2\equiv 0 \text{ on } \mu=0 ~ &\Leftrightarrow ~ &
\left\{\begin{array}{l}
\{\mu=0\} \text{ is a complex submanifold and}\\[.3em]
 \C \partial_\mu\subset \bar \D_f|_{\mu=0} \text{ is a holomorphic subbundle.}
\end{array}\right.
\end{eqnarray*}
\end{lemma}

\begin{proof}
	The submanifold $\{\mu=0\}$ is complex iff the two spanning $(1,0)$-forms $\omega_1=dz - f_1 d\bar z$ and $\omega_2 = d\mu -f_2 d\bar z$ pull-back to linearly dependent forms along the embedding $\iota(z)=(z,0)$. So the first equivalence follows from $\iota^*(\omega_1\wedge \omega_2)=f_2(z,0) d\bar z \wedge dz$. The map $\iota$ is holomorphic, iff both $\iota^*\omega_1$ and $\iota^*\omega_2$ do not have any $d\bar z$-component and this yields the second equivalence. For the last equivalence we assume that $f_2(z,0)\equiv 0$ and compute  for $\mu=0$ that
	\[
		i_{\partial_{\bar z}} d\omega_1 = \partial_\mu f_1(z,0) \omega_2 \quad \text{ and } \quad i_{\partial_{\bar z}} d\omega_2 = \partial_\mu f_2(z,0) \omega_2.
	\]
	Hence, the $\bar \partial$-operator of the holomorphic vector bundle $\Lambda^{1,0}|_{\mu=0}$ acts 
	 by
	\[
		(g_1\omega_1 + g_2\omega_2) \mapsto (\partial_{\bar z} g_1) d\bar z\otimes \omega_1 + (\partial_{\bar z} g_2 + g_1\partial_\mu f_1|_{\mu=0} + g_2 \partial_\mu f_2|_{\mu=0}) d\bar z\otimes \omega_2,
	\]
	where $g_1,g_2\colon U\to \C$ are smooth functions. The line bundle $\C \partial_\mu$ is holomorphic iff it equals the kernel of a holomorphic section $\omega=g_1\omega_1+g_2\omega_2$ of $\Lambda^{1,0}|_{\mu=0}$. Suppose such a section exists, then $i_{\partial_\mu}\omega = 0$ and hence $g_2\equiv 0$. Holomorphicity  of $\omega$ implies that $g_1 \partial_\mu f_1(z,0)\equiv 0$ and since $g_1$ must be nowhere vanishing for $\ker \omega$ to be $1$-dimensional, we conclude $\partial_\mu f_1(z,0)\equiv 0$. Vice versa, if $\partial_\mu f_1(z,0)\equiv 0$, then $\omega_1$ is a holomorphic section that cuts out $\C \partial_\mu$.
\end{proof}

 \subsection{Characterisation of transport twistor  space}\label{sec_charZ}
Let $(M,g)$ be an oriented Riemannian surface.
 
\begin{proposition}\label{prop_charZ} Let $\D$ be an involutive structure on $DM$. Then the following are equivalent:
\begin{enumerate}
	\item There exists a smooth $\lambda\colon SM\to \R$ with $\deg(\lambda)\le 1$ and a smooth $\tau\colon SM\to \R$ such that the following rotation map is a biholomorphism:
	\[
		R_\tau\colon (DM,\D_\lambda)\to (DM,\D),\quad R_\tau(x,v)=(x,e^{i\tau(x)}v)
	\]
	\item \label{prop_charZii}$\D$ satisfies the following properties:
	\begin{enumerate}[label=\rm(\alph*)]
		\item\label{prop_charZiia} $\D\cap \bar \D = 0$ on $DM^\circ$;
		\item\label{prop_charZiib} $\D\cap \bar \D = \C F$  on $SM$, for a smooth  vector field $F$ that is nowhere vanishing, real and tangent to $SM$;
		\item\label{prop_charZiic} $T^{0,1}D_xM\subset \D$ for all $x\in M$;
		\item\label{prop_charZiid} the zero section $x\mapsto (x,0)$ is holomorphic as map $M\hookrightarrow (DM,\D)$;
		\item\label{prop_charZiie} the lines $T^{1,0}_0D_xM$ ($x\in M$) make up a holomorphic subbundle of $T^{1,0}DM^\circ|_{M}$;
		\item\label{prop_charZiif} the Maslov index of the bundle pair $(\Lambda^{1,0},\Lambda^{1,0}_\R)$ over a disk $D_xM$  ($x\in M$) is equal to $-4$.
	\end{enumerate}
\end{enumerate}
	The function $\lambda$ is unique up to the involution $\lambda\mapsto -\lambda\circ \mathsf{f}$.
\end{proposition}

In (f) we define $\Lambda^{1,0}$ as the annihilator of $\D$ and $\Lambda^{1,0}_\R\subset \Lambda^{1,0}|_{SM}$ the subspace of $(1,0)$-forms that are real valued on $TSM$---in view of Lemma \ref{lem_maslovD} and (b) this totally real, such that Maslov index is defined.

\begin{proof}
Suppose (i) holds true with $\tau\equiv 0$. Then properties (a)--(c) follow immediately from Proposition \ref{prop_defD}, with $F=X+\lambda V$. Properties (d)--(f) are checked in isothermal coordinates $(z,\mu)$ using Lemmas \ref{lem_maslovD} and \ref{lem_zeroD}. Clearly, all properties (a)--(f) are preserved under fibrewise rotations $R_\tau$ and in summary, we have proved (i) $\Rightarrow$ (ii).

For the converse we start by expressing the vector field $F$ in (b) in terms of the frame $\{X,H,V\}$.
 The vectors $F$ and $V$ must be linearly independent at every $(x,v)\in SM$, for otherwise the inclusion $T^{0,1}_vD_xM = \C(V_\perp(x,v)+i V(x,v))\subset \D(x,v)$ would imply that $\dim_\C(\D\cap \bar \D)=2$ at $(x,v)$, which is wrong in view of (b).
Up to multiplying $F$ by a positive function on $SM$ we may assume that there are smooth functions $\tau,\lambda\colon SM\to \R$ such that
\begin{equation}\label{Fbeforerot}
F=(\cos\tau)X+(\sin\tau)H +\lambda V.
\end{equation}
We claim that $\tau$ is independent of $v$. To this end, let $(z,\mu)\in U\times \DD$ be isothermal coordinates on $DM$ and express the involutive structure in normal form $\D=\D_f$ as in Lemma \ref{lem_normalformD}.
By Lemma \ref{lem_zeroD} and  (d)+(e), we must have $f_1(z,0)=\partial_\mu f_1(z,0)=f_2(z,0)=0$, that is,  
\[
	f_1(z,\mu)=\mu^2 h_1(z,\mu),\quad f_2(z,\mu)=\mu h_2(z,\mu)
\]
for smooth $\mu$-holomorphic functions $h_1,h_2\colon U\times \DD\to \C$. By Lemma \ref{lem_maslovD} and (b)+(f) the map $f_1(z,\cdot)\colon S^1\to S^1$ has degree $2$ and thus $h_1$ admits a logarithm
\[
	h_1(z,\mu)= \exp(2i\tilde \tau(z,\mu)),
\]
where $\tilde \tau\colon U\times \DD\to \C$ is smooth and $\mu$-holomorphic. Expressing $X$ and $H$ in isothermal coordinates gives $\mu e^{\sigma + i \tau}F\equiv \partial_z+ \mu^2 e^{2i\tau} \partial_z$ modulo vertical contributions and thus $\tilde \tau(z,\mu)=\tau(z,\mu)\in \R$ for $|\mu|=1$. For fixed $z\in U$, the map  $\mu\mapsto \tilde \tau(z,\mu)$  is holomorphic in $|\mu|<1$ and real-valued on $|\mu|=1$, hence it must be constant. Consequentially, $\tau$ is independent of $v$.

The diffeomorphism $R_\tau\colon DM\to DM$ satisfies
\begin{equation}\label{actionrtau}
	(R_\tau)_*X=(\cos \tau) X+ (\sin\tau) H + (X\tau) V\quad \text{ and }\quad (R_\tau)_*V=V
\end{equation}
and hence $X+(\lambda -X\tau) V$ is pushed forward to the vector field $F$ in \eqref{Fbeforerot}. Making this identification implcit, we assume now that $\tau\equiv 0$, that is, $F=X+\lambda V$ for a smooth function $\lambda\colon SM\to \R$.

In isothermal coordinates and modulo $\partial_{\bar \mu}$, the expression \eqref{Fcoordinates} gives
\[
	\mu e^\sigma F \equiv \partial_{\bar z} + \mu^2 \partial_{z} + [-\mu^3 \sigma_z + \mu \sigma_{\bar z}  + i e^\sigma \mu^2 \lambda] \partial_\mu,\quad |\mu|=1.
\]
Comparing this with normal form $\D_f$ considered above, we obtain
\[
	h_2(z,\mu) =-\mu^2 \sigma_z+ \sigma_{\bar z} + ie^\sigma\mu \lambda,\quad |\mu|=1.
\]
Since the left hand side has a $\mu$-holomorphic extension, the Fourier modes of $\lambda(z,\mu)$ must vanish in degrees $\le -2$, which is to say that $\deg(\lambda)\le 1$.

Lastly we consider the uniqueness of $\lambda$: If there are two tuples $(\lambda,\tau)$ and $(\lambda',\tau')$ as in (i), then the composition $R_{\tau'}^{-1}\circ R_{\tau} = R_{\tau - \tau'}$ is a biholomorphism from $(DM,\D_\lambda)$ to $(DM,\D_{\lambda'})$. In particular, on $SM$ it sends $X+\lambda V$ into $\R(X+\lambda' V)$ and by \eqref{actionrtau} we must have $\sin(\tau-\tau')=0$. This leaves two options: $(\tau-\tau')/ \pi\in 2\Z$ and $\lambda(x,v)\equiv \lambda'(x,v)$ or $(\tau-\tau')/ \pi \in 2\Z +1$ and $\lambda(x,-v)\equiv-\lambda'(x,v)$.
\end{proof}

 \subsection{Proof of Theorem \ref{introthmB}}\label{sec_pfB} Given an oriented closed surace $(M,g)$ and a fibrewise holomorphic blow-down map $\beta\colon DM\to W$ into a complex surface $W$, we need to show the equivalence of the following statements:
 
\begin{enumerate}
	\item\label{thmBi} There exists $\lambda\colon SM\to \R$ of degree $\le 1$ and a rotation $\tau\colon M\to \R$ such that $d\beta_\tau(X+\lambda V)=0$ on $SM$ (and $\lambda$ is unique mod $\lambda\mapsto -\lambda\circ \mathsf{f}$).
	\item\label{thmBii} The following properties are satisfied:
	\begin{enumerate}[label=\rm (\alph*)]
		\item\label{thmBiia} $P=\beta(SM)$ is totally real and $\mu(\beta)=4$;
		\item\label{thmBiib} $\beta(\cdot,0)\colon M\to W$ is a holomorphic curve and the complex lines $d\beta_{(x,0)}(T_0D_xM)$ ($x\in M$) form a holomorphic subbundle of $TW|_Q$, where $Q=q(M)$.
	\end{enumerate}
\end{enumerate}

\begin{proof}[Proof of \ref{thmBi}$\Rightarrow$\ref{thmBii}]
Suppose that (i) holds true for $\tau=0$ and $Z=Z(g,\lambda)$ is the associated transport twistor space. By Lemma \ref{lemholmap} the map $\beta\colon Z^\circ\to W\backslash P$ is a biholomorphism and so the properties in (b) follow immediately from the corresponding ones on $Z$ (Proposition \ref{prop_charZ}). 

For (a) we consider the frame  $\{X, H, V, V_\perp\}$ over $SM$ and recall that $\D_\lambda=\spn_\C(X+\lambda V,V_\perp + i V)$ and that $d\beta$ has full rank on the span of $\{H,V,V_\perp\}$ since it is a blow down map. Given $a,b\in \C$ we claim that \[ d\beta(aH + b (V_\perp -i V))\in T^{0,1}W \quad \Rightarrow \quad a=b=0.\] 
Since $d\beta(V_\perp - i V) \in T^{1,0}W\backslash  0$ we cannot have $a=0$ and $b\neq 0$. If $a\neq 0$, then taking projections onto the $(1,0)$-portion gives
\[
d\beta(H)^{1,0} =  -\frac{b}{a} d\beta(V-iV)
\]
and hence
\(
	d\beta(H) = d\beta(H)^{1,0} + \overline{d\beta(H)^{1,0}} \in \C d\beta(V) \oplus \C d\beta(V_\perp),
\)
which contradicts $d\beta$ having full rank on the span of $\{H,V,V_\perp\}$. This proves our claim, which in turns shows that $d\beta$ gives an isomorphism as follows:
\begin{equation}\label{dbetaiso}
	d\beta\colon T_\C Z/{\D_\lambda} \xrightarrow{\sim} TW/T^{0,1}W 
\end{equation}
(On $DM^\circ$ this follows immediately from $\beta$ being a biholomorphism.) We can now prove that $P=\beta(SM)$ is totally real:  Let $(x,v)\in SM$, $p=\beta(x,v)$ and $w\in T_p P$ with $Jw\in T_pP$. Since $\beta|_{SM}\colon SM\to P$ is a submersion, there exist $\vartheta_1,\vartheta_2\in T_{(x,v)}SM$ with $d\beta(\vartheta_1)=w$ and $d\beta(\vartheta_2)=Jw$. Hence\[
	d\beta(\vartheta_1 + i \vartheta_2) = w + i Jw \in T_p^{0,1}W.
\]
By \eqref{dbetaiso} this implies that $\vartheta_1+i\vartheta_2 \in \D$ and hence $\langle \omega ,\vartheta_k\rangle =0$ for all $\omega \in (\Lambda^{1,0}_\R)_{(x,v)}$, the bundle that was defined in Lemma  \ref{lem_totallyreal}. From this we also get $\vartheta_k\in \C (X+\lambda V)$ and hence $w=d\beta(\vartheta_1)=0$. We conclude that $T_p P\cap JT_pP=0$ and $P$ is totally real.

From \eqref{dbetaiso} we also obtain that $d\beta^\top \colon \Lambda^{1,0}W\to \Lambda_\lambda^{1,0}$ is an isomorphism. Moreover, this is an isomorphism between the totally  real subbundles $\Lambda^{1,0}_\R W\subset T^{1,0}W|_P$ and $\Lambda^{1,0}_{\lambda,\R}\subset \Lambda^{1,0}_\lambda|_{SM}$, which were defined in Lemma \ref{lem_totallyrealequivalence} and Lemma \ref{lem_maslovD}, respectively. Hence
\begin{equation}\label{eqn_chernmaslov}
	\beta^*c_1(W,P) = \beta^*c_1(\Lambda^{2,0}W,\Lambda^{2,0}_\R W) = c_1(\Lambda^{1,0}_\lambda,\Lambda^{1,0}_\lambda)
\end{equation}
and pairing this with $[D_x M]\in H^2(DM,SM)$ yields $-\mu(\beta)/2$ on the left hand and, in view of Lemma \ref{lem_maslovD}, $-2$ on the right hand side. Hence $\mu(\beta)=4$ and the proof of  
\ref{thmBii} is complete.
\end{proof}

\begin{proof}[Proof of \ref{thmBii}$\Rightarrow$\ref{thmBi}]
Our strategy is to construct an involutive structure $\D \subset T_\C(DM)$ that satisfies $d\beta(\D)\subset T^{0,1}W$ and meets all the requirements of Proposition \ref{prop_charZ}\ref{prop_charZii}.

We first construct the vector field $F$ on $SM$: 
 There are isomorphisms $SM \cong SNP \cong SP$, all considered as circle bundles over $P$, where on $SM$ we use $\beta|_{SM}$ as projection. The first isomorphism follows $\beta$ being a blow-down map and the second one is implemented by the complex structure $J$, using that $P$ is totally real. Since $P$ is assumed to be orientable, the unit tangent bundle $SP\to P$ is spanned by a nowhere vanishing vector field. Pulling this back via the isomorphisms  yields a nowhere vanishing vector field $F$ on $SM$ such that: \[
		\ker d(\beta|_{SM}) = \R F.
	\]
	Next we construct $\D$: inside $DM^\circ$ it is clear how to do this, as $\beta$ is a diffeomorphism there. Moreover, any smooth extension of $\D$ to the boundary will automatically be involutive. 
	To find such an extension, let $U\subset W$ be a neighbourhood of some point $p\in P$ over which $\Lambda^{1,0}W$ admits the frame $\{\eta_1,\eta_2\}\subset C^\infty(U,\Lambda^{1,0}W)$, then consider $\omega_k=\beta^*\eta_k \in C^\infty(\beta^{-1}(U),\Lambda^1_\C DM)$ ($k=1,2$). We claim that $\{\eta_1, \eta_2\}$ is everywhere linearly independent, such that an extension can be defined by
	\[
		\D|_{\beta^{-1}(U)} := \ker f_1 \cap \ker f_2 \text{ on } \beta^{-1}(U).
	\]
	 We only need to check that $f_1,f_2$ are linearly independent on $SM$ and for this we consider $(x,v)\in SM$ and $\vartheta_1,\vartheta_2\in T_{(x,v)}SM$ such that $\{\vartheta_1,\vartheta_2,F(x,v)\}$ is a basis. Then the vectors $w_k = d\beta_{(x,v)}(\vartheta_k)\in T_{\beta(x,v)}P$ $(k=1,2)$ are linearly independent and hence
	\[
		\langle \omega_1\wedge \omega_2, \vartheta_1\wedge \vartheta_2\rangle = \langle \eta_1\wedge \eta_2, w_1\wedge w_2\rangle \neq 0.
	\]
	For the last  conclusion we need the restrictions $\eta_k|_{TP}$ to form a basis of $(\Lambda^1_{\C}P)_{\beta(x,v)}$, but this follows from $P$ being totally real and Lemma \ref{lem_totallyrealequivalence}.

Note that $\langle \omega_k, F\rangle =0$ for $k=1,2$ and hence $\C F \subset \D\cap \bar \D$ on $SM$. Moreover, if  $\vartheta\in (\D\cap \bar \D)_{(x,v)}$ over some $(x,v)\in DM$, then \[ d\beta_{(x,v)}(\vartheta)\in T^{0,1}_{\beta(x,v)}W\cap T^{1,0}_{\beta(x,v)}W=0\] which implies that $(x,v)\in SM$ and $\vartheta\in \C F(p)$. This implies properties \ref{prop_charZiia} and \ref{prop_charZiib} of Proposition \ref{prop_charZ}\ref{prop_charZii}.
Property \ref{prop_charZiic} follows immediately from $\beta$ being fibrewise holomorphic and \ref{prop_charZiid} and \ref{prop_charZiie} follow from \ref{thmBiib} of this theorem. Finally, statement \ref{prop_charZiid} on the Maslov index follows by the same argument as following
\eqref{eqn_chernmaslov} and so the proof is complete.
\end{proof}

Finally, we can prove:

\begin{proof}[Proof of Lemma \ref{lem_confmag}]
If $\lambda$ is conformally magnetic, then by Lemma \ref{lem_confmagZ} there is a holomorphic $(2,0)$-form $\Upsilon_Z$ on $Z(g,\lambda)\backslash 0$ that is positive on $SM$ and blows up at most to $2nd$ order at the zero section. We then have $\beta_*\Upsilon_Z = f \Upsilon$ on $W\backslash (P\cup Q)$, where $f$ is a holomorphic function. Since $\Upsilon_Z$ extends to $SM$, the function $f$ is bounded near $P$ and thus extends across $P$. If $U\subset W$ is an open subset of a point in $Q$ and $\varpi\colon U\to \C$ is a local holomorphic defining function of $Q$, then $\beta^*\varpi/|v|$ is bounded on $\beta^{-1}(U)$ and thus $f \cdot (\varpi^2 \Upsilon)$ is bounded on $U$. Since $\operatorname{div}(\Upsilon)=-2Q$, the local $(2,0)$-form $\varpi^2\Upsilon$ is nowhere vanishing and thus $f$ must be bounded on $U$. We conclude that $f$ extends to a holomorphic function on $W$ and thus it is constant $f\equiv c\neq 0$. Moreover, $P$ being Lagrangian for $\Im(c\Upsilon)$ follows directly from $\Upsilon_Z$ being real valued on $SM$.

Vice versa, if we know that $P$ is Lagrangian for $\Im(c\Upsilon)$ for some constant $c\in \C\backslash \{0\}$, then $\Upsilon_Z:=\beta^*(c\Upsilon)$ is a a $(2,0)$-form on $Z(g,\lambda)\backslash 0$ that blows up to $2$nd order at the zero section. Moreover, $\Upsilon_Z(H,V)\colon SM\to \R$ is real-valued by the Lagrangian condition and, perhaps replacing $c$ by $-c$, we can assume that it is everywhere positive. Then Lemma \ref{lem_confmagZ} guarantees that $\lambda$ is conformally magnetic. 
\end{proof}

\section{Model cases}\label{sec_model}
In this section we prove Theorem \ref{introthmC} by constructing explicit holomorphic blow-down maps in the model cases ($K\equiv \mathrm{const.}$ and $\lambda\equiv \mathrm{const.}$ with $K+\lambda^2>0$). It suffices to do this for $K\in \{\pm 1, 0\}$, where the metric takes the form
\[
	g = e^{2\sigma}|dz|^2,\quad  e^{\sigma(z)}=\begin{cases}  \frac{2}{1\pm |z|^2} & K = \pm 1\\
	1 & K=0 \end{cases}
\]
in local isothermal coordinates $z\in U\subset \C$. As discussed in Section \ref{sec_isocord}, in these coordinates, the transport twistor space is given by  $U\times \DD\subset \C^2$, together with the involutive structure $\D=\spn_\C(\Xi_{\sigma,\lambda},\partial_\mu)$, where
\begin{equation}\label{eqnXi}
	\Xi_{\sigma,\lambda}= e^{-\sigma}\Big[\partial_{\bar z} + \mu^2\partial_z +(\sigma_{\bar z} + ie^\sigma \lambda\mu -\sigma_z\mu^2)(\mu\partial_\mu - \bar \mu \partial_\mu)\Big].
\end{equation}
Locally, holomorphic blow-down maps are then described by tuples $(H_1,H_2)$ of solutions to the following Cauchy--Riemann equations:\[
	\Xi_{\sigma,\lambda} H _k = \partial_{\bar \mu} H_ k = 0,\quad k=1,2. 
\]
We proceed case by case ($K\in \{+1,0,-1\}$) by constructing local solutions $(H_1,H_2)$ and then showing that they give rise to global holomorphic blow-down maps as described in Tables \ref{table1} and \ref{table2}.

\subsection{Positive curvature} 
Let $S^2 = \{(\eta,t)\in \C\times \R: |\eta|^2+t^2 =1\}$ be the standard embedding of the $2$-sphere.  Away from the south-pole $\s=(0,-1)$, isothermal coordinates are provided by the stereographic projection:
\[
	\pi_\s\colon S^2\backslash \s\to \C,\quad \pi_\s(\eta,t)=\frac{\eta}{1+t}.
\]
The standard metric $g_\mathrm{can}$ pushes forward to 
 $(\pi_\s)_* g_\mathrm{can}=e^{2\sigma} |dz|^2$, where $e^{\sigma(z)} = 2/(1+|z|^2)$.
Given $\lambda\in \R$, we thus obtain a biholomorphism
\[
	\mathrm{sc}_\sigma\circ (\pi_\s)_\sharp\colon Z(S^2\backslash \s,g_{\mathrm{can}},\lambda) \to \Big(U\times \DD,\spn_\C(\Xi_{\sigma,\lambda},\partial_{\bar \mu})\Big),
\]
where $\mathrm{sc}_\sigma(z,w)=(z,e^{\sigma(z)}w)$ and the complex vector field \eqref{eqnXi} becomes: \[
\Xi_{\sigma,\lambda}
=\frac{1+|z|^{2}}{2}\Big(
\mu^{2}\partial_{z}+\partial_{\bar z}
-\left(\frac{z-\mu^{2}\bar{z}}{1+|z|^{2}}\right)(\mu\partial_{\mu}-\bar\mu\partial_{\bar\mu})
\Big)
+i\lambda\,\mu(\mu\partial_{\mu}-\bar\mu\partial_{\bar\mu}).
\]
To search for holomorphic maps we make the ansatz:
\[
H_{\gamma}(z,\mu):=\frac{z+\gamma\,\mu}{1-\gamma\,\bar z\,\mu}, \quad (z,\mu)\in \C\times \DD.
\]
We try to find $\gamma\in \C$ such that $\Xi_{\sigma,\lambda}H_{\gamma}=0$.
Compute the partial derivatives with $D:=1-\gamma\,\bar{z}\,\mu$:
\[
\partial_z H_\gamma=\frac{1}{D},\quad
\partial_{\bar z}H_\gamma=\frac{\gamma\mu(z+\gamma\mu)}{D^2},\quad
\partial_\mu H_\gamma=\frac{\gamma(1+|z|^2)}{D^2},\quad
\partial_{\bar\mu}H_\gamma=0.
\]
Then
\[
(\mu\partial_\mu-\bar\mu\partial_{\bar\mu})H_\gamma
=\frac{\gamma\mu(1+|z|^2)}{D^2}.
\]
Next, we compute
\[
\begin{aligned}
B
&=\mu^{2}\partial_{z}H_\gamma+\partial_{\bar z}H_\gamma
-\frac{z-\mu^{2}\bar z}{1+|z|^{2}}(\mu\partial_\mu-\bar\mu\partial_{\bar\mu})H_\gamma\\
&=\frac{\mu^2}{D}
+\frac{\gamma\mu(z+\gamma\mu)}{D^2}
-\frac{z-\mu^{2}\bar z}{1+|z|^{2}}\cdot\frac{\gamma\mu(1+|z|^2)}{D^2}\\
&=\frac{\mu^2 D+\gamma\mu(z+\gamma\mu)-(z-\mu^{2}\bar z)\gamma\mu}{D^2}
=\frac{\mu^2(1+\gamma^2)}{D^2}.
\end{aligned}
\]
Therefore
\[
\Xi_{\sigma,\lambda}H_\gamma
=\frac{1+|z|^2}{2}\,B
+i\lambda\,\mu(\mu\partial_\mu-\bar\mu\partial_{\bar\mu})H_\gamma
=\frac{(1+|z|^2)\mu^2}{D^2}\left(\frac{1+\gamma^2}{2}+i\lambda\gamma\right).
\]
The condition for $\Xi_{\sigma,\lambda}H_\gamma=0$ is
\[
\gamma^2+2i\lambda\gamma+1=0,
\]
hence
\[
\gamma=-i\lambda\pm i\sqrt{1+\lambda^2}.
\]
Let \(\gamma_{1}=-i\lambda+i\sqrt{1+\lambda^{2}}\) and \(\gamma_{2}=-i\lambda-i\sqrt{1+\lambda^{2}}\), so that 
\(\gamma_{1}\neq\gamma_{2}\), \(\gamma_{1}\gamma_{2}=1\), and \(\gamma_{1}+\gamma_{2}=-2i\lambda\).
For \(k=1,2\) we obtain the following meromorphic functions:
\[
H_{k}(z,\mu)=\frac{z+\gamma_{k}\mu}{1-\gamma_{k}\bar z\,\mu},\quad (z,\mu)\in \C \times \DD.
\]
We now show that the maps $H_k$ ($k=1,2$) give rise to global holomorphic maps
\(
	\Phi_k\colon Z(S^2, g_\mathrm{can},\lambda)\to  \C P^1.
\)
Away from $\s$ we define
\[
	\Phi_k = H_k \circ \mathrm{sc}_\sigma\circ (\pi_\s)_\sharp,\quad k=1,2,
\]
and proceed by computing this explicitly. Let us write elements of $T_{(\eta,t)}S^2$ as tuples $(\dot \eta,\dot t)$, and likewise elements in $T_z \C$ by $\dot z$. Then
\[
	(\pi_\s)_\sharp\colon T_{(\eta,t)}S^2 \to T_z \C,\quad (\dot \eta,\dot t)\mapsto \dot z= \frac{\dot \eta}{1+t} - \frac{\dot t\eta}{(1+t)^2}.
\]
The sphere constraint $|\eta|^2 = 1- t^2 = (1+t)(1-t)$ gives that
\[
	|z|^2 = \frac{|\eta|^2}{(1+t)^2} = \frac{1-t}{1+t} \quad \Rightarrow \quad e^{\sigma(z)} = \frac{2}{1+|z|^2} = \frac{2}{1+\frac{1-t}{1+t}} = 1+t,
\]
and hence
\[
	\mathrm{sc}_\sigma\colon T_z \C \to T_z\C,\quad \dot z \mapsto (1+t) \dot z.
\]
Substituting  $z = \eta/(1+t)$ and $\mu = \dot \eta - \dot t \eta/(1+t)$ into the enumerator and denominator of $H_k$ separately gives:
\begin{align*}
	z + \gamma_ k \mu &= \frac{1}{1+t} \Big[\eta + \gamma_k\left((1+t)\dot \eta - \dot t\eta\right)\Big],\\
	1-\gamma_k \bar z \mu &= \frac{1}{1+t} \Big[1+t - \gamma_k\left(\bar \eta \dot \eta - \frac{\dot t |\eta|^2}{1-t^2}\right)\Big]\\
	&= \frac{1}{1+t} \Big[1+t - \gamma_k\left(\bar \eta \dot \eta - (1-t) \dot t\right)\Big].
\end{align*}
This results in the following expression:
\[	
	\Phi_k(\eta,t,\dot \eta,\dot t)= H_k\left(\frac{\eta}{1+t} , \dot \eta + \frac{\dot t\eta}{1+t}\right) = \frac{\eta + \gamma_k \left((1+t)\dot \eta - \dot t\eta\right)}{1+t - \gamma_k\left(\bar \eta \dot \eta - (1-t) \dot t\right)}.
\]
It turns out that this map admits a very pretty intrinsic formulation. Let $x = (x_1,x_2,x_3) \in S^3$ and $v = (v_1,v_2,v_3)\in T_x S^2$ and identify $\eta = x_1 + i x_2, t=x_3,\dot \eta = v_1 + iv_2,\dot t=v_3$.
Writing $\times$ for the cross product in $\R^3$, we  compute:
\begin{align*}
	(x \times v)_1 + i(x \times v)_2 &= (x_2v_3 - x_3v_2) + i(x_3v_1 - x_1v_3) = i(t\dot{\eta} - \eta\dot{t}),\\
	(x \times v)_3 &= x_1v_2 - x_2v_1 = \Im(\bar{\eta}\dot{\eta}),
\end{align*}
Here we have used that $x\cdot v = 0$ implies that $\Re(\bar \eta\dot \eta ) = -t\dot t$.
Given $(x,v)\in TS^2$ we define the complex vectors \begin{equation}\label{defWk}
	W_k(x,v):= x + \gamma_k(v - i (x\times v)) \in \C^3,\quad k=1,2.
\end{equation}
Then
\begin{align*}
(W_k)_{1} + i (W_k)_2  &= x_1 + i x_2 + \gamma_k\Big(v_1+iv_2 - i\big((x\times v)_1 + i (x\times v)_2\big)\Big)\\
& = \eta + \gamma_k\big(\dot \eta + t\dot \eta - \eta \dot t\big),\\
1+(W_k)_3 & = 1+ x_3 + \gamma_k\big(v_3 - i (x\times v)_3\big)\\
&= 1 + t + \gamma_k\big(\dot t - t\dot t - \bar \eta \dot \eta\big),
\end{align*}
agree with the numerator and denominator of $\Phi_k$ respectively. Hence
\[
	 \Phi_k(x,v)=\frac{(W_k(x,v))_1+i(W_k(x,v))_2}{1+(W_k(x,v))_3},\quad x\neq \s
\]
Let us check that this extends to a smooth map on $DS^2$:  Taking the complex dot product
\[
	W_k(x,v)\cdot W_k(x,v) = |x|^2 +  \gamma_k^2(|v|^2-|x\times v|^2)=1
\] 
shows that $W_k$ lies on the complex sphere $S^2_\C=\{W\in \C^3: W\cdot W=1\}$. For $W\in S^2_\C$ it holds that
\[
	\frac{W_1 + iW_2}{1+W_3} = \frac{W_1 + iW_2}{1+W_3}  \frac{W_1 - iW_2}{W_1 - iW_2}  =\frac{W_1^2 + W_2^2}{(1+W_3)(W_1-iW_2)} = \frac{1-W_3}{W_1-iW_2}.
\]
Defining the complex stereographic projection by
\[
	\Pi_\s\colon S^2_\C\to \C P^1,\quad \Pi_\s(W):= \begin{cases}
	[W_1 + iW_2: 1+W_3], & W_3\neq -1\\
	[1-W_3: W_1-iW_2],& W_3\neq 1,
	\end{cases}
\]
we thus obtain the following formula:
\[
	\Phi_k\colon DS^2\to \C P^1,\quad \Phi_k(x,v) = \Pi_\s\circ W_k(x,v).
\]
We have proved the following:
\begin{proposition}\label{prop_constructionpositive}
Given $\lambda \in \R$, define
\(
	\gamma_1=-i\lambda+ i\sqrt{1+\lambda^2}\) and \( \gamma_2=-i\lambda- i\sqrt{1+\lambda^2}
\) and $W_k\colon DS^2\to S^2_\C$ ($k=1,2$) as in \eqref{defWk}. Then the map
\[
	\beta\colon Z(S^2,g_{\mathrm{can}},\lambda) \to \C P^1 \times \C P^1,\quad \beta(x,v)=\big(\Pi_\s\circ W_1(x,v),\Pi_\s\circ W_2(x,v)\big)
\]
is smooth up to the boundary and holomorphic. Moreover,
\[
	q\colon S^2\to \C P^1\times \C P^1,\quad q(x):=\beta(x,0) =(\pi_\s(x),\pi_\s(x))
\]
is an embedding of the diagonal $\Delta\subset \C P^1\times  \C P^1$.
\end{proposition}

\begin{proposition}\label{prop_sphereblow} The $\beta\colon DS^2\to \C P^1\times \C P^1$ is a blow-down map along the anti-diagonal $\bar \Delta = \big\{ w_2 = -1/\bar w_1\big\} \subset \C P^1 \times \C P^1$.
\end{proposition} 
\begin{proof}
On the affine patch $\C \times \C \subset \C P^1\times \C P^1$ the antidiagonal is cut out by the equation $F(w_1,w_2):=\bar w_1 w_2  + 1 = 0$.  In local coordinates it holds that
\[
\beta^*F(z,\mu)=F(H_1(z,\mu),H_2(z,\mu))=\frac{(1+|z|^{2})(1-|\mu|^{2})}{(1+\gamma_1 z\bar\mu)(1-\gamma_2\bar z \mu)},
\]
where we have used that $\gamma_1\gamma_2=1$ and $\bar \gamma_k = -\gamma_k$. The numerator vanishes if and only if $|\mu|=1$,  and thus
\(
	\beta^{-1}(\bar \Delta)= \partial (DS^2).
\)
For $|\mu|=1$ we have
\[
	\frac{d}{dr}\Big|_{r=1}\beta^*F(z,r\mu) =\frac{-2(1+|z|^2)}{(1+\gamma_1 z\bar\mu)(1-\gamma_2\bar z \mu)} \neq 0
\]
and thus we obtain an injective map
\(
	d\beta\colon N(\partial DS^2) \to N(\bar \Delta) 
\) between the associated normal bundles. In particular,  Lemma \ref{lem_lift} implies the existence of a smooth lift $\hat \beta\colon DS^2 \to [\C P^1\times \C P^1,\bar \Delta]$ to the blow-up.

We next show $\hat \beta$ is a local diffeomorphism by applying Lemma \ref{lem_melroselocdiff}: given volume forms on $Z=DS^2$ and $W=\C P^1 \times \C P^1$, we need to check that $\beta^*\omega_W = (1-|v|^2) h \omega_Z$ for a positive function $h\colon Z\to (0,\infty)$. We again do this in coordinates.
 On the affine patch $\C\times \C\subset W$ we use the volume form $\Omega_W = dw_1 \wedge d w_2 \wedge d \bar w_1\wedge d \bar w_2 $ and want to compute the pull-back
\[
	\beta^*\Omega_W = dH_1\wedge dH_2 \wedge \overline{dH_1\wedge dH_2}.
\]
We first evaluate the holomorphic wedge product $dH_1 \wedge dH_2$. Because $\partial_{\bar{\mu}} H_k = 0$, we have $dH_k = \partial_z H_k dz + \partial_{\bar{z}} H_k d\bar{z} + \partial_\mu H_k d\mu$. From the previously computed partial derivatives:
\[ \partial_z H_k = \frac{1}{D_k}, \qquad \partial_{\bar{z}} H_k = \frac{\gamma_k\mu H_k}{D_k}, \qquad \partial_\mu H_k = \frac{\gamma_k(1+|z|^2)}{D_k^2}. \]
The 2-form $dH_1 \wedge dH_2$ expands as:
\begin{align*} dH_1 \wedge dH_2 &= (\partial_z H_1 \partial_\mu H_2 - \partial_z H_2 \partial_\mu H_1) dz \wedge d\mu \\
&\quad + (\partial_{\bar{z}} H_1 \partial_\mu H_2 - \partial_{\bar{z}} H_2 \partial_\mu H_1) d\bar{z} \wedge d\mu + J dz \wedge d\bar{z}. \end{align*}
(We do not need the explicit form of $J$.)
We evaluate the coefficients directly. For $dz \wedge d\mu$:
\[ \frac{1}{D_1} \frac{\gamma_2(1+|z|^2)}{D_2^2} - \frac{1}{D_2} \frac{\gamma_1(1+|z|^2)}{D_1^2} = \frac{1+|z|^2}{D_1^2 D_2^2} (\gamma_2 D_1 - \gamma_1 D_2). \]
Since $D_k = 1 - \gamma_k \bar{z} \mu$, the term $\gamma_2 D_1 - \gamma_1 D_2 = (\gamma_2 - \gamma_1)$. The first coefficient can thus be written as \[K:=\frac{(\gamma_2 - \gamma_1)(1+|z|^2)}{D_1^2 D_2^2}.\]
For $d\bar{z} \wedge d\mu$:
\[ \left(\frac{\gamma_1\mu H_1}{D_1}\right) \frac{\gamma_2(1+|z|^2)}{D_2^2} - \left(\frac{\gamma_2\mu H_2}{D_2}\right) \frac{\gamma_1(1+|z|^2)}{D_1^2} = \frac{\gamma_1\gamma_2\mu(1+|z|^2)}{D_1^2 D_2^2} (H_1 D_1 - H_2 D_2). \]
Using $\gamma_1\gamma_2 = 1$ and the numerator identity $H_k D_k = z + \gamma_k \mu$, we find $H_1 D_1 - H_2 D_2 = (\gamma_1 - \gamma_2)\mu$. Thus, the second coefficient simplifies to $-K\mu^2$. 
In summary, the 2-form becomes:
\[ dH_1 \wedge dH_2 = K(dz \wedge d\mu - \mu^2 d\bar{z} \wedge d\mu) + J dz \wedge d\bar{z} = K\alpha + J dz \wedge d\bar{z}, \]
where $\alpha = dz \wedge d\mu - \mu^2 d\bar{z} \wedge d\mu$. This gives:
\[ \beta^*\Omega_W = |K|^2 (\alpha \wedge \bar{\alpha}). \]
Evaluating the wedge product yields $\alpha \wedge \bar{\alpha} = (1 - |\mu|^4) dz \wedge  d\mu \wedge d\bar{z}  \wedge d\bar{\mu}$. 
Therefore, 
\[ \beta^*\Omega_{W} = |K|^2 (1 - |\mu|^4) dz \wedge  d\mu \wedge d\bar{z}  \wedge d\bar{\mu}. \]
The coefficient $|K|^2 = \frac{|\gamma_2 - \gamma_1|^2 (1+|z|^2)^2}{|D_1 D_2|^4}$ is positive and (up to multiplication with positive functions introduced from expressing the given volume forms $\omega_Z$ and $\omega_W$ in coordinates) gives the map $h\colon DM\to (0,\infty)$.

Since $Z=DS^2$ is compact, the lift $\hat \beta$ must be a covering map (Remark \ref{rmk_covering}) and we can conclude the proof by showing that the preimage $([0:1],[0:1]) \in \C P^1\times \C P^1$ consists of a single point. For this we note that  \[\{W\in S^2_\C: \Pi_\s(W)=[0:1]\} = \{(\xi,i\xi,1):\xi\in \C\},\] hence if $\beta(x,v)=([0:1],[0:1])$, then there exist $\xi_1,\xi_2\in \C$ with
	\[
		W_k(x,v)=(\xi_k,i\xi_k,1)\in \C^3,\quad k=1,2.
	\]
	Moreover,
	\begin{align*}
		x = \frac{\gamma_1 W_2 - \gamma_2 W_1}{\gamma_1-\gamma_2} = \left(\frac{\gamma_1 \xi_2-\gamma_2\xi_1}{\gamma_1-\gamma_2},i\frac{\gamma_1 \xi_2-\gamma_2\xi_1}{\gamma_1-\gamma_2},1\right) 
	\end{align*}
	and for this to be a real vector, we must have $\gamma_1\xi_2 = \gamma_2\xi_1$. This shows that $x=(0,0,1)$ is the north pole and hence $v=(v_1,v_2,0)$ for $v_1,v_2\in \R$ to be determined. At this point we compute
	\[
		W_1(x,v)= (0,0,1)+\gamma_1(v_1+iv_2,v_2-iv_1,0) \overset{!}{=} (\xi_1, i\xi_1,1),
	\] 
	which enforces $v_1+iv_2 = \xi_1$ and $-i(v_1+iv_2)=i\xi_1$ and thus $v = 0$.
\end{proof}

\begin{remark} The $\lambda$-geodesic through a given point $(x,v)\in SM$ has centre
\[
	C_{\lambda,+}(x,v) = \frac{\lambda x + x\times v}{\sqrt{1+\lambda^2}}
\]
and it holds that
\(
	\pi_s\circ C_{\lambda,+}(x,v)=\Phi_1(x,v)
\). Due to the $SO(3)$-equivariance of both sides it suffices to check this at one point, and e.g.~for $x=(1,0,0)$ and $v=(0,1,0)$ a short computation yields the value $\lambda/(1+\sqrt{1+\lambda^2})$
 for both sides. A similar argument shows that $\Phi_2 = \pi_\s\circ C_{\lambda,-}$ on $SM$ (the antipodal centre), such that $\beta_\lambda$ is indeed an extension of the circle centre maps.
\end{remark}

\begin{remark}\label{symps2}
Under the diffeomorphism $\pi_\s\times \pi_s\colon S^2\times S^2 \to \C P^1\times \C P^1$ the imaginary part of $\Upsilon = \frac{1}{2i} (dw_1\wedge dw_2)/{(w_1-w_2)^2}$ can be written as follows:
\begin{equation}\label{symps20}
	\Im \Upsilon = \frac{1}{8} d\left[\frac{xdy - y dx}{1-x\cdot y}\right] 
\end{equation}
To see this, we make $\pi_\s$ implicit and start with the following observation:
\begin{equation}\label{symps21}
	 \frac{2|w_1-w_2|^2}{(1+|w_1|^2)(1+|w_2|^2)} = 1-x\cdot y.
\end{equation}
Indeed, writing $x =(\eta,t)$ and $y=(\xi,s)$ as above, we have $w_1 = \eta/(1+t)$ and $w_2=\xi/(1+s)$, such that the left hand side of the preceding display equals
\begin{align*}
	2\frac{\frac{|\eta|^2}{(1+t)^2} - 2 \frac{\Re(\bar \eta \xi)}{(1+t)(1+s)} + \frac{|\xi|^2}{(1+s)^2}}{\left(1+\frac{|\eta|^2}{(1+t)^2}\right)\left(1+\frac{|\xi|^2}{(1+s)^2}\right)} &= 2\frac{\frac{1-t}{1+t} - 2 \frac{\Re(\bar \eta \xi)}{(1+t)(1+s)} + \frac{1-s}{1+s}}{\left(1+\frac{1-t}{1+t}\right)\left(1+\frac{1-s}{1+s}\right)} \\
	&= 2\frac{(1-t)(1+s) - 2 \Re(\bar \eta \xi) + (1-s)(1+t)}{\big((1+t)+(1-t)\big)\big((1+s)+(1-s)\big)} \\
	& = 1 - \Re(\bar \eta \xi)  - st = 1- x\cdot y.
\end{align*}
Next, freeze $y$ (and hence $w_2$), take the logarithm of \eqref{symps21} and then the differential. The left hand side gives:
\begin{align*}
	&d\log\left((w_1-w_2)(\bar w_1 - \bar w_2)\right) + d\log (1+|w_1|^2) 
	\\
	=&\frac{dw_1}{w_1-w_2} + \frac{dw_2}{\bar w_1-\bar w_2} + \text{exact}\\
	 = & 2\Re \left(\frac{dw_1}{w_1-w_2}\right)  + \text{exact}.
\end{align*}
and must equal the corresponding expression for the right hand side:
\[
	 = -\frac{ydx}{1-x\cdot y}.
\]
Moreover,
\[
	\Upsilon = \frac{1}{2i} \frac{dw_1\wedge dw_2}{(w_1-w_2)^2}=d \left(\frac{i}{2} \frac{dw_1}{w_1-w_2}\right),
\]
and using the preceding displays we get: 
\[
	\Im \Upsilon =  d \left(\frac{1}{2} \Re \left(\frac{dw_1}{w_1-w_2}\right)\right) = \frac{1}{4} d\left(\frac{-ydx}{1-x\cdot y}\right).
\]
Changing the roles of $x$ and $y$ just produces a minus sign and so \eqref{symps20} follows by antisymmetrisation.
\end{remark}

\subsection{Zero curvature}

Let us start with the Euclidean plane $(\mathbb{C}, |dz|^{2})$, that is, $\sigma(z)\equiv 0$ and $\lambda\neq 0$.  Then the vector field from \eqref{eqnXi} simplifies  to
\[
\Xi_{0,\lambda}
=\mu^{2}\partial_{z}+\partial_{\bar z}
+i\lambda\,\mu(\mu\partial_{\mu}-\bar\mu\partial_{\bar\mu}).
\]
It is straightforward to check that the functions
\[
H_{1}(z,\mu)= z + \frac{i}{\lambda}\mu,
\qquad
H_{2}(z,\mu)= \bar{z} - \frac{i}{\lambda\mu}
\]
satisfy the Cauchy--Riemann equations and thus yield a holomorphic map
\begin{equation}\label{betalifttorus}
	Z_\lambda:=Z(\C,|dz|^2,\lambda)\to \C\times \C P^1,\quad (z,\mu)\mapsto (H_1(z,\mu),H_2(z,\mu)).
\end{equation}
 Let $\Lambda\subset \mathbb{C}$ be a lattice and $M=\C/\Lambda$ the associated torus with its canonical metric $g_\mathrm{can}.$ The action of $\Lambda$ lifts to transport twistor space 
and we have
 \[
 	Z(M,g_{\mathrm{can}}, \lambda) = Z_\lambda/\Lambda.
 \]	
 We would like to use $H_{1},H_{2}$ to construct a holomorphic blow-down map on this quotient.  For this, observe that under a lattice translation $z\mapsto z+\omega$ with $\omega\in\Lambda$,
\begin{equation}
H_{1}\mapsto H_{1} + \omega,\qquad
H_{2}\mapsto H_{2} + \bar\omega.
\label{eq:Lambdainv}
\end{equation}
We define the complex surface $E_{\tau}$ as the quotient:
\[
E_{\tau}
=\frac{\mathbb{C}\times \mathbb{CP}^{1}}{\sim_{\tau}},
\qquad
(z,p)\sim_{\tau}(z+\omega,\;p+\overline{\omega}),
\quad \omega\in\Lambda.
\]
Each map $(z,p)\mapsto(z+\omega,p+\overline{\omega})$ is holomorphic, so $E_{\tau}$ is a complex 2-manifold.
The projection to the first factor descends to a holomorphic submersion
\[
\rho\colon E_\tau \longrightarrow M,\qquad [(z,p)] \longmapsto [z],
\]
whose fibres are isomorphic to $\mathbb{CP}^1$. Since both the base $M$ and the fibre $\mathbb{CP}^1$ are compact, it follows that $E_\tau$ is compact (closed) and is a holomorphic $\mathbb{CP}^1$-bundle over the elliptic curve $M$ (i.e., a ruled surface over an elliptic curve).

\begin{remark}\label{rmk_atiyah}
The complex surface $E_{\tau}$ is diffeomorphic to $M\times  \mathbb{CP}^{1}$. This can be seen by considering the map $\Psi: \mathbb{C} \times \mathbb{CP}^1 \to \mathbb{C} \times \mathbb{CP}^1$ given by:
\[\Psi(z, p) = (z, p - \bar{z}).\]
It is easy to check that it descends to a diffeomorphism $E_{\tau}\cong M\times  \mathbb{CP}^{1}$.
On the other hand $E_{\tau}$ is holomorphically not a product: it is in fact the projectivisation of the unique topologically trivial non-split holomorphic rank $2$ bundle over $M$, that is, {\it Atiyah’s ruled surface} \cite[Theorem 5]{At_57}.
To see this, we observe that the monodromy  of $E_\tau$ is contained in the translation subgroup of $\text{Aut}(\mathbb{CP}^1)$, given in affine coordinates by $p\mapsto p+c$. 
Holomorphic $\mathbb{CP}^1$-bundles over an elliptic curve $M$ with structural group restricted to translations ($p \mapsto p + c$) are classified by $H^1(M, \mathcal{O}_M) \cong H^{0,1}(M) \cong \mathbb{C}$. The zero class yields the trivial product $M \times \mathbb{CP}^1$. Because $\mathbb{C}^*$ acts by scaling the fiber, all non-zero classes are biholomorphically equivalent, defining the unique non-trivial affine bundle known as the Atiyah surface. 
To determine the class of $E_\tau$, we measure its obstruction to holomorphic triviality. The global smooth frame from the map $F$ above, $q = p - \bar{z}$, yields the relation $p = q + \bar{z}$. The failure of this trivialization to be holomorphic is given by the $(0,1)$-form $\bar{\partial} p = d\bar{z}$. But the class $[d\bar{z}] \neq 0$ generates $H^{0,1}(M)$ and thus $E_\tau$ is biholomorphic to the Atiyah surface.
\end{remark}

Due to the equivariance \eqref{eq:Lambdainv}, the map in \eqref{betalifttorus} descends to a holomorphic map
\[\beta=:\beta_{\lambda}: DM=Z_{\lambda}/\Lambda\to E_{\tau}.\]
Along the zero section, we have $\beta([z],0)=([z],\infty)$, and the image is the complex curve $Q = M\times \{\infty\}$.

\begin{proposition} The map $\beta\colon DM\to E_\tau$ is a blow-down map along the totally real surface
\(
	P =\{[w, \bar{w}] \in E_{\tau} : w \in \mathbb{C} \}.
\)
\end{proposition}

\begin{proof}
	We work on the universal cover, where the lift of $P$ is cut out by the equation $F(w,\xi):= w - \bar \xi = 0$. We have
	\[
		\beta^*F(z,\mu)= \frac{i}{\lambda} \mu - \frac{i}{\lambda\bar \mu} = \frac{1-|\mu|^2}{i\lambda \bar \mu} 
	\]
	and since the numerator vanishes if and only if $|\mu|=1$, we obtain $\beta^{-1}(P)=\partial (DM).$ Moreover $\frac{d}{dt}\big|_{r=1}F(z,r\mu)= 2i\mu \lambda\neq 0$ for $|\mu|=1$ and hence by Lemma \ref{lem_lift} there is a lift $\hat \beta\colon DM\to [E_\tau,P]$ to the blow-up.

Next we check the requirement of Lemma \ref{lem_melroselocdiff} by computing the Jacobian on the universal cover: 
	Let $w = H_1 = z + \frac{i}{\lambda}\mu$ and $\xi = H_2 = \bar{z} - \frac{i}{\lambda\mu}$.
We first compute the holomorphic wedge product $dw \wedge d\xi$:
\[ dw \wedge d\xi = \left(dz + \frac{i}{\lambda}d\mu\right) \wedge \left(d\bar{z} + \frac{i}{\lambda\mu^2}d\mu\right). \]
Expanding this yields:
\[ dw \wedge d\xi = dz \wedge d\bar{z} + \frac{i}{\lambda\mu^2} dz \wedge d\mu - \frac{i}{\lambda} d\bar{z} \wedge d\mu \]
and factoring out the term $\frac{i}{\lambda\mu^2}$ gives:
\[ dw \wedge d\xi = dz \wedge d\bar{z} + \frac{i}{\lambda\mu^2} (dz \wedge d\mu - \mu^2 d\bar{z} \wedge d\mu) = dz \wedge d\bar{z} + \frac{i}{\lambda\mu^2} \alpha, \]
where $\alpha=dz \wedge d\mu - \mu^2 d\bar{z} \wedge d\mu$
We now compute the volume form $dw \wedge d\bar{w} \wedge d\xi \wedge d\bar{\xi}$, which is defined by $- (dw \wedge d\xi) \wedge \overline{(dw \wedge d\xi)}$. 
Any term wedged with $dz \wedge d\bar{z}$ and $d\bar{z} \wedge dz$ yields zero. Furthermore, since $\bar{\alpha}$ consists entirely of terms with $d\bar{z}$ and $d\bar{\mu}$, the cross terms $dz \wedge d\bar{z} \wedge \bar{\alpha}$ vanish. The expression reduces to the $\alpha \wedge \bar{\alpha}$ term:
\[ -(dw \wedge d\xi) \wedge \overline{(dw \wedge d\xi)} = -\frac{1}{\lambda^2|\mu|^4} \alpha \wedge \bar{\alpha} \]
Substituting $\alpha \wedge \bar{\alpha} = -(1-|\mu|^4) dz \wedge d\bar{z} \wedge d\mu \wedge d\bar{\mu}$, the local coordinate pullback is:
\[ \beta^*(dw \wedge d\bar{w} \wedge d
\xi \wedge d\bar{\xi}) = \frac{1-|\mu|^4}{\lambda^2|\mu|^4} dz \wedge d\bar{z} \wedge d\mu \wedge d\bar{\mu} \]

To see that this extends smoothly to the interior, including the zero section $\mu=0$, we note that $dw \wedge d\bar{w} \wedge d\xi \wedge d\bar{\xi}$ is merely the standard volume form on $\mathbb{C}^2$, not the compactified surface $E_\tau$. Near the zero section $\mu=0$, the fiber coordinate $p \to \infty$. A globally smooth volume form $\omega_E$ on the $\mathbb{CP}^1$-bundle $E_\tau$ behaves near infinity as:
\[ \omega_E \sim dw \wedge d\bar{w} \wedge d(1/\xi) \wedge d(1/\bar{\xi}) = \frac{1}{|\xi|^4} dw \wedge d\bar{w} \wedge d\xi \wedge d\bar{\xi} \]
Because $\xi \sim -i/(\lambda\mu)$ near $\mu=0$, we have $|\xi|^4 \sim \frac{1}{\lambda^4|\mu|^4}$. Pulling back this honest volume form introduces a $|\mu|^4$ factor in the numerator that cancels the $1/|\mu|^4$ pole in the Jacobian.
Thus, given a volume form  $\omega_Z$ on $Z=DM$, the rescaled Jacobian $\frac{\beta^*\omega_E}{(1-|z|^2)\omega_Z}$ is a strictly positive, smooth function everywhere on $DM$. This implies that the lift $\hat \beta\colon Z\to [E_\tau,P]$  is a local diffeomorphism.

To conclude, we only need to exhibit one point $[w,\xi]\in E_{\tau}\setminus P$ such that $\beta^{-1}([w,\xi])$ consists only of one point. The simplest point to pick is $[0,\infty]$ and due to equivariance is suffices to check that there is a unique solution to $H_{1}(z,\mu)=0$ and $H_{2}(z,\mu)=\infty$.  Indeed the only solution to these equations is $(z,\mu)=(0,0)$.
\end{proof}

\subsection{Negative curvature}
We consider the hyperbolic disc $\DD^\circ$, with metric
\[
g=e^{2\sigma}|dz|^{2},\quad e^{\sigma}=\frac{2}{1-|z|^{2}},
\]
together with  a constant magnetic parameter $\lambda\in\mathbb{R}\backslash [-1,1]$. The associated transport twistor space is $Z_\lambda = \DD^\circ \times \DD$, where the complex vector field from \eqref{eqnXi} now takes the following form:
\[
\Xi_{\sigma,\lambda}
=\frac{1-|z|^{2}}{2}\Big(
\mu^{2}\partial_{z}+\partial_{\bar z}
+\left(\frac{z-\mu^{2}\bar{z}}{1-|z|^{2}}\right)(\mu\partial_{\mu}-\bar\mu\partial_{\bar\mu})
\Big)
+i\lambda\,\mu(\mu\partial_{\mu}-\bar\mu\partial_{\bar\mu}).
\]
To search for holomorphic first integrals we make the ansatz:
\[
H_{\gamma}(z,\mu):=\frac{z+\gamma\,\mu}{1+\gamma\,\bar z\,\mu},
\]
and we try to find $\gamma$ such that $\Xi_{\sigma,\lambda}H_{\gamma}=0$. (Note the change in sign compared to the positive curvature case!)
Let $D:=1+\gamma\,\bar z\,\mu$, then
\begin{align*}
&\partial_z H_{\gamma}=\frac{1}{D},
&\partial_{\bar z}H_{\gamma}=-\frac{\gamma\,\mu(z+\gamma\mu)}{D^{2}},\\
&\partial_{\mu}H_{\gamma}=\frac{\gamma(1-|z|^{2})}{D^{2}},
&\partial_{\bar\mu}H_{\gamma}=0, ~~\quad\qquad\qquad
\end{align*}
and therefore
\[
(\mu\partial_{\mu}-\bar\mu\partial_{\bar\mu})H_{\gamma}
=\frac{\gamma\,\mu(1-|z|^{2})}{D^{2}}.
\]
Substituting the derivatives into $\Xi_{\sigma,\lambda}$ gives
\[
\Xi_{\sigma,\lambda}H_{\gamma}
=\frac{1-|z|^{2}}{2(1+\gamma\bar z\mu)^{2}}
\mu^{2}\Big[(1-\gamma^{2})+2i\lambda\gamma\Big].
\]
Hence the condition for $\Xi_{\sigma,\lambda}H_{\gamma}=0$ is
\(
\gamma^{2}-2i\lambda\gamma-1=0.
\) and 
therefore
\[
\gamma=i\lambda\pm\sqrt{1-\lambda^{2}}.
\]

\subsubsection{Equivariance}

Let 
\[
A=\begin{pmatrix} a & b \\ \bar b & \bar a \end{pmatrix}\in \mathrm{SU}(1,1), 
\qquad |a|^{2}-|b|^{2}=1.
\]
The associated automorphism of the disc is
\[
\phi_{A}(z)=\frac{a z + b}{\bar b z + \bar a},
\]
and its lift to the unit disc bundle 
\(\DD^\circ\times\DD\) is
\[
\Phi_{A}(z,\mu)
=\left(\phi_A(z),\phi_A'(z) \frac{1-|z|^2}{1-|\phi_A(z)|^2 } \mu\right) \equiv \left(
\phi_{A}(z),\;
\mu\,\frac{a+b\bar z}{\bar a+\bar b z}
\right).
\]
Then one checks directly that
\[
H_{\gamma}(z,\mu)
=\frac{z+\gamma\,\mu}{1+\gamma\,\bar z\,\mu}
\quad\text{satisfies}\quad
H_{\gamma}\!\big(\Phi_{A}(z,\mu)\big)
=\phi_{A}\!\big(H_{\gamma}(z,\mu)\big).
\]
Hence \(H_{\gamma}\) is \emph{equivariant} with respect to 
the natural action of \(\mathrm{Aut}(\mathbb{D})\) on 
\(\mathbb{D}^\circ\times{\mathbb{D}}\) and on \(\mathbb{CP}^1\)
via the maps \(\Phi_{A}\) and \(\phi_{A}\).

\subsubsection{When does $H_{\gamma}$ map into $\mathbb{D}^\circ$?} 
The discussion that follows is key to understand when we get a genuine holomorphic map all the way to $|\mu|=1$. It will be particularly relevant when we start taking compact quotients of the hyperbolic disc. Assume from now on that $\lambda>1$.

Recall our generic ansatz 
\[
H_{\gamma}(z,\mu)=\frac{z+\gamma\,\mu}{1+\gamma\,\bar z\,\mu}, 
\qquad |z|<1,\;|\mu|\le 1,
\]
where \(\gamma\) is one of the roots satisfying $\gamma^{2}-2i\lambda\,\gamma-1=0$, namely:
\[
\gamma_{1,2}=i\left(\lambda\mp\sqrt{\lambda^{2}-1}\right),
\qquad
\gamma_{1}\gamma_{2}=-1.
\]
A direct computation gives
\[
|H_{\gamma}(z,\mu)|^{2}
=\frac{|z+\gamma\,\mu|^{2}}{|1+\gamma\,\bar z\,\mu|^{2}}
=\frac{|z|^{2}+|\gamma|^{2}|\mu|^{2}+2\,\mathrm{Re}(\gamma\,\bar z\,\mu)}
{1+|\gamma|^{2}|z|^{2}|\mu|^{2}+2\,\mathrm{Re}(\gamma\,\bar z\,\mu)}.
\]
After simplification one finds:
\[
|H_{\gamma}(z,\mu)|<1 
\quad\Longleftrightarrow\quad
(1-|z|^{2})\bigl(1-|\gamma|^{2}|\mu|^{2}\bigr)>0
\quad\Longleftrightarrow\quad
|\gamma|\,|\mu|<1.
\]
Let $H_1$ and $H_2$ denote the maps corresponding to the roots $\gamma_1$ and $\gamma_2$. 
For the first root \(\gamma_{1}\), one has \(|\gamma_{1}|<1\), hence 
\(|\gamma_{1}|\,|\mu|\le|\gamma_{1}|<1\) for all \(|\mu|\le1\), and therefore
\[
|H_1(z,\mu)|<1 
\quad\text{for all }(z,\mu)\in\mathbb D^\circ\times{\mathbb D}.
\]
In this case the denominator \(1+\gamma_{1}\bar z\,\mu\) never vanishes in the domain.
For the second root \(\gamma_{2}\), one has \(|\gamma_{2}|>1\), so 
the condition becomes \(|\mu|<1/|\gamma_{2}|\). Thus
\[
|H_2(z,\mu)|<1
\quad\text{if and only if}\quad 
|\mu|<\frac{1}{|\gamma_{2}|}.
\]
When \(|\mu|=1/|\gamma_{2}|\) one has \(|H_2(z,\mu)|=1\) for all \(z\),
and for \(|\mu|>1/|\gamma_{2}|\) the inequality fails.
Poles appear when \(1+\gamma_{2}\bar z\,\mu=0\),
that is, when \(|z|=1/(|\gamma_{2}||\mu|)<1\).
Hence, the map $H_1$ (corresponding to $\gamma_1$) is the one that sends
\(\mathbb D^\circ\times{\mathbb D}\) entirely into the open unit disc.

\subsubsection{Compact quotients} 

Consider now a cocompact Fuchsian group $\Gamma\subset \text{Aut}(\mathbb{D})$ so that $M=\mathbb{D}^\circ/\Gamma$ is a closed Riemann surface. We know that for $|\lambda|<1$, $X_{\lambda}$ is Anosov and for the critical value $|\lambda|=1$ we hit the horocycle flow. We do not have first integrals in these situations and this is clearly reflected in the discussion above.

However, for $\lambda>1$, the map $H_1 : Z_{\sigma,\lambda}\to\mathbb{D}^\circ$ has the correct equivariance properties to induce a holomorphic fibration
\[
H_1:\; Z_{\sigma,\lambda}/\Gamma \;\longrightarrow\; M = \mathbb{D}^\circ/\Gamma,
\]
where we slightly abuse notation by using $H_1$ for the induced map on the quotient. The fibers $H_1^{-1}([w])$ are holomorphic discs in $Z_{\sigma,\lambda}/\Gamma$.

For a fixed $w \in \mathbb{D}$, the fiber in the universal cover is:
\[
H_1^{-1}(w)
= \Big\{(z,\mu) \in Z_{\sigma,\lambda} :
\mu = \frac{z - w}{\gamma_1\,(w\bar z - 1)} \Big\}.
\]
Restricting to the unit circle bundle $|\mu| = 1$ gives
\[
|\mu| = 1
\ \Longleftrightarrow\
|z - w| = |\gamma_1|\,|1 - \bar w z|
\ \Longleftrightarrow\
\Big|\frac{z - w}{1 - \bar w z}\Big| = |\gamma_1|
= \tanh(\rho/2),
\]
where $\rho$ is the hyperbolic radius associated to the magnetic flow curvature. Hence the projection of $\partial H_1^{-1}(w)$ onto the base $\mathbb{D}^\circ$ is exactly the hyperbolic circle
\[
C_w(\rho) = \{\, z \in \mathbb{D}^\circ : d_{\mathbb{H}}(z, w) = \rho \,\}.
\]
On the boundary, the unit tangent vector satisfies
\[
\mu(z) = -\frac{1}{\gamma_1} \left( \frac{z - w}{1 - \bar w z} \right) =: -\frac{1}{\gamma_1} \varphi_w(z).
\]
Because $|\varphi_w(z)| = |\gamma_1|$ on this boundary and $\gamma_1 = i|\gamma_1|$, this simplifies to:
\[
\mu(z)
= i\, \frac{\varphi_w(z)}{|\varphi_w(z)|}.
\]
Hence for $(z,\mu)\in SM$, the point $H_1(z,\mu)\in \DD^\circ$ is the hyperbolic centre of the geodesic circle through $(z,\mu)$. The fibres of $H_1$ can be viewed as holomorphic disk fillings of the geodesic  circles in $SM$.

\subsubsection{The blow-down map}

Recall that we are assuming $\lambda>1$. Let $\gamma_2$ be the second root of the quadratic ($|\gamma_2|>1$) and recall the map
\[
H_2(z,\mu)=\frac{z+\gamma_2\mu}{1+\gamma_2\bar z\,\mu}.
\]
Although $H_2$ is only meromorphic on $Z_{\sigma,\lambda}$ 
(it has poles where $1+\gamma_2\bar z\,\mu=0$),
it satisfies the same $\Gamma$--equivariance as $H_1$.
Hence it determines a holomorphic map $H_2 : Z_{\sigma,\lambda}\to \mathbb{CP}^1$ giving rise to
\[
\beta=(H_1,H_2): Z_{\sigma,\lambda}\longrightarrow \mathbb{D}^\circ\times\mathbb{CP}^1
\]
that descends to a holomorphic map
\[\beta:DM=Z_{\sigma,\lambda}/\Gamma\to M\times_{\Gamma}\mathbb{CP}^1 := (\mathbb{D}^\circ\times\mathbb{CP}^1)/\Gamma,\]
so the target again a ruled surface over $M$. Moreover, $\beta([z,0])=([z,z])$.

\begin{remark}
	The complex surface $W=\DD^\circ \times_\Gamma \C P^1$ is diffeomorphic to $M\times S^2$. One way to see this is via the map
	\[
		\Psi\colon \DD^\circ\times \C P^1\to \DD^\circ\times \C P^1,\quad \Psi(w,\xi)=(1-|w|^2)\frac{\xi-w}{1-\bar w\xi}.
	\]
	A direct computation shows that
	\[
		\Psi(\phi_A(w),\phi_A(\xi))=\phi_A'(w) \Psi(w,\xi),\quad A\in SU(1,1).
	\]
	Under the stereographic identification $\C P^1\equiv S^2\subset \C \times \R$, the factor $\phi_A'(w)$ only acts on the first factor and we can identify the quotient $(\DD^\circ \times \C)/\Gamma$
	with the tangent bundle $TM$. Hence $\Psi$ descends to a map $\DD^\circ \times_\Gamma \C P^1 \to S(TM\oplus \R)$ into the unit circle bundle of $TM\oplus \R = M\times \R^3$, which gives the desired diffeomorphism. 
\end{remark}

\begin{proposition} The map $\beta\colon DM\to W = M\times_\Gamma \C P^1$ is a blow-down map along the totally real surface
\[P=\{[(w,1/\bar{w})]:\;w\in\mathbb{D}\}\subset W.\]
\end{proposition}

\begin{proof}
	We work on the universal cover, where the lift of $P$ is cut out by the equation $F(w,\xi):= \bar w\xi - 1 = 0$. To understand the pull-back by $\beta$,
we evaluate the product of $H_1$ and the conjugate of $H_2$:
\[
H_1(z,\mu)\,\overline{H_2(z,\mu)}
=\frac{(z+\gamma_1\mu)\,(\bar z+\overline{\gamma_2}\,\bar\mu)}
{(1+\gamma_1\bar z\,\mu)\,(1+\overline{\gamma_2}\,z\,\bar\mu)}.
\]
Because \(\gamma_2\) is purely imaginary, \(\overline{\gamma_2}=-\gamma_2\). Together with \(\gamma_1\gamma_2=-1\), this yields \(\gamma_1\,\overline{\gamma_2}=1\). 
Expanding the products and applying this identity simplifies the expression to:
\[
H_1(z,\mu)\,\overline{H_2(z,\mu)}
=\frac{|z|^{2}+|\mu|^{2}+\gamma_1\,\bar z\,\mu+\overline{\gamma_2}\,z\,\bar\mu}
{1+|z|^{2}|\mu|^{2}+\gamma_1\,\bar z\,\mu+\overline{\gamma_2}\,z\,\bar\mu}.
\]
Substituting \(\overline{\gamma_2}=\frac{1}{\gamma_1}\), we obtain
\[
	\beta^*F(z,\mu) = -\frac{(1-|z|^2)(1-|\mu|^2)}{1+|z|^{2}|\mu|^{2}+\gamma_1\,\bar z\,\mu+\frac{1}{\gamma_1}\,z\,\bar\mu}.
\]
This vanishes if and only if $|\mu|=1$, such that $\beta^{-1}(P)=\partial (DM)$. Analogously to above we see that $\frac{d}{dr}\big|_{r=1}\beta^*F(z,r\mu)\neq 0$, such that the lift $\hat \beta\colon DM\to [W,P]$ exists.

	Next we compute the Jacobian: Consider the coordinate volume form $\Omega_W = dw \wedge d\bar{w} \wedge d\xi \wedge d\bar{\xi}$.
The total differential is $\beta^*\Omega_E = dH_1 \wedge d\bar{H}_1 \wedge dH_2 \wedge d\bar{H}_2 = - (dH_1 \wedge dH_2) \wedge \overline{(dH_1 \wedge dH_2)}$. 
We first evaluate the holomorphic wedge product $dH_1 \wedge dH_2$. Since $\partial_{\bar{\mu}} H_k = 0$, we have $dH_k = \partial_z H_k dz + \partial_{\bar{z}} H_k d\bar{z} + \partial_\mu H_k d\mu$. From our previous definitions where $D_k = 1 + \gamma_k \bar{z} \mu$:
\[ \partial_z H_k = \frac{1}{D_k}, \qquad \partial_{\bar{z}} H_k = -\frac{\gamma_k\mu H_k}{D_k}, \qquad \partial_\mu H_k = \frac{\gamma_k(1-|z|^2)}{D_k^2}. \]
The 2-form $dH_1 \wedge dH_2$ expands as:
\begin{align*}
dH_1 \wedge dH_2 &= (\partial_z H_1 \partial_\mu H_2 - \partial_z H_2 \partial_\mu H_1) dz \wedge d\mu \\
&+ (\partial_{\bar{z}} H_1 \partial_\mu H_2 - \partial_{\bar{z}} H_2 \partial_\mu H_1) d\bar{z} \wedge d\mu + J dz \wedge d\bar{z}. 
\end{align*}
(We do not need the explicit form of the function $J$.) 

We compute the first two coefficients directly. For $dz \wedge d\mu$:
\[ \frac{1}{D_1} \frac{\gamma_2(1-|z|^2)}{D_2^2} - \frac{1}{D_2} \frac{\gamma_1(1-|z|^2)}{D_1^2} = \frac{1-|z|^2}{D_1^2 D_2^2} (\gamma_2 D_1 - \gamma_1 D_2). \]
Since $D_k = 1 + \gamma_k \bar{z} \mu$, the term $\gamma_2 D_1 - \gamma_1 D_2 = \gamma_2 - \gamma_1$. Let $K = \frac{(\gamma_2 - \gamma_1)(1-|z|^2)}{D_1^2 D_2^2}$. The first coefficient is precisely $K$.

For $d\bar{z} \wedge d\mu$:
\[ \left(-\frac{\gamma_1\mu H_1}{D_1}\right) \frac{\gamma_2(1-|z|^2)}{D_2^2} - \left(-\frac{\gamma_2\mu H_2}{D_2}\right) \frac{\gamma_1(1-|z|^2)}{D_1^2} = \frac{\gamma_1\gamma_2\mu(1-|z|^2)}{D_1^2 D_2^2} (-H_1 D_1 + H_2 D_2). \]
Recalling $\gamma_1\gamma_2 = -1$ and $H_k D_k = z + \gamma_k \mu$, we get $-H_1 D_1 + H_2 D_2 = (\gamma_2 - \gamma_1)\mu$. Thus, the second coefficient simplifies  to $-K\mu^2$. 

Our 2-form is therefore:
\[ dH_1 \wedge dH_2 = K(dz \wedge d\mu - \mu^2 d\bar{z} \wedge d\mu) + J dz \wedge d\bar{z}. \]
Let $\alpha = dz \wedge d\mu - \mu^2 d\bar{z} \wedge d\mu$. Then 
\[\beta^*\Omega_{E}=- (K\alpha + J dz \wedge d\bar{z}) \wedge (\bar{K}\bar{\alpha} + \bar{J} d\bar{z} \wedge dz)=- |K|^2 (\alpha \wedge \bar{\alpha}). \]
We calculate $\alpha \wedge \bar{\alpha}= -(1 - |\mu|^4) dz \wedge d\bar{z} \wedge d\mu \wedge d\bar{\mu}$.
Therefore, the local volume form transforms as:
\[ \beta^*(dw \wedge d\bar{w} \wedge d\xi \wedge d\bar{\xi}) = |K|^2 (1 - |\mu|^4) dz \wedge d\bar{z} \wedge d\mu \wedge d\bar{\mu}. \]
Because $|K|^2 = \frac{4(\lambda^2-1)(1-|z|^2)^2}{|D_1 D_2|^4}$ is positive, the pullback's zeroes are governed entirely by $1 - |\mu|^4 = (1+|\mu|^2)(1-|\mu|^2)$. 
Analogously to the previous cases this shows that the lift $\hat \beta$ is a local diffeomorphism. 

To conclude, we need to exhibit one point $[w,\xi]\in W\setminus P$ such that $\beta^{-1}([w,\xi])$ consists only of one point. The simplest point to pick is $[0,0]$ and due to equivariance is suffices to check that there is a unique solution to $H_{1}(z,\mu)=0$ and $H_{2}(z,\mu)=0$.  Indeed the only solution to these equations is $(z,\mu)=(0,0)$.
\end{proof}

\section{Perturbation theory for disk families}\label{sec_diskfam}
In this section we prove Theorem \ref{introthmD} on the persistence of disk families under perturbations. Throughout, $(M,g)$ is an oriented closed Riemannian surface,
\[
	\rho\colon W\to M
\]
is a ruled surface over $M$, which comes with a distinguished holomorphic section $q\colon M\to W$ whose image we write as $Q=q(M)$, as well as a meromorphic $2$-form $\Upsilon$ on $W$ such that
\[
	\operatorname{div}(\Upsilon)=-2Q\quad \text{ and }\quad \mathrm{Res}_Q(\Upsilon)=0.
\]
As preparation for the perturbation theory, we first discuss some consequences of the existence of $\Upsilon$ and set up a suitable moduli space of holomorphic disks.

\subsection{Splittings and residue} 

\subsubsection{Splittings of the normal bundle} 
Let 
\begin{equation}\label{normalses}
	0\to TQ\to TW|_Q\to NQ\to 0
\end{equation}
be the normal bundle sequence of $Q$. This admits a holomorphic splitting $TW|_Q=TQ\oplus L$, where $L=\ker d\rho|_Q\subset TW|_Q$ is the {\it vertical} line bundle with respect to the ruling. We claim that there exists an isomorphism
\begin{equation}\label{defpsi}
	\psi\colon NQ\xrightarrow{\sim} TQ.
\end{equation}
Writing $\O(D)$ for  the line bundle associated to $D\in \mathrm{Div}(W)$, the adjunction formulas \cite[Chapter 1.1]{GrHa78} state that 
\begin{equation}
	\O (Q)|_Q\cong NQ \quad \text{ and } \quad K_Q\cong (K_W\otimes \O(Q))|_Q,
\end{equation}
where $K_Q$ and $K_W$ are the canonical line bundles of $Q$ and $W$, respectively. Since $\Upsilon$ is a meromorphic section of $K_W$ with divisor $-2Q$, we have $\O(-2Q)\cong K_W$ and hence $K_Q \otimes NQ$ is trivial. In particular there exists an isomorphism $\psi$ as in \eqref{defpsi}. Below we also give an explicit description of $\psi$ in terms of $\Upsilon$.
As a consequence, we can compute the extension group of \eqref{normalses} as
\[
	H^0(Q,N^*Q\otimes TQ)\cong \C,
\]
which means the following: Once $\psi$ as in \eqref{defpsi} is fixed, any holomorphic splitting of \eqref{normalses} can be described by a complex number $\zeta\in \C$ as $TW|_Q=TQ\oplus L^\zeta$, where the line bundle $L^\zeta\to Q$ is defined by
\begin{equation}
\label{defL}
	L^\zeta_q = \spn_\C(v+\zeta\psi_q([v])) ,\quad v\in L_q,q\in Q.
\end{equation}

\subsubsection{The residue} \label{sec_residue}
Suppose $U\subset W$ is an open set with $U\cap Q\neq \emptyset$ and $(z,w)$ is a complex coordinate system on  $U$ with $U\cap Q=\{w=0\}$. Since $\operatorname{div}(\Upsilon)=-2Q$, there is a Laurent expansion
\begin{equation}\label{expansion}
	\Upsilon = \upsilon_{-2}(z)dz \wedge \frac{dw}{w^2}+ \upsilon_{-1}(z)dz \wedge \frac{dw}{w} + \text{holomorphic $2$-form}
	\end{equation}
with  $\upsilon_{-2}(z)\neq 0$ everywhere. The terms $\upsilon_{-2}$ and $\upsilon_{-1}$ have a global meaning as
\[
	\psi^{-1}|_{U\cap Q}:=\upsilon_{-2}(z) dz\otimes \partial_w \quad \text{ and } \quad \Res_Q\Upsilon|_{U\cap Q}:=\upsilon_{-1}(z) dz.
\]
Here $\psi$ gives an isomorphism as in \eqref{defpsi} and the holomorphic $1$-form $\Res_Q \Upsilon$ is the residue of $\Upsilon$ along $Q$. For the latter to be well-defined, we  restrict our attention to charts with
\(
	dz \perp L\) along $Q$.

\begin{lemma}\label{lem_residue} \,
	\begin{enumerate}
	\item There is a unique nowhere vanishing section $\psi^{-1}$ of $K_Q\otimes NQ$ such that $\psi^{-1}=\upsilon_{-2} dz\otimes \partial_w$ in any local chart $(U,z,w)$ with $U\cap Q=\{w=0\}$.
	\item There is a unique holomorphic $1$-form $\Res_Q\Upsilon$ on $Q$ such that $\Res_Q \Upsilon = \upsilon_{-1}(z)dz$ in any local chart $(U,z,w)$ with $U\cap Q=\{w=0\}$ and $dz\perp L$ along $Q$.
	\end{enumerate}
\end{lemma}


\begin{proof}
	Suppose we are given a coordinate system $(U,z,w)$ with $Q\cap U=\{w=0\}$ and $f,g\colon U\to \C$ are new coordinate functions such that $Q\cap U=\{g=0\}$. Then $g(z,w)=wh(z,w)$ for some holomorphic $h\colon U\to \C$. By an elementary computation,
	\begin{equation}\label{residue2}
	\begin{split}
		\frac{df\wedge dg}{g^2} &=	\left[\frac{f_z(h+wh_w)- wf_wh_z}{h^2}\right]\frac{dz\wedge dw}{w^2}\\
		&=
		 \Big[(f_z/h)|_{w=0} + (f_w/h)_z|_{w=0}\cdot w + O(w^2)\Big] \frac{dz\wedge dw}{w^2}
	\end{split}	
	\end{equation}
	 If we expand $\Upsilon$ with respect to the coordinates $(f,g)$, then
	\[
		\Upsilon= \left(\tilde \upsilon_{-2}\circ f + (\tilde \upsilon_{-1}\circ f)g \right)\wedge \frac{df\wedge dg}{g^2} + \text{holomorphic $2$-form}.
	\]
	Expanding the coefficients gives in a power series in $w$ gives
	\begin{eqnarray*}
		\tilde \upsilon_{-2}\circ f &=& (\tilde \upsilon_{-2}\circ f)|_{w=0} +(\tilde \upsilon_{-2}'\circ f)f_w|_{w=0} \cdot w +O(w^2)\\
		\tilde \upsilon_{-1}\circ f & =& h(\tilde \upsilon_{-1}\circ f)|_{w=0} \cdot w + O(w^2)
	\end{eqnarray*}
	and combining this with \eqref{residue2} we obtain
	the expansion
	\[
	\begin{split}
		\Upsilon &=\left[\frac{f_z(\tilde \upsilon_{-2}\circ f)}{h}\Big|_{w=0}+ \left((\tilde \upsilon_{-1}\circ f)f_z + \Big(\frac{(\tilde \upsilon_{-2}\circ f)f_w}{h}\Big)_{z}\right)\Big|_{w=0}\cdot w + O(w^2)\right]\\
		&\quad \times \frac{dz\wedge dw}{w^2}
	\end{split}.	
	\]
	Comparing this with \eqref{expansion}, and writing $f_0(z)=f(z,0)$, we obtain
	\begin{equation}\label{residueidentiy}
		\begin{split}
		\upsilon_{-2}(z) dz\otimes \partial_w &= (\tilde \upsilon_{-2}\circ f_0) df_0 \otimes (h|_{w=0})^{-1}\partial_w, \\
		\upsilon_{-1}(z) dz&= \tilde \upsilon_{-1}(f_0) df_0 + d\left(h^{-1} (\tilde \upsilon_{-2}\circ f) f_w|_{w=0}\right).		
		\end{split}	
	\end{equation}
	Since $(h|_{w=0})^{-1} \otimes \partial_w\equiv \partial_g$ mod $TQ$, the first equation shows that $\upsilon_{-2}(z)dz\otimes \partial_w$ transforms like a section of $K_Q\otimes NQ$ and thus we obtain a global section $\psi^{-1}$. Now assume that $dz\perp L$ and $df\perp L$ along $Q$. Then we must have $\partial_w\in L$ and hence $f_w|_{w=0} = 0$, such that the last term in \eqref{residueidentiy} disappears. In particular, $\upsilon_{-1}(z)dz$ transform like a section of $K_Q$.
 \end{proof}

\subsection{Adapted surfaces and the moduli space of disks}

\subsubsection{Adapted surfaces} Let us collect some properties pertaining to adapted surfaces $P\subset W$ (see Definition \ref{def_adapted}).


\begin{lemma}\label{h2comp}
	If $P\subset W$ is an adapted surface, then $H_2(W,P)\cong \Z$. It is generated by the relative fibre class $[F]_\mathrm{rel}$.
\end{lemma}

\begin{proof}
We first claim that the inclusion $P\hookrightarrow W$ induces an isomorphism $H_1(P)\xrightarrow{\sim} H_1(W)$. Writing $[P]=m[Q]+n[F]$, we have $m=[P]\cdot [F]=1$ and hence $\rho_*[P]=[M]$, which means that $\rho|_P\colon P \to M$ has degree $1$. By Poincar{\'e} duality and the naturality of the cup product, this implies that $(\rho|_P)_*\colon H_1(P)\to H_1(M)$ respects the intersection pairing and must therefore be an isomorphism. Factoring
	\[
		\begin{tikzcd}
			H_1(P) \arrow{r}\arrow[ swap, "(\rho|_P)_*"]{d} & H_1(W) \arrow["\rho_*"]{dl}\\
			H_1(M) 
		\end{tikzcd},
	\]
	both lower arrows are isomorphisms and thus also the top one is, as desired. As a consequence, the long exact homology sequence for the pair $(W,P)$ contains the following portion:
	\[
		0\to H_2(P)\xrightarrow{i_*} H_2(W)\xrightarrow{j_*} H_2(W,P)\to 0.
	\] 
	Since $[P]=[Q]+n [F]$, the group $H_2(W)$ is also generated by $[P]$ and $[F]$ and thus $H_2(W,P)$ is generated by $j_*[F]=[F]_\mathrm{rel}$. 
\end{proof}

\begin{lemma}\label{lem_intersec}
	Let $P\subset W$ be an adapted surface and $u\colon \DD\to W$   a {holomorphic map} with $u(\partial \DD)\subset P$. Then the following are equivalent:
	\begin{enumerate}
		\item $[u(\DD)] = [F]_\mathrm{rel}$ in $H_2(W,P)$;
		\item there is a unique $\omega_0\in \DD^\circ$ with $u(\omega_0)\in Q$; it holds that  $u'(\omega_0)\notin T_{u(0)}Q$.
	\end{enumerate} 
\end{lemma}

\begin{proof}
Since $P\cap Q=\emptyset$, the set $u^{-1}(Q)$ is compact inside $\DD^\circ$.
Let $L=\O(Q)$ be the line bundle associated to the divisor $Q$, such that there is a holomorphic section $s$ with $\operatorname{div}(s)=Q$. Then $u^{-1}(Q)$ coincides with the vanishing locus of $u^*s$ and thus it must be finite; moreover, $\mathrm{ord}_\omega(u^*s)=1$ if and only if $u'(\omega)\notin T_{u(0)}Q$. Writing $[u(\DD)]=n[F]_\mathrm{rel}\in H_2(W,P)$ we claim that
\[
	\sum_{\omega\in u^{-1}(Q)} \operatorname{ord}_{\omega}(u^*s) = n,
\]
which clearly implies the lemma. To this end, we equip $L$ with a Hermitian structure such that $|s|=1$ near $P$. Let $\eta$ be $(i/2\pi)$ times the curvature of the associated Chern connection, such that $[\eta]=c_1(L)\in H_2(W)$. Near $P$ we must have $\eta = 0$ since $s$ is holomorphic and of unit length and hence
\[
	\int_\DD u^*\eta = n \int_F \eta = n [F]\cdot [Q] = n,
\]
where the second equation uses that $c_1(L)$ equals the Poincar{\'e} dual of $Q$. Next, we pull-back the Hermitian structure to $u^*L$ and compute the curvature term $u^*\eta$ by different means. Let $s_0$ be a holomorphic nowhere vanishing section of $u^*L$ and define $f\colon \DD\to \C$ and $h\colon \DD\to (0,\infty)$ by
\[
	u^*s(\omega) = f(\omega) s_0(\omega)\quad \text{ and } \quad h(\omega)=|s_0(\omega)|^2,\quad \omega\in \DD.
\]
The curvature of $u^*L$ can then be expressed as $\bar \partial \partial \log h$ and thus
\[
	\int_\DD u^*\eta = \frac{i}{2\pi} \int_\DD \bar \partial \partial \log h = \frac{i}{2\pi} \int_{\partial \DD} \partial \log h.
\]
But by our choice of Hermitian structure we have $|f|^2h \equiv  1$  and thus $\partial \log h = - \partial f/f$ in a neighbourhood of $\partial \DD$. Using the argument principle we can therefore continue the preceding display with
\[
	 = -\frac{i}{2\pi}\int_{\partial \DD} \frac{\partial f}{f} = \sum_{f(\omega)=0} \operatorname{ord}_{\omega}(f)=\sum_{\omega\in u^{-1}(Q)} \operatorname{ord}_{\omega}(u^*s) 
\]
and this concludes the proof.
\end{proof}

\subsubsection{Moduli spaces}

 For the remaining subsection we fix an adapted surface $P\subset W$ and consider the following moduli space:
\[
	\mathcal M(P):=\left\{u\colon \DD\to W \text{ holomorphic {immersion}}: \begin{array}{l} [u(\DD)]=[F]_\mathrm{rel}\\[.2em]
	u^{-1}(P)= \partial \DD
	\end{array}\right\}.
\]

\begin{lemma}
The total Maslov index of  any disk $u\in \mathcal M(P)$ equals $4$.

\end{lemma}

\begin{proof}
The total Maslov index any disk with $[u]=[F]_\mathrm{rel}$ is twice the number $\langle c_1(TW,TP),[F]_\mathrm{rel}\rangle=\langle c_1(TW),[F]\rangle = - \deg(K_W|_F)$, where  $K_W\to W$ is the canonical line bundle of $W$. The adjunction formula tells us that $K_F \cong K_W|_F\otimes NF$ and since the normal bundle $NF$ is holomorphically trivial, we must have $\deg(K_W|_F)=-2$, resulting in the total Maslov index $4$.\end{proof}


As a consequence, $\mathcal M(P)$ is an open subset of $\mathcal M_4(P)$, the moduli space considered in Section \ref{prelim_moduli} (see also Lemma \ref{lem_opentouch}). In particular, it is a smooth $6$-dimensional manifold and its tangent space at $u\in \mathcal M(P)$ equals
\[
	T_u\mathcal M(P)= \{\xi\in C^\infty(\DD,u^*TW) \text{ holomorphic}: \xi(\partial \DD)\subset u^*TP \}.
\]
\begin{lemma}\label{lem_ev}The evaluation map  $\mathrm{ev}_0\colon \mathcal M(P) \to W$ is a submersion with
\begin{equation}\label{eqdev}
	(d\mathrm{ev}_0)_u\colon T_u\mathcal M(P)\to T_{u(0)}W,\quad \xi\mapsto \xi(0).
\end{equation}
\end{lemma}

\begin{proof}
To verify \eqref{eqdev}, one considers $u_\lambda = \exp_u(\lambda\xi)$ (where $\exp$ is the exponential map for an auxiliary Riemannian metric  as in Section \ref{prelim_moduli}) for a parameter $\lambda\in \R$. Then $\mathrm{ev}_0(u_\lambda)=\exp_{u(0)}(\lambda \xi(0))$ and taking the derivative at $\lambda=0$ gives $(d\exp_{u(0)})_0(\xi(0))=\xi(0)$, as desired. Next, given a vector $w\in T_{u(0)}W$, we need to find $\xi\in T_u\mathcal M(P)$ with $\xi(0)=w$. For this we use the frame $\{\eta,iu'\}$ from Lemma \ref{lem_normalformframe} and express $w$ as $a_0 \eta(0) + ib_0 u'(0)$ for suitable $a_0,b_0\in \C$. Then $\xi=(a_0+\bar a_0\omega^2)\eta+ i(b_0+\bar b_0\omega^2)u'$ lies in $T_u\mathcal M(P)$ and satisfies  $\xi(0)=w$.
\end{proof}

Hence $\mathcal M(P,Q)=\{u\in \mathcal M(P): u(0)=0\}$ is a $4$-dimensional submanifold whose tangent space at $u\in \mathcal M(P,Q)$ consist of all $\xi\in T_u\mathcal M(P)$ with $\xi(0)\in T_{u(0)}Q$. Let us also define the following subspaces of $T_u\mathcal M(P)$:
\begin{equation*}
	\begin{array}{rclcrcl}
	\mathcal V_u&:=&\ker d(\mathrm{ev}_0)_u, &\quad & \mathcal H_u^+&:=&\displaystyle\frac{(1+ \omega^2)}{\omega}\mathcal V_u,\\[.3em]
	\mathcal H_u^-&:=&\displaystyle i\frac{(1- \omega^2)}{\omega}\mathcal V_u, &\quad & \mathcal H_u &:=& (\mathcal H_u^+\oplus \mathcal H_u^-)\cap T_u\mathcal M(P,Q).
	\end{array}
\end{equation*}
\begin{lemma}\label{lem_horizontalvertical}
The following direct sum decompositions hold true:
\begin{eqnarray}
	T_u\mathcal M(P)=&\mathcal V_u \oplus \mathcal H_u^+ \oplus \mathcal H_u^-,\quad &u\in \mathcal M(P),\\[.5em]
	T_u\mathcal M(P,Q)=&\mathcal V_u \oplus \mathcal H_u ,\quad &u\in \mathcal M(P,Q).
\end{eqnarray}
\end{lemma}
\begin{proof}
Consider the frame $\{\eta,iu'\}$ from Lemma \ref{lem_normalformframe}, then for any $\xi\in T_u\mathcal M(P)$ there are unique coefficients $(a_0,a_1,b_0,b_1)\in \C\times \R\times \C \times \R$ with
\begin{align*}
	\xi(\omega)=& (a_{0} + a_1
	\omega + \bar a_0\omega^2)\eta(\omega) + i(b_0+b_1\omega+\bar b_0\omega^2) u'(\omega)\\[1em]
	=& \Big({a_1\omega \eta(\omega) + i b_1 \omega u'(\omega)}\Big) + \frac{1+\omega^2}{\omega} {\Big((\Re a_0) \omega \eta(\omega) + i (\Re b_0) \omega u'(\omega)\Big)}
	 \\[.5em]
	& + i \frac{1-\omega^2}{\omega}{\Big((\Im a_0) \omega \eta(\omega) + i (\Im b_0) \omega u'(\omega)\Big)}.
\end{align*}
Since all three brackets lie in $\mathcal V_u$, this gives the first decomposition. Since $\mathcal V_u\subset T_u\mathcal M(P,Q)$, the second one is an immediate consequence.
\end{proof}

\subsubsection{The map $\Phi$}
To study the subspaces  $
\C u'(0) \subset T_{u(0)}W$, we define 
\begin{equation}\label{defPhi}
	\Phi\colon \mathcal M(P,Q)\to \C,\quad \Phi(u):=\zeta\quad \text{ when }  u'(0)\in L_{u(0)}^\zeta.
\end{equation}
Here $L^\zeta\subset TW|_Q$ is the holomorphic line bundle from \eqref{defL} and the map $\Phi$ is well-defined by Lemma \ref{lem_intersec}. Moreover, given a holomorphic section $\xi\in C^\infty(\DD,u^*TW)$ (not necessarily in $T_u\mathcal M(P)$), we define complex numbers $\gamma_0(\xi)$ and $\gamma_1(\xi)$ by
\begin{equation}\label{defgamma}
	\langle \omega^2 \Upsilon(u(\omega)),u'(\omega)\wedge \xi(\omega)\rangle = \gamma_0(\xi) + \gamma_1(\xi) \omega + O(\omega^2).
\end{equation}
Note that $\omega^2\Upsilon(u(\omega))$ is holomorphic on  $\DD$, as the double pole is cancelled. 
\begin{lemma}\label{lemdphi} Let $u\in \mathcal M(P,Q)$ and $\xi\in T_u\mathcal M(P,Q)$. Then
\begin{equation}\label{lemdphi0}
	d\Phi_u(\xi)=\big\langle \Res_Q\Upsilon(u(0)),\psi(u'(0))\big\rangle \gamma_0(\xi) - \gamma_1(\xi).
\end{equation}
\end{lemma}
\begin{remark}
The proof implies that $\gamma_1(\xi + i\omega b_1 u'(0))=\gamma_1(\xi) \neq 0$ for all $\xi \in \mathcal V_u$ that are linearly independent from $i\omega u'\in \mathcal V_u$.
\end{remark}
\begin{proof}
Let $(U,z,w)$ be a local coordinate system near $u(0)\in Q$ with $L=\C \partial_w$ along $Q$, such that $z$ is a coordinate on $Q$ and $u(0)$ corresponds to $(z,w)=(0,0)$. In these coordinates $\Upsilon$ has a Laurent expansion
\[
	\Upsilon = \upsilon(z,w) \frac{dz\wedge dw}{w^2},\quad \upsilon(z,w)=\upsilon_{-2}(z) + \upsilon_{-1}(z)w + O(w^2).
\]
{\it Step 1 -- Computation of $d\Phi$:} By Lemma \ref{lem_residue} we can express the line bundle $L^\zeta$ in terms of $\upsilon_{-2}(z)$ as
\[
	L^\zeta_z=\C(\partial_w + \zeta\psi_z([\partial_w])) = \C(\upsilon_{-2}(z)^{-1}  \zeta \partial_z + \partial_w).
\]
For  $\xi\in T_u\mathcal M(P,Q)$ consider the $1$-parameter family of disks $u_\lambda=\exp_u(\lambda\xi)$ ($\lambda\in \R$) with $\exp$ as in Section \ref{prelim_moduli}. Then there are $\zeta_0,\zeta_1\colon \R\to \C$ with
\begin{equation}\label{lemdphi1}
	u_\lambda'(0)=\upsilon_{-2}(u_\lambda(0))^{-1} \zeta_0(\lambda) \partial_z + \zeta_1(\lambda)\partial_w
\end{equation}
and we can express $d\Phi_u$ as
\begin{equation}\label{lemdphi2}
	d\Phi_u(\xi)=(\zeta_0/\zeta_1)'|_{\lambda=0} = \frac{\zeta_0'\zeta_1- \zeta_0\zeta_1'}{\zeta_1^2},
\end{equation}
where we write $\zeta_0=\zeta_0(0)$, etc., when it is understood that we evaluate at $\lambda=0$.
Let us differentiate \eqref{lemdphi1} (with respect  to the Levi--Civita connection $\nabla$ of the Euclidean metric on $U$, that is, regarding $\partial_z$ and $\partial_w$ as parallel):
\begin{equation}\label{lemdphi3}
	\xi'(0)= \partial_\lambda|_{\lambda=0} u_\lambda'(0)=\left(-\upsilon_{-2}(0)^{-2} \langle d\upsilon_{-2}(0),\xi(0)\rangle \zeta_0 + \upsilon_{-2}(0)^{-1}\zeta_0'\right)\partial_z + \zeta_1'\partial_w.
\end{equation}

\noindent {\it Step 2 -- Computation of $\gamma_0,\gamma_1$:}
Let $\{\eta,iu'\}$ be the frame from Lemma \ref{lem_normalformframe}. Then there exists a unique $c\in \C$ with $\eta(0)+icu'(0)\in T_{u(0)}Q$ and we define $\vartheta(\omega):=\eta(\omega)+ic u'(\omega)$. It will be convenient use $\{u',\vartheta\}$ as a holomorphic frame for $u^*TW$. We can write
\[
u'(\omega)=\upsilon^{-1}(u(\omega)) f_0(\omega) \partial_z + f_1(\omega) \partial_w,\quad \vartheta(\omega) =\upsilon^{-1}(u(\omega)) g_0(\omega) \partial_z + g_1(\omega) \partial_w 
\]
for holomorphic functions $f_0,f_1,g_0,g_1\colon \DD\to \C$, which satisfy $(f_0(0),f_1(0))=(\zeta_0,\zeta_1)$ and $g_1(0)=0$.  We claim that
\begin{equation}\label{lemdphi31}
	\gamma_0(\vartheta) = -g_0(0)/\zeta_1,\quad \gamma_1(\vartheta) = -(g_0'(0) - \zeta g'_1(0))/\zeta_1.
\end{equation}
Since $\partial_\omega w(u(\omega)) = \langle dw, u'(\omega)\rangle=f_1(\omega)$, we have
\[\frac{w(u(\omega))}{\omega} = \zeta_1 + \frac 12 f_1'(0) \omega + O(\omega^2),\quad \left(\frac{w(u(\omega))}{\omega}\right)^{-2}
=\frac{1}{\zeta_1^2}  - \frac{f_1'(0)}{\zeta_1^3} \omega + O(\omega^2)
\]
and hence
\begin{align*} &	\langle \omega^2 \Upsilon(u(\omega)),u'(\omega)\wedge \vartheta(\omega)\rangle\\[.3em]
=&\left(\frac{w(u(\omega))}{\omega}\right)^{-2} \big\langle \upsilon(u(\omega)){dz\wedge dw}, u'(\omega)\wedge \eta(\omega)\big\rangle \\[.3em]
	=&\left(\frac{1}{\zeta_1^2}  - \frac{f_1'(0)}{\zeta_1^3}\omega\right) \Big(f_0(\omega) g_1(\omega) - f_1(\omega) g_0(\omega)\Big) + O(\omega^2) \\[.3em]
	=&\left(\frac{1}{\zeta_1^2}  - \frac{f_1'(0)}{\zeta_1^3}\omega\right) \Big(-\zeta_1 g_0(0) + \big( \zeta_0g_1'(0) - f_1'(0)g_0(0)- \zeta_1 g_0'(0)\big)\omega\Big) + O(\omega^2)\\[.3em]
	=&-\frac{g_0(0)}{\zeta_1}  -\left(\frac{g_0'(0)-\zeta g_1'(0)}{\zeta_1}\right) \omega + O(\omega^2),
\end{align*}
as desired.

\noindent {\it Step 3 -- The case $\xi\in \mathcal V_u$:} We now proceed by considering the cases $\xi\in \mathcal V_u$ and $\xi\in \mathcal H_u$ separately, which suffices since both sides of \eqref{lemdphi0} are $\R$-linear in $\xi$.  If $\xi\in \mathcal V_u$, then in the frame $\{\eta,iu'\}$ we have $\xi(\omega)=a_1\omega \eta(\omega)+ib_1\omega u'(\omega)$ with $a_1,b_1\in \R$. Plugging this into \eqref{lemdphi3} gives
\[
	a_1 \eta(0) + ib_1u'(0)= \upsilon_{-2}(0)^{-1}\zeta_0'\partial_z + \zeta_1'\partial_w.
\]
The left hand side can also be expressed as 
\begin{eqnarray*}
	a_1 \eta(0) + ib_1u'(0) &=& a_1 \vartheta(0) + i(b_1-a_1c)u'(0) \\
	&=& \upsilon_{-2}(0)^{-1}\big(a_1 g_0(0)+i(b_1-a_1c) \zeta_0\big)\partial_z + i(b_1-a_1c) \zeta_1 \partial_w
\end{eqnarray*}
and by comparing coefficients we get
\(
	\zeta_0' = a_1 g_0(0) + i(b_1-a_1c)\zeta_0\) and \(\zeta_1'=i(b_1-a_1c) \zeta_1
\),
such that
\[
	d\Phi_u(\xi)=\frac{\zeta_0'\zeta_1-\zeta_0\zeta_1'}{\zeta_1^2} =  \frac{a_1g_0(0)}{\zeta_1}.
\]
Moreover, since
\(
	u'(\omega)\wedge \xi(\omega)= a_1 \omega (u'(\omega)\wedge \vartheta(\omega))
\) we have
\[
	\gamma_0(\xi)=0\quad \text{ and }\quad \gamma_1(\xi)=a_1 \gamma_0(\vartheta) = -\frac{a_1g_0(0)}{\zeta_1},
\]
which established formula \eqref{lemdphi0}.

\noindent {\it Step 4 -- The case $\xi\in \mathcal H_u$:}
We start by computing the derivatives of $\upsilon$ that show up now, writing $\upsilon_{-2}'(z) = \partial_z\upsilon_{-2}(z)$:
\begin{eqnarray}
	\langle d\upsilon,u'(0)\rangle&=& \upsilon_{-2}(0)^{-1} \upsilon_{-2}'(0) \zeta_0 + \upsilon_{-1}(0) \zeta_1,\label{lemdphi4}\\
	\langle d\upsilon,\vartheta(0)\rangle&=& \upsilon_{-2}(0)^{-1} \upsilon_{-2}'(0) g_0(0). \label{lemdphi5}
\end{eqnarray}
Given $\xi\in \mathcal H_u$, we have
 $\xi(\omega)=(a_0 + \bar a_0\omega^2) \eta(\omega) + i(b_0+\bar b_0\omega^2)u'(\omega)$ for coefficients $(a_0,b_0)\in \C$. Since $\xi(0)\in T_{u(0)}Q$ we must have $b_0=ca_0$ and hence
 \begin{align*}
 	\xi(0)&=a_0\vartheta(0)=\upsilon_{-2}(0)^{-1} a_0g_0(0) \partial_z
 \end{align*}
 as well as
 \begin{equation*}
 \begin{array}{ll}
 &\xi'(0)=a_0\vartheta'(0) \\[.3em]
 &=\big(- \upsilon_{-2}(0)^{-2} \langle d\upsilon,u'(0)\rangle a_0g_0(0) + \upsilon_{-2}(0)^{-1} a_0g_0'(0) \big)
 \partial_z+  a_0 g_1'(0) \partial_w.
 \end{array} 
\end{equation*}
Inserting the preceding two displays into  \eqref{lemdphi3} and comparing the coefficients in front of $\partial_z$ and $\partial_w$ yields two equations. After evaluating the derivatives of $\upsilon$ via \eqref{lemdphi4} and \eqref{lemdphi5}, these can be arranged into the following form:
\begin{eqnarray*}
\zeta_0' &=& - \upsilon_{-2}(0)^{-1}\upsilon_{-1}(0)\zeta_1 a_0g_0(0) + a_0 g_0'(0),\\
\zeta_1' & =& a_0g_1'(0).
\end{eqnarray*}
Hence
\begin{equation}\label{lemdphi6}
	d\Phi_u(\xi)=\frac{\zeta_0'\zeta_1-\zeta_0\zeta_1'}{\zeta_1^2}= - \upsilon_{-2}(0)^{-1} \upsilon_{-1}(0) a_0g_0(0) + \frac{a_0}{\zeta_1}  \big(g_0'(0)-\zeta g_1'(0)\big).
\end{equation}
Since $\xi(\omega)=a_0\vartheta(\omega)+O(\omega^2)$, we can use \eqref{lemdphi31} to compute
\[
	\gamma_0(\xi)=a_0\gamma_0(\vartheta)= - \frac{a_0}{\zeta_1}g_0(0),\quad \gamma_1(\xi)=a_0\gamma_1(\vartheta)=-\frac{ a_0}{\zeta_1}(g_0'(0)-\zeta g_1'(0)).
\]
Inserting this into \eqref{lemdphi6} gives
\(
	d\Phi_u(\xi)= \upsilon_{-1}(0)\left(\upsilon_{-2}(0)^{-1}\zeta_1\right)\gamma_0(\xi) - \gamma_1(\xi),
\)
and this concludes the proof.
\end{proof}

\subsection{Disk families}

\subsubsection{Horizontal section}
Let us relate the just constructed moduli space to the set of disk families $\mathcal F(P)$ (see Definition \ref{def_diskfamily}). To this end, we first take the quotient 
\[
	\mathcal N(P):=\mathcal M(P,Q)/S^1
\]
by the free $S^1$-action $u_\lambda(\omega)=u(e^{i\lambda} \omega)$ ($\lambda \in \R$), which yields  a smooth $3$-dimensional moduli space of {\it unparametrised} holomorphic disks. The maps $\mathrm{ev}_0$ and $\Phi$ from \eqref{defPhi} descend to the quotient, where we write  $\pi_Q([u])=u(0)$ and $\varphi([u])=\Phi(u)$. This results in the following picture:
\begin{equation}\label{parallelsec}
	\begin{tikzcd}
		\mathcal N(P) \arrow["\varphi"]{r} \arrow["\pi_Q"]{d} &\C\\
			Q \arrow[bend left=30, dashed,"\delta"]{u}
	\end{tikzcd} 
\end{equation}

\begin{lemma}\label{lem_correspondencedisk} There is a $1:1$-correspondence between:
\begin{enumerate}
	\item Disk families $(\Delta_q:q\in Q)\in \mathcal F(P)$
	\item Continuous sections $\delta\colon Q\to \mathcal N(P)$ of $\pi_Q$ with $\varphi\circ \delta\equiv \text{const.}$.
\end{enumerate}	
\end{lemma}

\begin{proof} Given a continuous section $\delta\colon Q\to \mathcal N(P)$, we set $\Delta_q=u(\DD)$, where $u$ is a representative of $\delta(q)\in \mathcal N(P)$.
In view of Lemma \ref{lem_intersec} this gives a $1:1$ correspondence between continuous section of $\pi_Q$ and continuous disk families with property (a) of Definition \ref{def_diskfamily}. It then holds that $T_q \Delta_q = L^\zeta_q$ iff $\varphi\circ \delta(q)=\zeta$ and thus property (b) corresponds to holomorphicity of   $\varphi\circ \delta\colon Q\to \C$, which is of course equivalent to it being constant.
\end{proof}

The generator of the $S^1$-action on $\mathcal M(P,Q)$ equals $i\omega u'\in \mathcal V_u$ and thus the splitting from Lemma \ref{lem_horizontalvertical} descends as a splitting
\[
	T_{[u]}\mathcal N(P) = \mathcal V_{[u]} \oplus \mathcal H_{[u]},\quad [u]\in \mathcal N(P),
\]
where $\mathcal V_{[u]} = \ker d(\pi_Q)_{[u]}$ is now $1$-dimensional. We will consider adapted surfaces with the following additional property:
\begin{eqnarray}\label{almostlagr}
\forall u\in \mathcal M(P,Q), \xi \in \mathcal H_u: \gamma_1(\xi)=0
\end{eqnarray}
	Here $\gamma_1$ is defined in \eqref{defgamma}. Recall from Lemma \ref{lemimups} that $W\backslash Q$ is a symplectic manifold with respect to $\Im \Upsilon$.
	
	\begin{lemma}
		Let $P\subset W$ be an adapted surface that is Lagrangian for $\Im \Upsilon$. Then property \eqref{almostlagr} holds true. 
	\end{lemma}
	
	\begin{proof}
	 If $\xi\in \mathcal H_u$, then as in the proof of Lemma \ref{lemdphi}, we have $\xi(\omega)=a_0 \vartheta(\omega)+ O(\omega^2)$ for some $a_0\in \C$, where $\vartheta(\omega)=\eta(\omega)+icu'(\omega)$ for a suitable $c\in \C$. Since $\gamma_1(\xi)=a_0\gamma_1(\vartheta)$, it suffices to show that $\gamma_1(\vartheta)=0$. For $\omega\in \partial \DD$, the tangent space $T_{u(\omega)}P$ is spanned by $\{i\omega u'(\omega),\omega \vartheta(\omega)\}$ and $P$ being Lagrangian for $\Im \Upsilon$ implies that
\[
	\Im \langle \omega^2\Upsilon(u(\omega)), iu'(\omega)\wedge \vartheta(\omega) \rangle = 0,\quad \omega\in \partial \DD.
\]
That is, the holomorphic function $\omega \mapsto \langle \omega^2 \Upsilon(u(\omega)),u'(\omega)\wedge \vartheta(\omega)\rangle$ has imaginary boundary values and is therefore constant. In particular,  $\gamma_1(\vartheta)=0$.	
	\end{proof}

\begin{lemma}\label{lem_horizontal}
 Let $P\subset W$ be an adapted surface with property \eqref{almostlagr}. Assume further that $\Res_Q\Upsilon=0$. Then 
 \[
 	\ker (d\varphi)_{[u]}=\mathcal H_{[u]},\quad [u]\in \mathcal N(P).
 \]
 In particular, $d\varphi$ has constant rank $1$ and the $2$-plane distribution $\mathcal H\subset T\mathcal N(P)$ is the tangent bundle of a $2$-dimensional foliation of $\mathcal N(P)$.
\end{lemma}

\begin{proof} The inclusion `$\supset$' follows directly from Lemma \ref{lemdphi} and property \eqref{almostlagr}. Applying the lemma to $\mathcal V_u$ shows that $d\varphi_{[u]}\neq 0$ and hence its kernel has dimension $2$, just like $\mathcal H_{[u]}$. Hence we have equality and $d\varphi$ has rank $1$. By the constant rank theorem, the pre-images $\varphi^{-1}(\{c\})\subset \mathcal N(P)$ ($c\in \C$) are the leaves of a $2$-dimensional foliation, with tangent space at $[u]$ given by $\ker (d\varphi)_{[u]}=\mathcal H_{[u]}$.
\end{proof}

\subsubsection{The blow-up procedure}

\begin{proposition}\label{prop_defbeta} Let $P\subset W$ be an adapted surface with property \eqref{almostlagr} and suppose that $\Res_Q(\Upsilon)=0$. Given a disk family
 $(\Delta_q:q\in Q)\in \mathcal F(P)$, there exists a fibrewise holomorphic blow-down map
\[\beta\colon DM\to W 
\]
with $\beta(SM)=P$, $\beta(x,0)=q(x)$ and $\beta(D_xM)=\Delta_{q(x)}$ for all $x\in M$. 
\end{proposition}

\begin{remark}\label{rk_immdisk}
As mentioned in the introduction, also the other direction holds true: Given a fibrewise holomorphic blow-down map $\beta\colon DM\to W$ and $\beta(x,0)=q(x)$, we obtain a disk family with $\Delta_q=\beta(D_x M)$. { The map $\beta(x,\cdot)\colon D_xM\to \Delta_q\subset  W$ is indeed an immersion: in the interior this follows from $\beta$ being a diffeomorphism, and at $v\in S_xM$ we use the fibrewise holomorphicity and  the second property of Lemma \ref{lem_lift} to see that
\[
	d\beta_{(x,v)}(V)=i d\beta_{(x,v)}(V_\perp)  \neq 0.
\]
The map $\beta\colon DM^\circ \to W\backslash P$ being a diffeomorphism also implies that each disk $\Delta_q$ has a unique transversal intersection with $Q$ and indeed we obtain a continuous disk family $(\Delta_q:q\in Q)$ satisfying property (a) of Definition \ref{def_diskfamily}.
Property (b) is then satisfied iff the complex lines $d\beta(T_0D_xM)\subset T_{\beta(x,0)}W$ form a holomorphic bundle.}
\end{remark}	

\begin{proof}[Proof of Proposition \ref{prop_defbeta}]
\noindent {\it Step 1 -- Construction of $\beta$:} 
There is an isomorphism $TM\cong TQ \cong NQ \cong L$ of holomorphic line bundles, covering $q$. Given $(x,v)\in SM$, this selects a basis vector $\vartheta(x,v)\in L_{q(x)}$, resulting in a canonical holomorphic parametrisation $u_{(x,v)}\colon \DD\to \Delta_{q(x)}$ with $u_{(x,v)}(0)=q(x)$ and $u_{(x,v)}'(0)\in \vartheta(x,v)\R_{>0}$. We define
\[
	\bar \beta\colon SM\times \DD\to W,\quad \bar \beta(x,v,\omega)=u_{(x,v)}(\omega).
\]
Since $\vartheta(x,e^{it}v) = e^{it} \vartheta(x,v)$, it also holds that $u_{(x,v)}(e^{it} \omega) = u_{(x,e^{it} v)}(\omega)$ and $\bar \beta(x,e^{it} v,e^{-it} \omega)= \bar \beta(x,v,\omega)$. Hence $\bar \beta = \beta\circ \p$ for a  smooth map $\beta\colon DM\to W$. The construction of $\beta$ makes property (i) obvious.

Locally we can describe $\beta$ as follows: In view of Lemma \ref{lem_correspondencedisk}  the disk family can be described by a continuous section $\delta\colon Q\to \mathcal N(P)$ and since it maps into a level set of the smooth rank $1$-map $\varphi\colon \mathcal N(P)\to \C$, it has to be smooth. Given $x_0\in M$ and a local trivialising function $\mathbf 1\colon  U\to SM$ this lifts to a map $\hat \delta\colon q(U)\to \mathcal M(P,Q)$ with $\hat \delta(q(x))=u_{x,\mathbf 1(x)}$. Then $\Sigma = \hat \delta(U)\subset \mathcal M(P,Q)$ is an embedded surface and
\[
	\tau\colon DM|_U\xrightarrow{\sim} \Sigma\times \DD,\quad \tau(\omega\mathbf 1(x))\mapsto (\hat \delta(q(x)),\omega)
\]
locally trivialises the disk bundle. On $DM|_U$ we have $\beta = \beta_\Sigma\circ \tau$, where $\beta_\Sigma$ is defined as follows:
\[\beta_\Sigma\colon \Sigma\times \DD\to W, \quad  \beta_\Sigma(u,\omega)=u(\omega).
\]

\noindent {\it Step 2 -- Lift through blow-down:}  We claim that $\beta$ lifts to a diffeomorphism $\hat \beta \colon Z\xrightarrow{\sim} [W,P]$ through the standard blow-down $[W,P]\to W$. To obtain a lift as local diffeomorphism, it suffices to consider $\beta_\Sigma\colon \Sigma\times \DD\to [W,P]$ and check the four properties of Lemma \ref{lem_lift2}. Property (i) is satisfied since $u^{-1}(P)= \partial \DD$ for all $u\in \Sigma$. For the remaining properties we use that $\Sigma$ is horizontal (i.e.~$T_u\Sigma= \mathcal H_u$ for all $u\in \Sigma$), which follows from Lemma \ref{lem_horizontal}. We compute the differential at a point $(u,\omega)\in \Sigma\times \DD$:
\[
	(d\beta_\Sigma)_{(u,\omega)} \colon T_{(u,\omega)}(\Sigma\times \DD)=\mathcal H_u\times \C \to T_{u(\omega)}W,\quad (\xi,\dot \omega)\mapsto  \xi(\omega) + \dot \omega u'(\omega).
\]
Let $\{\eta,iu'\}$ be the frame from Lemma \ref{lem_normalformframe} and $c\in \C$  such that $\eta(0)+icu'(0) \in T_{u(0)}Q$. Then any $\xi\in \mathcal H_u$ takes the form $\xi(\omega)=(a_0+\bar a_0 \omega^2) \eta(\omega) + i(ca_0 +  \bar{ca_0} \omega^2)u'(\omega)$ for some $a_0\in \C$ (cf.\,proof of Lemma \ref{lem_horizontalvertical}). In particular, \[(\xi,\dot \omega)\in \ker (d\beta_\Sigma)_{(u,
\omega)} \quad \Leftrightarrow \quad a_0+\bar a_0\omega^2 =0 \quad \text{ and } \quad \dot \omega = 0.
\]
If $|\omega|<1$, the right hand side enforces that $(\xi,\dot \omega) = 0$ and hence $\beta_\Sigma\colon \Sigma\times \DD^\circ\to  W\backslash P$ is a local diffeomorphism, which is property (ii) of Lemma \ref{lem_lift2}.  For $|\omega|=1$ the kernel of $(d\beta_\Sigma)_{(u,\omega)}$ is given by the equations $\Re(\bar a_0\omega)= \dot \omega = 0$; hence it has dimension $1$ and is tangential to the boundary of $\Sigma\times \DD$, as required by property (iii) of Lemma \ref{lem_lift2}. For the last property we may fix $\omega_0\in \partial \DD$ and consider the $1$-parameter family $(u_\lambda, \omega_0 , 0, \omega_0 \cdot \partial_s) \in T(\Sigma\times \DD)$ in $|\lambda|<\epsilon$ with $u_\lambda(\omega_0)\equiv p\in P$. Using the notation from Lemma \ref{lem_lift2} (iv), we consider 
\[
	w_\lambda  = d\beta_{(u_\lambda,\omega_0)}(0,\omega_0\cdot \partial_s) = \omega_0 u_\lambda'(\omega_0) \in T_pW. 
\]
Suppose $\xi = (\partial_\lambda|_{\lambda=0}) u_\lambda \in \mathcal H_u$ is given by $a_0\in \C$ with $\Re(\bar a_0\omega_0) = 0$. Then
\[
	w_0 = \omega_0 u'(\omega_0),\quad \dot w_0 = \omega_0 \xi'(\omega_0) = 2(\bar a_0\omega_0) \omega_0\eta(\omega_0) + 2ic(\bar a_0\omega_0)  \omega_0 u'(\omega_0)
\]
and, assuming $a_0\neq 0$, these are complemented to a basis of $T_pW$ by the vectors $\{\omega_0\eta(\omega_0),i\omega_0 u'(\omega_0)\}\subset T_pP$. Hence $\beta_\Sigma$ meets all conditions of Lemma \ref{lem_lift2}.

It remains to show that $\hat \beta\colon Z\to [W,P]$ is injective and here we argue as in Remark \ref{rmk_covering}. Let $x_0\in M$ and suppose that $\beta(x,v)=\beta(x_0,0)\in Q$ for some $(x,v)\in DM$. Then the disk $\Delta_{q(x)}$ intersects $Q$ in $q(x)$ and $q(x_0)$ and since the intersection point is unique and $q\colon M\to Q$ is injective, we conclude that $x=x_0$ and $v=0$. This concludes the proof that $\beta$ is a blow-down map.
\end{proof}

\subsection{Proof of Theorem \ref{introthmD}}
We consider the following two sets:
\begin{eqnarray*}
	\mathbf{P}&=&\{P\subset W: \text{adapted surface that is Lagrangian for} \Im \Upsilon\}\\
	\mathbf{P_*}&=& \{P\in \mathbf{P}: \mathcal F(P)\neq \emptyset\}
\end{eqnarray*}
Note that this description of $\mathbf{P_*}$ agrees with the one from Theorem \ref{introthmD} due to Proposition \ref{prop_defbeta}. Our goal is to show that $\mathbf{P}_*\subset \mathbf{P}$ is an open subset, and in view of Lemma \ref{lem_correspondencedisk} this requires to establish the persistence of parallel sections (i.e.~$\delta\colon Q\to \mathcal N(P)$ with $\pi_Q\circ \delta \equiv \Id$ and $\varphi\circ \delta \equiv \mathrm{const.}$) under small perturbations of $P$.

This persistence phenomenon is based on the following observation:

\begin{lemma}\label{lem_stability} Let $\mathcal N$ be a $3$-manifold  and suppose there are a smooth maps
\[
	\begin{tikzcd}
		\mathcal N\arrow{r}{f_0}\arrow{d}{\pi} & \R^2\\
		Q
	\end{tikzcd},
\]
such that $df_0$ has constant rank $1$ and $\pi$ is a submersion. Assume further that there is a smooth section $g_0\colon Q\to \mathcal N$ of $\pi$ with $f_0\circ g_0 \equiv \text{const.}$. Then if $f\in C^\infty(\mathcal N,\R^2)$ is sufficiently close to $f_0$ in the $C^1$-topology and $df$ also has constant rank $1$, there exists a smooth section $g\colon Q\to \mathcal N$ with $f\circ g\equiv \mathrm{const.}$.
\end{lemma}

\begin{proof}
	Without loss of generality, we may assume that $\mathcal N=Q\times \R$, $f_0(x,h)=(0,h)\in \R^2$ and $g_0(x)=(x,0)$. Take $x_0\in Q$ and suppose that $f$ satisfies
	\begin{equation}
	\begin{split}
		|f_2(x_0,0)|<1/2,\quad \pm f_2(x,h)>1/2 \quad &\text{for } x\in Q, 1\le \pm h\le 2,\\
		 d_2f_2(x,h)>0 \quad &\text{for } x\in Q,|h|\le 1.
	\end{split}	\label{convexcond}
	\end{equation}
	These conditions define a convex $C^1$-open neighbourhood of $f_0$ in $C^\infty(\mathcal N,\R^2)$. Let $c=f(x_0,0)\in \R^2$, then $|c_2|<1/2$ and hence
	\[
		f_2^{-1}(\{c_2\}) \subset \{|h|<1\}\cup \{|h|>2\}.
	\]
	Let $\Sigma\subset f_2^{-1}(\{c_2\})$ be the connected component containing $(x_0,0)$, then $\Sigma$ is closed inside $Q\times \R$ and satisfies $\Sigma\subset \{|h|<1\}$, hence  it is compact. Further, $f_2$ is a submersion on $\{|h|<1\}$ and therefore $\Sigma$ must be an embedded $2$-dimensional submanifold. At $(x,h)\in \Sigma$ we have 
	\[
		T_{(x,h)}\Sigma \cap \ker d\pi{(x,h)} = \ker df_2(x,h)\cap \ker d\pi(x,h)= \ker d_2f_2(x,h)=0
	\] 
	and thus $\pi|_\Sigma\colon \Sigma\to Q$ is a local diffeomorphism. Since both surfaces are connected and compact, it must be a covering map. Let us assume for the moment that $Q$ is simply connected, such that there exists a global section  $g\colon Q\to \Sigma\subset Q\times \R$. 
	At $(x,h)\in \Sigma$ it holds that
	\[
		\ker df(x,h) \subset \ker df_2(x,h).
	\]	
	and since $df(x,h)$ has rank $1$we must have `$=$' for dimension reasons. Writing $h=g(x)$, this implies that
	\[
	\operatorname{rg} dg(x)=T_{(x,h)}\Sigma = \ker df_2(x,h) = \ker df(x,h),
	\]
	and thus $f\circ g\equiv \mathrm{const.}$, as desired.
	If $Q$ is not simply connected an additional step is required to ensure that the covering $\pi|_\Sigma\colon \Sigma\to Q$ is $1$-sheeted. Since $f_t=(1-t) f_0 + t f\in C^\infty(\mathcal N,\R^2)$ satisfies the conditions \eqref{convexcond} for all $0\le t\le 1$, the preceding constructions can be carried out continuously in $t$ and we obtain a $1$-parameter family of coverings $\Sigma_t\to Q$. Since $\Sigma_0=Q\times \{0\}\to Q$ is $1$-sheeted, the same has to be true for $\Sigma=\Sigma_1$.
\end{proof}

To apply the lemma, we  translate perturbations of $P$ into ones for $\varphi$:

\begin{lemma}
	Let $P$ be an adapted surface and $K\subset \mathcal N(P)$ compact. Then for any sufficiently close $P'$ in the $C^\infty$-topology there exists a neighbourhood $U\subset \mathcal N(P)$ of $K$ and a $C^1$-embedding $\Psi\colon U\to \mathcal N(P')$ such that $\pi_Q\circ \Psi=\pi_Q$.
\end{lemma}

\begin{proof}
Let $K'\subset \mathcal M(P,Q)$ be the preimage of $K$ under the quotient map $\mathcal M(P,Q)\to \mathcal N(P)$. Then by Theorem \ref{thm_modulimaster}\ref{thm_modulimasterii} there exists a neighbourhood $U'\subset \mathcal M(P)$ of $K'$ and a  $C^1$-embedding $\Psi'\colon U' \to \mathcal M(P')$ with $\mathrm{ev}_0\circ \Psi' = \mathrm{ev}_0$. In particular, $\Psi'(U'\cap \mathcal M(P,Q)) \subset \mathcal M(P,Q)$. Moreover, $\Psi'$ is equivariant with respect to rotations and hence it descends to the desired embedding.
\end{proof}

This allows us to complete:

\begin{proof}[Proof of Theorem \ref{introthmD}]
Let $P\in \mathbf{P_*}$, such that there is a section $\delta\colon Q\to \mathcal N(P)$ with $\varphi_P\circ \delta \equiv \mathrm{const.}$ (we use the notation $\varphi_P$ to denote the map in \eqref{parallelsec}).  We choose $P'\in \mathbf{P}$ sufficiently close to $P$, such that the preceding lemma gives an embedding $\Psi\colon U\to \mathcal N(P')$ of a neighbourhood of $K=\delta(Q)$. We can then also assume that $\varphi_{P'}\circ \Psi\colon U\to \C\equiv \R^2$ is $C^1$-close to $\varphi_P$. By Lemma \ref{lem_horizontal} it has constant rank $1$, and applying Lemma \ref{lem_stability} on $\mathcal N=U$ we obtain a section $\delta'\colon Q\to \mathcal N(P)$ with $\varphi_{P'}\circ \Psi\circ \delta' \equiv \mathrm{const.}$. But then $(\Psi\circ \delta')\colon Q\to \mathcal N(P')$ is a parallel section for $P'$, which proves that $P'\in \mathbf{P}_*$, as desired.
\end{proof}

\begin{remark}
The same proof shows that $\tilde{\mathbf{P}}_* \subset \tilde{\mathbf{P}}$ is open, where $\tilde{\mathbf{P}}$ is the set of adapted surfaces with property \eqref{almostlagr} and $\tilde{\mathbf{P}}_* = \{P\in \tilde{\mathbf{P}}: \mathcal F(P)\neq \emptyset\}$.
\end{remark}

\bibliographystyle{plain}
\bibliography{zollpaper.bbl}

\end{document}